\documentclass[tbtags,reqno]{amsart}
\usepackage{epsfig,amsthm,amsopn,amsfonts,amssymb,tabmac,pb-diagram}
\usepackage[hypertex]{hyperref}
\usepackage{verbatim}
\usepackage{xcolor}
\usepackage[active]{srcltx}

\newtheorem{theorem}{Theorem}
\newtheorem{lemma}[theorem]{Lemma}
\newtheorem{lem}[theorem]{Lemma}
\newtheorem{proposition}[theorem]{Proposition}
\newtheorem{prop}[theorem]{Proposition}
\newtheorem{property}[theorem]{Property}

\newtheorem{conjecture}[theorem]{Conjecture}

\newtheorem{corollary}[theorem]{Corollary}
\newtheorem{cor}[theorem]{Corollary}

\theoremstyle{definition}
\newtheorem{definition}[theorem]{Definition}
\newtheorem{notation}[theorem]{Notation}

\theoremstyle{remark}
\newtheorem{example}[theorem]{Example}
\newtheorem{remark}[theorem]{Remark}
\newtheorem*{acknow}{Acknowledgments}

\def \la {\lambda}

\def\la {\lambda}

\def\charge{ {\rm {charge}}}

\def\shape{ {\rm {shape}}}

\newcommand{\Bdd}{\mathcal{B}}
\newcommand{\bdy}{\partial}

\newcommand{\Bot}{\mathrm{bot}}

\newcommand{\bp}{\mathbf{p}}
\newcommand{\bq}{\mathbf{q}}
\newcommand{\bqt}{\mathbf{\tilde{q}}}
\newcommand{\br}{\mathbf{r}}
\newcommand{\brt}{\mathbf{\tilde{r}}}
\newcommand{\bt}{\mathbf{t}}

\newcommand{\change}{\Delta}
\newcommand{\changerev}{\Delta}
\newcommand{\changecol}{\Delta_{\rm cs}}
\newcommand{\changerow}{\Delta_{\rm rs}}

\newcommand{\col}{\mathrm{col}}
\newcommand{\comp}{\mathrm{comp}}
\newcommand{\Core}{\mathcal{C}}
\newcommand{\cs}{\mathrm{cs}}
\newcommand{\diag}{d}
\newcommand{\dk}{\widetilde{\mathfrak{S}}}
\newcommand{\dkr}{\mathfrak{S}}
\newcommand{\dom}{\trianglerighteq}
\newcommand{\ecs}{\mathrm{ecs}}
\newcommand{\ers}{\mathrm{ers}}

\newcommand{\Gr}{\mathrm{Gr}}

\newcommand{\hs}{\hat{s}}
\newcommand{\ip}[2]{\langle #1\,,\,#2\rangle}
\newcommand{\Ind}{\mathrm{Ind}}

\newcommand{\Int}{\mathrm{Int}}
\newcommand{\Ksh}{\Pi}
\newcommand{\ksf}{s}
\newcommand{\ksgen}{{\mathfrak s}}
\newcommand{\La}{\Lambda}
\newcommand{\Left}{\mathrm{left}}

\newcommand{\Part}{\mathbb{Y}}
\newcommand{\Path}{\mathcal{P}}
\newcommand{\Patheq}{\overline{\Path}}

\newcommand{\per}{\mathrm{per}}
\newcommand{\pred}{\mathrm{pred}}
\newcommand{\Prim}{\mathrm{Prim}}
\newcommand{\pull}{\mathrm{pull}}
\newcommand{\push}{\mathrm{push}}
\newcommand{\Right}{\mathrm{right}}
\newcommand{\rk}{\mathrm{rank}}

\newcommand{\row}{\mathrm{row}}
\newcommand{\rs}{\mathrm{rs}}
\newcommand{\RT}{\mathrm{Tab}}

\newcommand{\ts}{\tilde{s}}
\newcommand{\SL}{\mathrm{SL}}
\newcommand{\Strong}{\mathrm{STab}}
\newcommand{\Str}{\mathrm{Strip}}
\newcommand{\sucs}{\mathrm{succ}}
\newcommand{\T}{\widetilde{\mathrm{Tab}}}

\newcommand{\tM}{\tilde{M}}
\newcommand{\tm}{\tilde{m}}
\newcommand{\tn}{\tilde{n}}
\newcommand{\tN}{\tilde{N}}
\newcommand{\tS}{\tilde{S}}
\newcommand{\tx}{\tilde{x}}

\newcommand{\Top}{\mathrm{top}}

\newcommand{\Weak}{\mathrm{WTab}}
\newcommand{\WS}{\mathrm{Weak}}
\newcommand{\wt}{\mathrm{wt}}
\newcommand{\Z}{\mathbb{Z}}

\newcommand{\defit}[1]{\textit{#1}}

\begin{document}

\title[$k$-shape poset]{$k$-shape poset and branching of $k$-Schur functions}

\author{Thomas Lam} \address
 {Department of Mathematics, University of Michigan, Ann Arbor MI 48109 USA.}
\email{tfylam@umich.edu}
\author{Luc Lapointe}
\address{Instituto de Matem\'{a}tica Y F\'{i}sica, Universidad de Talca, Casilla 747, Talca, Chile.}
\email{lapointe@inst-mat.utalca.cl}
\author{Jennifer Morse}
\address{Department of Mathematics, Drexel University, Philadelphia, PA 19104 USA.}
\email{morsej@math.drexel.edu}
\author{Mark Shimozono}
\address{Department of Mathematics, Virginia Tech,
         Blacksburg, VA 24061 USA.}
\email{mshimo@vt.edu}


\begin{abstract}
We give a combinatorial expansion of a Schubert homology class
in the affine Grassmannian $\Gr_{\SL_k}$ into Schubert homology
classes in $\Gr_{\SL_{k+1}}$. This is achieved by studying the
combinatorics of a new class of partitions called {\it $k$-shapes},
which interpolates between $k$-cores and $k+1$-cores.  We define a
symmetric function for each $k$-shape, and show that they expand
positively in terms of dual $k$-Schur functions.  We obtain an explicit
combinatorial description of the expansion of an ungraded $k$-Schur
function into $k+1$-Schur functions.  As a corollary, we give a
formula for the Schur expansion of an ungraded $k$-Schur function.
\end{abstract}

\maketitle

\section{Introduction}

\subsection{$k$-Schur functions and branching coefficients}

The theory of $k$-Schur functions arose from the study of Macdonald
polynomials and has since been connected to quantum and affine Schubert
calculus, $K$-theory, and representation theory.
The origin of the $k$-Schur functions  is related to Macdonald's positivity
conjecture, which asserted that in the expansion
\begin{equation}
\label{macdo}
H_\mu[X;q,t] = \sum_\lambda K_{\lambda\mu}(q,t) \,s_\lambda
\,,
\end{equation}
the coefficients $K_{\lambda\mu}(q,t)$, called $q,t$-Kostka polynomials,
belong to $\Z_{\ge0}[q,t]$.  Although the final piece
in the proof of this conjecture was made by Haiman \cite{H} using
representation theoretic and geometric methods, the long study of this
conjecture brought forth many further problems and theories.  The
study of the $q,t$-Kostka polynomials remains a matter of great interest.

It was conjectured in \cite{LLM} that by fixing an integer $k>0$, any
Macdonald polynomial indexed by $\lambda\in \Bdd^k$ (the set of partitions
such that $\lambda_1\leq k$) could be decomposed
as:
\begin{equation}
 \label{E:kSchurMac}
H_{\mu}[X;q,t\, ] = \sum_{\lambda\in\Bdd^k}
K_{\lambda \mu}^{(k)}(q,t) \, s_{\lambda}^{(k)}[X;t\, ] \quad\text{where}\quad
K_{\lambda \mu}^{(k)}(q,t) \in \Z_{\ge0}[q,t] \, ,
\end{equation}
for some symmetric functions
$s^{(k)}_\la[X;t]$ associated to sets of tableaux called
atoms.  Conjecturally equivalent characterizations of
$s^{(k)}_\la[X;t]$ were later given in \cite{LM:filtr,LLMS}
and the descriptions of \cite{LLM,LM:filtr,LLMS}
are now all generically called (graded) $k$-Schur functions.
A basic property of the $k$-Schur functions is that
\begin{equation}\label{E:klarge}
s_\la^{(k)}[X;t]=s_\la\quad\text{for $k\geq |\lambda|$}\,,
\end{equation}
and it thus follows that Eq.~\eqref{E:kSchurMac} significantly refines
Macdonald's original conjecture since the expansion
coefficient $K_{\lambda\mu}^{(k)}(q,t)$ reduces
to $K_{\lambda\mu}(q,t)$ for large $k$.

Furthermore, it was conjectured that the $k$-Schur functions
satisfy a highly structured filtration, which is our primary focus
here.  To be precise:
\begin{conjecture} \label{CJ:LM}
For $k'>k$ and partitions $\mu\in \Bdd^k$ and $\la\in\Bdd^{k'}$,
there are polynomials $\tilde b_{\mu\la}^{(k\to k')}(t)\in\Z_{\ge0}[t]$
such that
\begin{equation} \label{E:branch}
  \ksf_\mu^{(k)}[X;t] = \sum_{\la\in \Bdd^{k'}} \tilde b_{\mu\la}^{(k\to k')}(t)\,
\ksf_\la^{(k')}[X;t].
\end{equation}
\end{conjecture}
\noindent
In particular, the Schur function expansion of a
$k$-Schur function is obtained from \eqref{E:klarge}
and \eqref{E:branch} by letting $k'\to\infty$.
The remarkable property described in Conjecture~\ref{CJ:LM}
provides a step-by-step approach to understanding $k$-Schur functions since
the polynomials
$\tilde b_{\mu\la}^{(k\to k')}(t)$ can be expressed positively in terms of the
\defit{branching polynomials} $$\tilde b^{(k)}_{\mu\la}(t) :=
\tilde b^{(k-1\to k)}_{\mu\la}(t)\,,$$
via iteration
(tables of branching polynomials are given in Appendix~\ref{Appendix}).

It has also come to light that ungraded $k$-Schur
functions (the case when $t=1$) are intimately tied to
problems in combinatorics, geometry, and representation theory
beyond the theory of Macdonald polynomials.
Thus, understanding the \defit{branching coefficients},
$$
\tilde b^{(k)}_{\mu\la}:= \tilde b^{(k)}_{\mu\la}(1)\,
$$
gives a step-by-step  approach to problems in
areas such as affine Schubert calculus and $K$-theory
(for example, see \S\S\ref{SS:geom}).

Our work here gives a combinatorial description for the
branching coefficients, proving Conjecture~\ref{CJ:LM}
when $t=1$.  We use the ungraded $k$-Schur functions
$s_\lambda^{(k)}[X]$ defined in \cite{LM:ktab}, which coincide
with those defined in \cite{LLMS} terms of strong $k$-tableaux.
Moreover, we conjecture a formula for the branching polynomials
in general.  The combinatorics behind these formulas
involves a certain {\it $k$-shape poset}.

\subsection{$k$-shape poset}
A key development in our work is the introduction of a new family
of partitions called $k$-shapes and a poset on these partitions
(see \S\ref{sec:kshapes} for full details and examples).
Our formula for the branching coefficients is given in
terms of path enumeration in the $k$-shape poset.

For any partition $\la$ identified by its
Ferrers diagram, we define its \defit{$k$-boundary} $\bdy\la$
to be
the cells of $\la$ with hook-length no greater
than $k$.  $\bdy \la$ is a skew shape, to which
we associate compositions $\rs(\la)$ and $\cs(\la)$, where $\rs(\la)_i$
(resp. $\cs(\la)_i$) is the number of cells in the $i$-th row (resp.
column) of $\bdy \la$.
A partition $\la$ is said to be a \defit{$k$-shape} if
both $\rs(\la)$ and $\cs(\la)$ are partitions.
The rank of $k$-shape $\la$ is defined to be $|\bdy \la|$,
the number of cells in its $k$-boundary.
$\Ksh^k$ denotes the set of all $k$-shapes.

We introduce a poset structure on $\Ksh^{k}$ where the partial order
is generated by distinguished downward relations in the poset called
\defit{moves} (Definition \ref{D:defmove}).  The set of $k$-shapes
contains the set $\Core^k$ of all $k$-cores (partitions with no
cells of hook-length $k$) {\it and} the set $\Core^{k+1}$
of $k+1$-cores.  Moreover, the maximal elements of $\Ksh^k$ are given by
$\Core^{k+1}$ and the minimal elements by $\Core^{k}$.
In Definition~\ref{D:charge} we give a charge statistic on moves
from which we obtain an equivalence relation on paths (sequences
of moves) in $\Ksh^k$; roughly speaking, two paths are equivalent if
they are related by a sequence of charge-preserving diamonds
(see Eqs.~\eqref{E:elemequiv}-\eqref{E:charge}).  Charge is thus
constant on equivalence classes of paths.

For $\la,\mu\in\Ksh^k$, $\Path^k(\la,\mu)$ is the set of paths in
$\Ksh^k$ from $\la$ to $\mu$ and $\Patheq^k(\la,\mu)$ is the set of
equivalence classes in $\Path^k(\la,\mu)$. Our main result is that
the branching coefficients enumerate these equivalence classes.  To
be precise, for $\la\in\Core^{k+1}$ and $\mu\in\Core^k$, set
\begin{align}
b_{\mu \lambda}^{(k)}(t) &:=\tilde b_{\rs(\mu)
\rs(\lambda)}^{(k)}(t)
\\
 b_{\mu \lambda}^{(k)} &:=\tilde b_{\rs(\mu) \rs(\la)}^{(k)}
\end{align}
so that
\begin{align} \label{E:bdef}
  s_\mu^{(k-1)}[X] = \sum_{\la\in \Core^{k+1}} b_{\mu\la}^{(k)}
  s_\la^{(k)}[X].
\end{align}
Hereafter, we will label $k$-Schur functions by cores rather than
$k$-bounded partitions using the bijection between $\Core^{k+1}$ and
$\Bdd^k$ given by the map $\rs$.

\begin{theorem} \label{T:branch}
For all $\la \in\Core^{k+1}$ and $\mu\in\Core^k$,
\begin{equation} \label{E:branching}
  b^{(k)}_{\mu\la}
=  |\Patheq^k(\la,\mu)|.
\end{equation}
\end{theorem}

We conjecture that the charge statistic on paths gives the branching polynomials.

\begin{conjecture} \label{CJ:branch}
For all $\la \in\Core^{k+1}$ and $\mu\in\Core^k$,
\begin{equation} \label{E:gradedbranching}
b^{(k)}_{\mu\la}(t)
  = \sum_{[\bp]\in \Patheq^k(\la,\mu)} t^{\charge(\bp)}.
\end{equation}
\end{conjecture}

\subsection{$k$-shape functions}

The proof of Theorem \ref{T:branch} relies on the introduction
of a new family of symmetric functions
indexed by $k$-shapes.  These functions generalize the
dual (affine/weak) $k$-Schur functions studied in
\cite{LM:QC,Lam:affstan, LLMS}.

The images of the dual $k$-Schur functions
$\{\WS^{(k)}_\la[X]\}_{\la\in\Core^{k+1}}$ form a basis for the
quotient
\begin{align}\label{E:symmquotient}
\Lambda/I_k \quad \text{where}\quad I_k = \langle
m_\lambda:\lambda_1>k\rangle
\end{align}
of the space $\Lambda$ of symmetric functions over $\mathbb{Z}$,
while the ungraded $k$-Schur functions
$\{s^{(k)}_\la[X]\}_{\la\in\Core^{k+1}}$ form a basis for the
subring $\Lambda^{(k)}=\mathbb{Z}[h_1,\dots, h_k]$ of $\Lambda$. The
Hall inner product $\ip{\cdot}{\cdot}:\La \times \La \to \mathbb{Z}$
is defined by $\ip{m_\la}{h_\mu}=\delta_{\la\mu}$. For each $k$
there is an induced perfect pairing $\ip{\cdot}{\cdot}_k:
\Lambda/I_k \times \Lambda^{(k)} \to \mathbb{Z}$, and it was shown
in \cite{LM:QC} that
\begin{equation} \label{Eq:dual1}
\ip{\WS^{(k)}_\la[X]}{s_\mu^{(k)}[X]}_k = \delta_{\la\mu}
\end{equation}
Moreover, it was shown in \cite{Lam:kSchur} that
$\{\WS^{(k)}_\la[X]\}_{\la\in\Core^{k+1}}$ represents Schubert
classes in the cohomology of the affine Grassmannian
$\Gr_{SL_{k+1}}$ of $SL_{k+1}$.

The original characterization of $\WS^{(k)}_\la[X]$ was given in
\cite{LM:QC} using $k$-tableaux. A $k$-tableau encodes a sequence of
$k+1$-cores
\begin{equation}
\label{seqshapes}
\emptyset = \la^{(0)} \subset \la^{(1)}\subset \dotsm\subset
\la^{(N)}=\la\,,
\end{equation}
where $\la^{(i)}/\la^{(i-1)}$
are certain horizontal strips.
The weight of a $k$-tableau $T$ is
\begin{equation}
\label{kweight} \wt(T)=(a_1,a_2,\dotsc,a_N)\quad\text{where} \quad
a_i = |\bdy \la^{(i)}|-|\bdy \la^{(i-1)}|\,.
\end{equation}
For $\la$ a $k+1$-core, the dual $k$-Schur function is defined as
the weight generating function
\begin{equation}
\WS^{(k)}_\la[X] =
\sum_{T\in \Weak_\lambda^k} x^{\wt(T)}\,,
\end{equation}
where $\Weak_\lambda^k$ is the set of $k$-tableaux of shape
$\lambda$.

Here we consider \defit{$k$-shape tableaux}.
These are defined similarly, but now
we allow the shapes in \eqref{seqshapes} to be
$k$-shapes and $\la^{(i)}/\la^{(i-1)}$
are certain \defit{reverse-maximal} strips
(defined in \S\ref{sec:strips}).  The weight
is again defined by \eqref{kweight} and
for each $k$-shape $\la$, we then define the \defit{cohomology
$k$-shape function} $\dkr_\la^{(k)}$ to be the weight generating
function
\begin{align}
  \dkr_{\la}^{(k)}[X] &= \sum_{T\in \RT_{\la}^k} x^{\wt(T)}\,,
\end{align}
where $\RT_{\la}^k$ denotes the set of reverse-maximal $k$-shape
tableaux of shape $\la$.

We show the $k$-shape functions are symmetric and that when $\la$ is
a $k+1$-core,
\begin{equation}
\label{E:reduces} \WS_\la^{(k)}[X]=\dkr_\la^{(k)}[X] \mod I_{k-1}\,,
\end{equation}
(see Proposition~\ref{L:maxcore}).
We give a combinatorial expansion of any $k$-shape function in terms
of dual $(k-1)$-Schur functions.

\begin{theorem} \label{T:ktab_decomp} For $\la\in\Ksh^k$, the
cohomology $k$-shape function $\dkr_\la^{(k)}[X]$ is a symmetric
function with the decomposition
\begin{align}\label{E:ktabdecomp}
\dkr_\la^{(k)}[X] = \sum_{\mu\in \Core^k} |\Patheq^k(\la,\mu)|\,
\WS_\mu^{(k-1)}[X]\,.
\end{align}
\end{theorem}

It is from this theorem that we deduce Theorem \ref{T:branch}.
Letting $\la\in \Core^{k+1}$ and $\mu\in \Core^k$, we have
\begin{align*}
  b_{\mu\la}^{(k)} &= \ip{\WS^{(k)}_\la[X]}{s_\mu^{(k-1)}[X]}_{k} \\
&=\ip{\WS^{(k)}_\la[X]}{s_\mu^{(k-1)}[X]}_{k-1} \\
  &= \ip{\dkr^{(k)}_\la[X]}{s_\mu^{(k-1)}[X]}_{k-1} \\
  &= \Patheq^k(\la,\mu)
\end{align*}
using \eqref{E:bdef}, \eqref{Eq:dual1} for $k-1$, \eqref{E:reduces},
and Theorem \ref{T:ktab_decomp}.

A (homology) $k$-shape function can also be defined for each
$k$-shape $\mu$ by
\begin{equation}\label{E:ksf}
\ksgen_\mu^{(k)}[X;t] = \sum_{\la\in \Core^{k+1}}\sum_{[\bp]\in
\Patheq^k(\la,\mu)} t^{\charge(\bp)}\, \ksf^{(k)}_\la[X;t]\,,
\end{equation}
and its ungraded version is
$\ksgen_\mu^{(k)}[X]:=\ksgen_\mu^{(k)}[X;1]$. We trivially have from
this definition that $\ksgen_\mu^{(k)}[X] = \ksf_\mu^{(k)}[X]$ when
$\mu \in \Core^{k+1}$. Further, from \eqref{E:ksf} at $t=1$,
Theorem~\ref{T:branch} and \eqref{E:bdef}, we have that
\begin{equation} \label{Eq:kmoins1}
  \ksgen_\mu^{(k)}[X] = \ksf_\mu^{(k-1)}[X]\qquad\text{for $\mu\in
  \Core^k$.}
\end{equation}
The Pieri rule for ungraded homology $k$-shape functions is given by
\begin{theorem}\label{T:Pieri}
For $\la \in
\Ksh^k$ and $r \leq k-1$, one has
$$
h_r[X] \, \ksgen^{(k)}_\la[X] = \sum_{\nu \in \Ksh^k} \ksgen^{(k)}_\nu[X]
$$
where the sum is over maximal strips $\nu/\la$ of rank $r$.
\end{theorem}
When $\la$ is a $k$-core, Theorem \ref{T:Pieri} implies the Pieri
rule for $(k-1)$-Schur functions proven in \cite{LM:ktab}.

Here we have introduced the cohomology $k$-shape functions as the
generating function of tableaux that generalize $k$-tableaux
(those defining the dual $k$-Schur functions).  There is another
family of ``strong $k$-tableaux" whose generating functions
are $k$-Schur functions \cite{LLMS}. The generalization of
this family to give a direct characterization of homology $k$-shape
functions remains an open problem
(see \S\S\ref{Subssec:strong} for further details).


Theorems~\ref{T:ktab_decomp} and \ref{T:Pieri} are proved
using an explicit bijection (Theorem \ref{T:pushpullbij}):
\begin{equation}\label{E:weakbij}
\begin{split}
\RT^k_\la &\longrightarrow \bigsqcup_{\mu\in\Core^k} \Weak^k_\mu \times \Patheq^{k}(\la,\mu) \\
T &\longmapsto (U,[\bp])
\end{split}
\end{equation}
such that $\wt(T)=\wt(U)$.
The bulk of this article is in establishing this bijection, which
requires many intricate details.  See \S\S\ref{outline} for pointers
to the highlights of our development.

\subsection{Geometric meaning of branching coefficients}\label{SS:geom}

It is proven in
\cite{Lam:kSchur} that the ungraded $k$-Schur functions are
Schubert classes in the homology of the affine Grassmannian
$\Gr_{SL_{k+1}}$ of $SL_{k+1}$.
The ind-scheme $\Gr_{SL_{k+1}}$ is an affine Kac-Moody homogeneous
space and the homology ring $H_*(\Gr_{\SL_{k+1}})$ has a basis of
fundamental homology classes $[X_\la]_*$ of Schubert varieties
$X_\la\subset\Gr_{\SL_{k+1}}$, and $H^*(\Gr_{\SL_{k+1}})$ has the
dual basis $[X_\la]^*$, where $\la$ runs through the set of
$k+1$-cores.

There is a weak homotopy equivalence between $\Gr_{\SL_{k+1}}$ and
the topological group $\Omega SU_{k+1}$ of based loops $(S^1,1)\to(SU_{k+1},\text{id})$ into $SU_{k+1}$.  This induces isomorphisms of dual Hopf algebras
$H_*(\Omega SU_{k+1})\cong H_*(\Gr_{\SL_{k+1}})$ and $H^*(\Omega SU_{k+1})\cong H^*(\Gr_{SL_{k+1}})$.
The Pontryagin product in $H_*(\Omega SU_{k+1})$ is induced by the product in the group $\Omega SU_{k+1}$.

Using Peterson's characterization of the Schubert basis of $H_*(\Gr_{SL_{k+1}})$
and the definition of \cite{LM:ktab} for $s_\la^{(k)}[X]$,
it is shown in \cite{Lam:kSchur} that there is a Hopf algebra isomorphism
\begin{equation}
\label{E:homiso}
\begin{aligned}
    H_*(\Gr_{\SL_{k+1}}) \cong H_*(\Omega SU_{k+1})&\overset{j^{(k)}}\longrightarrow \Z[h_1,h_2,\dotsc,h_k] \subset \Lambda\\
[X_\la]_*   &\longmapsto s_\la^{(k)}[X]
\end{aligned}
\end{equation}
mapping homology Schubert classes to $k$-Schur functions.

Let $i^{(k)}: \Omega SU_k \to \Omega SU_{k+1}$ be the inclusion map and
$i^{(k)}_*:H_*(\Omega SU_k)\to H_*(\Omega SU_{k+1})$ the induced map on homology. We have
the commutative diagram
\begin{equation}
\label{E:branchiso}
\begin{diagram}
\node{H_*(\Omega SU_k)} \arrow{e,t}{j^{(k-1)}} \arrow{s,t}{i^{(k)}_*} \node{\Z[h_1,h_2,\dotsc,h_{k-1}]} \arrow{s,b}{\text{incl}} \\
\node{H_*(\Omega SU_{k+1})} \arrow{e,b}{j^{(k)}} \node{\Z[h_1,h_2,\dotsc,h_k]}
\end{diagram}
\end{equation}
Then
\begin{equation}\label{E:geombranchhom}
  i^{(k)}_*([X_\mu]_*) = \sum_{\la\in\Core^{k+1}} b^{(k)}_{\mu\la} [X_\la]_* \qquad\text{for $\mu\in\Core^k$}.
\end{equation}
It is shown using geometric techniques that $b^{(k)}_{\mu\la} \in \Z_{\geq 0}$ in \cite{Lam:note}.

This entire picture can be dualized.  There is a Hopf algebra isomorphism
\cite{Lam:kSchur}
\begin{equation}
\label{E:cohomiso}
\begin{aligned}
 H^*(\Gr_{\SL_{k+1}}) &\longrightarrow \La/I_k  \\
 [X_\la]^*&\longmapsto \WS_\la^{(k)}[X]
\end{aligned}
\end{equation}
mapping cohomology Schubert classes to dual $k$-Schur functions.
 Writing $i^{(k)*}:H^*(\Omega SU_{k+1})\to H^*(\Omega SU_k)$
and $\pi^{(k)}:\La/I_k \to \La/I_{k-1}$ for the natural projection, we have the commutative diagram
\begin{equation}
\begin{diagram}
\node{H^*(\Omega SU_{k+1})} \arrow{e,t}{\cong} \arrow{s,t}{i^{(k)*}} \node{\La/I_k} \arrow{s,b}{\pi^{(k)}} \\
\node{H^*(\Omega SU_k)} \arrow{e,b}{\cong} \node{\La/I_{k-1}}
\end{diagram}
\end{equation}

\begin{equation}\label{E:geombranchcohom}
  i^{(k)*} ([X_\la]^*) = \sum_{\mu\in\Core^k} b^{(k)}_{\mu\la} [X_\mu]^* \qquad\text{for $\la\in\Core^{k+1}$}
\end{equation}
Using \eqref{E:cohomiso} and \eqref{E:geombranchcohom}, one has
\begin{equation}\label{E:weakbranch}
  \pi^{(k)}(\WS^{(k)}_\la[X]) = \sum_{\mu\in\Core^k} b^{(k)}_{\mu\la} \,\WS^{(k-1)}_\mu[X]
\end{equation}

The combinatorics of this article is set in the cohomological side of
the picture.  However, we also speculate that the $k$-shape functions
$\ksgen_\la^{(k)}[X]$ ($\la \in \Ksh^k$) represent naturally-defined
finite-dimensional subvarieties of $\Gr_{\SL_{k+1}}$, interpolating between
the Schubert varieties of $\Gr_{\SL_{k+1}}$ and (the image in $\Gr_{\SL_{k+1}}$
of) the Schubert varieties of $\Gr_{\SL_{k}}$.  Definition \eqref{E:ksf}
would then express the decomposition of this subvariety in terms of Schubert
classes in $H_*(\Gr_{\SL_{k+1}})$.

\subsection{$k$-branching polynomials and strong $k$-tableaux}
\label{Subssec:strong}

The results of this paper suggest an approach to proving
Conjecture \ref{CJ:branch}.  Recall that the conjecture concerns
the graded $k$-Schur functions $s_\lambda^{(k)}[X;t]$, for
which there are several conjecturally equivalent characterizations.
Our approach lends itself to
proving the conjecture for the description of $k$-Schur
functions given in \cite{LLMS}; that is,
as the weight generating function of strong $k$-tableaux:
\begin{equation} \label{E:strongdef}
  \ksf_\la^{(k)}[X;t] = \sum_{T\in\Strong_\la^{k+1}} x^{\wt(T)} t^{{\rm spin}(T)}
\end{equation}
where $\Strong_\la^{k+1}$ is the set of strong $(k+1)$-core tableaux of
shape $\la$ and ${\rm spin}(T)$ is a statistic assigned to strong tableaux.
Note, it was shown \cite{LLMS} that the $s_\lambda^{(k)}$ used in
this article equals the specialization of this function when $t=1$.

To prove Conjecture \ref{CJ:branch}, it suffices to give a bijection for
each $\mu\in\Core^k$:
\begin{equation}\label{E:strongbij}
\begin{split}
\Strong^{k}_\mu &\to \bigsqcup_{\la\in\Core^{k+1}} \Strong^{k+1}_\la \times \Patheq^{k}(\la,\mu) \\
U' &\mapsto (T',[\bp])
\end{split}
\end{equation}
such that
\begin{equation}
\wt(U')=\wt(T')\quad\text{and}\quad
{\rm spin}(U') ={\rm spin}(T')+\charge(p).
\end{equation}
To achieve this, the notion of strong strip (defined on cores) needs
to be generalized to certain intervals
$\mu\subset\la$ of $k$-shapes $\la,\mu\in\Ksh^k$.

We should point out that the symmetry of the $k$-Schur
functions defined by \eqref{E:strongdef} is non-trivial.
A forthcoming paper of
Assaf and Billey \cite{AB} proves this result, as well as the
positivity of $b_{\mu\la}^{k\to\infty}(t)$, using dual equivalence
graphs.  The bijection described above would also give a direct
proof of the symmetry.

\subsection{Tableaux atoms and bijection  \eqref{E:weakbij}}
The earliest characterization of $k$-Schur functions is the tableaux
atom definition of \cite{LLM}.  The definition has the form
\begin{equation}
\label{E:atoms}
\ksf_{\mu}^{(k)}[X;t] = \sum_{T \in \mathbb A_\mu^{(k)}} t^{\charge(T)}
s_{\shape(T)}\,,
\end{equation}
where $\mathbb A_\mu^{(k)}$ is a certain set of tableaux constructed
recursively using katabolism.  It is immediate from the definition that
$$
b_{\mu\lambda}^{k\to\infty}(t)=\sum_{T\in\mathbb A_\mu^{(k)}\atop
\shape(T)=\lambda}t^{\charge(T)}
\,.
$$
Unfortunately, actually determining which tableaux are in an atom
$\mathbb A_\mu^{(k)}$ is an extremely intricate process.

Nonetheless, the construction of our bijection \eqref{E:weakbij} was guided
by the tableaux atoms and has led us to yet another conjecturally
equivalent characterization for the $k$-Schur functions.  In particular,
iterating the bijection from a tableau $T$ of weight $\mu\vdash n$, we get:
\begin{equation}
T \mapsto (T^{(n-1)},[\bp_{n-1}]), \, T^{(n-1)} \mapsto (T^{(n-2)},[\bp_{n-2}]),
\dots, \, T^{(k+1)} \mapsto (T^{(k)},[\bp_{k}])
\,.
\end{equation}
Namely, this provides a bijection between
$T$ and $(T^{(k)},[\bp_{n-1}],\dots,[\bp_k])$.
We then say that $T^{(k)}$ is the $k$-tableau associated to $T$
and conjecture that
\begin{conjecture}  Let $\rho$ be the  unique element of $\Core^{k+1}$
such that $\rs(\rho)=\mu$, and let
$T_{\mu}^{(k)}$ be the unique $k$-tableau of
weight $\mu$ and shape $\rho$ (see \cite{LM:cores}).
Then
\begin{equation}
\mathbb A_\mu^{(k)} = \left\{ T {\rm ~of~weight~} \mu \, \big| \, T_{\mu}^{(k)}
{\rm ~is~the~} k{\text -}{\rm tableau~associated~to~} T \right\}
\,.
\end{equation}
\end{conjecture}
Support for this conjecture is given in \cite{LP} where it
is shown that the bijection between $T$ and
$(T^{(k)},[\bp_{n-1}],\dots,[\bp_k])$ is compatible with charge.
In particular, it is shown that one can define a charge on
$k$-tableaux satisfying the relation
\begin{equation}
\charge(T) = \charge(T^{(k)}) + \charge([\bp_{n-1}]) + \cdots +
\charge([\bp_{k}])
\,.
\end{equation}

\subsection{Connection with representation theory}
In his thesis, L.-C. Chen \cite{Chen} defined a family of graded
$S_n$-modules associated to skew shapes whose row shape and column
shape are partitions. Applying the
Frobenius map (Schur-Weyl duality) to the characters of these
modules, one obtains symmetric functions. Chen has a remarkable
conjecture on their Schur expansions, formulated in terms of
katabolizable tableaux.  We expect that if the skew shape is the
$k$-boundary of a $k$-shape $\la$ then the resulting symmetric
function is the homology $k$-shape function $\ksgen_\la[X;t]$
defined in \eqref{E:ksf}. In \cite{Chen}, an important conjectural
connection is also made between the above $S_n$-modules and certain
virtual $GL_n$-modules supported in nilpotent conjugacy classes, via
taking the zero weight space.

Using a subquotient of the extended affine Hecke algebra, J. Blasiak
\cite{B} constructed a noncommutative analogue of the Garsia-Procesi
modules $R_\la$, whose Frobenius image is the modified Hall-Littlewood
symmetric function. In this setup there is
an analogue of katabolizable tableaux and conjectured analogues of
homology $k$-shape functions and the atoms of \cite{LLM} and
\cite{Chen}.

\subsection{Outline}
\label{outline} In \S\ref{sec:kshapes} we define basic objects of
interest here such as $k$-shapes, moves and the $k$-shape poset, and
give some of their elementary properties. In \S\ref{sec:equivalence}
we introduce an equivalence relation on paths in the $k$-shape poset
called diamond equivalence and show that it is generated by a
smaller set of equivalences called elementary equivalences. In
\S\ref{sec:strips} we introduce covers and strips for $k$-shapes,
and prove that there is a unique path in the $k$-shape poset
allowing the extraction of a maximal strip from a given strip
(Proposition~\ref{P:maxstripunique}). In \S\ref{sec:strips} we also
state the main result (bijection \eqref{E:weakbij}) of this article
 (Theorem \ref{T:pushpullbij}) and show how it leads to
Theorem~\ref{T:ktab_decomp} and Theorem~\ref{T:Pieri}.
Elementary properties of the functions $\dkr_{\la}^{(k)}[X]$
and $\ksgen_\la^{(k)}[X]$ such as triangularity and conjugation
are given in \S\S\ref{subsec:elemen}.

The remaining sections, which contain the bulk of the technical
details in this article, are concerned with the proof of bijection
\eqref{E:weakbij} by iteration of the \defit{pushout}. This
bijection sends compatible initial pairs (certain pairs $(S,m)$
consisting of a strip $S$ and a move $m$, both of which start from a
common $k$-shape) to compatible final pairs (certain pairs $(S',m')$
consisting of a strip $S'$ and a move $m'$, both of which end at a
common $k$-shape). The basic properties of the pushout are
established in \S\ref{sec:rowpushout} and \S\ref{sec:columnpushout}.
The most technical parts of this article (\S\ref{sec:pushaug} and
\S\ref{sec:commutingcube}) are devoted to the interaction between
pushouts and equivalences in the $k$-shape poset.
 The basic statement can be summarized as: {\it
pushouts send equivalent paths to equivalent paths}.  In
\S\ref{sec:pullbacks}-\S\ref{sec:pullbackequivalence} we
develop, in a brief form, the pullback, which is inverse to the
pushout (\S\ref{S:prooftheo}).

For those interested in getting a
quick hold on the pushout algorithm on which bijection
\eqref{E:weakbij} relies, we suggest reading the
beginning of  \S\ref{sec:kshapes}, \S\ref{sec:equivalence},
\S\ref{sec:strips} and \S\ref{sec:pushaug}
to get the basic definitions and ideas, along
with  \S\S\ref{SS:maxstripalg} and \S\S\ref{SS:canonical pushout sequence}
that describe canonical processes to
obtain a maximal strip and to perform the pushout respectively.

\setcounter{tocdepth}{1}
\tableofcontents

\begin{acknow}
This project was supported by NSF
grants DMS-0638625, DMS-0652641, DMS-0652648, DMS-0652668, and DMS-0901111.
T.L. was
supported by a Sloan Fellowship.
L.L was supported by FONDECYT (Chile) grant \#1090016 and
by CONICYT (Chile) grant ACT56 Lattices and Symmetry.

This project benefited from discussions with  Jason Bandlow, Francois
Descouens, Florent Hivert, Anne Schilling, Nicolas Thi\'{e}ry, and Mike Zabrocki.
\end{acknow}

\section{The $k$-shape poset}\label{sec:kshapes}

For a fixed positive integer $k$, the object central to our
study is a family of ``$k$-shape" partitions that contains
both $k$ and $k+1$-cores.  The formula for $k$-branching
coefficients counts paths in a poset on the $k$-shapes.  As
with Young order,
we will define the order relation in terms of adding boxes
to a given vertex $\lambda$, but now the added boxes must
form a sequence of ``strings".  Here we introduce
$k$-shapes, strings, and moves -- the ingredients for
our poset.

\subsection{Partitions}
Let $\Part =
\{\la=(\la_1\ge\la_2\ge\dotsm)\in\Z_{\ge0}^\infty\mid
\text{$\la_i=0$ for $i\gg0$}\}$
denote the set of partitions.  Each $\la\in\Part$
can be identified with its Ferrers diagram
$\{(i,j)\in\Z_{>0}^2\mid j\le \la_i\}$.
The elements of $\Z_{>0}^2$ are called cells.  The row and column
indices of a cell $b=(i,j)$ are denoted $\row(b)=i$ and $\col(b)=j$.
We use the French/transpose-Cartesian depiction of $\Z_{>0}^2$:
row indices increase from bottom to top.
The transpose involution on $\Z_{>0}^2$ defined by $(i,j)\mapsto(j,i)$
induces an involution on $\Part$ denoted $\la\mapsto\la^t$.
The \defit{diagonal index} of $b=(i,j)$ is given by $\diag(b)=j-i$
and we then define the \defit{distance} between
cells $x$ and $y$ to be $|\diag(x)-d(y)|$.

The arm (resp. leg) of $b=(i,j)\in \la$ is defined by
$a_\la(b)=\la_i-j$ (resp. $l_\la(b)=\la^t_j-i$) is the number
of cells in the diagram of $\la$ in the row of $b$ to its right (resp. in the column of $b$ and above it).
The hook length of $b=(i,j) \in \la$ is defined
by $h_\la(b) = a_\la(b)+l_\la(b)+1$.
Let $\Core^k=\{\la\in\Part\mid \text{$h_\la(b)\ne k$ for all $b\in\la$}\}$
be the set of $k$-cores.

Let $D=\mu/\la$ be a skew shape, the difference of Ferrers diagrams
of partitions $\mu\supset\la$. Although such a set of cells may
be realized by different pairs of partitions,
unless specifically stated otherwise, we shall use the notation $\mu/\la$
with the fixed pair $\la\subset\mu$ in mind. $D$ is referred to
as \defit{$\la$-addable} and \defit{$\mu$-removable}.
A \defit{horizontal} (resp. \defit{vertical}) strip is
a skew shape that contains at most one cell in each column (resp. row).
A $\la$-addable cell (corner) is a skew shape $\mu/\la$
consisting of a single cell.
Define $\Top_c(D)$ and $\Bot_c(D)$ to be the
top and bottom cells in column $c$ of $D$ and
let $\Right_r(D)$ and $\Left_r(D)$
be the rightmost and leftmost cells in row $r$ of $D$.
Let $c^+$ (resp.  $c^-$) denote the column right-adjacent
(resp. left-adjacent) to column $c$.  Similar notation is used
for rows.

\subsection{$k$-shapes}
\label{SS:kshapes}
The $k$-interior of a partition $\la$ is the subpartition
of cells with hook length exceeding $k$:
$$\Int^k(\la)=\{b\in \la \mid h_\la(b) > k\}\,.$$
The \defit{$k$-boundary} of $\la$ is the skew shape of cells with hook bounded by $k$:
$$\bdy^k(\la) =\la/\Int^k(\la)\,.$$
We define the \defit{$k$-row shape}
$\rs^k(\la)\in\Z_{\ge0}^\infty$
(resp. \defit{$k$-column shape} $\cs^k(\la)\in\Z_{\ge0}^\infty$) of $\la$
to be the sequence giving the numbers of cells in the rows (resp.
columns) of $\bdy^k(\la)$.

\begin{definition} \label{D:kshape}
A partition $\la$ is a \defit{$k$-shape} if
$\rs^k(\la)$ and $\cs^k(\la)$ are partitions.
$\Ksh^k$ denotes the set of $k$-shapes and $\Ksh^k_N=\{\lambda\in\Ksh^k:
|\bdy^k(\la)|=N\}$.
\end{definition}

\begin{example}\label{X:kshape} $\la=(8,4,3,2,1,1,1)\in\Ksh^{4}_{12}$,
since $\rs^4(\la)=(4,2,2,1,1,1,1)$ and $\cs^4(\la)=(3,2,2,1,1,1,1,1)$
are partitions and $|\bdy^4(\la)|=4+2+2+1+1+1+1=12$.
$\mu=(3,3,1)\not\in\Ksh^{4}$ since $\rs^4(\mu)=(2,3,1)$ is not a partition.
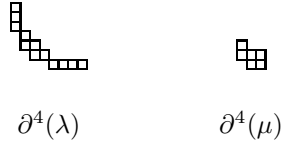
\begin{equation*}
\begin{matrix}
{\tableau[pby]{\\ \\ \\ \bl& \\ \bl&& \\ \bl&\bl& & \\
\bl&\bl&\bl&\bl&&&&}} &&&&& { \tableau[pby]{\\ && \\ \bl&& }} \\ \\
\bdy^4(\la) &&&&& \bdy^4(\mu)
\end{matrix}
\end{equation*}
\end{example}

\begin{remark}
The transpose map is an involution on $\Ksh^k_N$.
\end{remark}

The set of $k$-shapes includes both the $k$-cores and $k+1$-cores.
\begin{proposition}
$\Core^k\subset \Ksh^k$ and
$\Core^{k+1}\subset\Ksh^k$.
\end{proposition}
\begin{proof}
It is shown in
\cite{LM:cores} that
\begin{equation} \label{E:core_to_bounded}
\la \mapsto \rs^k(\la)
\end{equation}
is a bijection from $\Core^{k+1}\to \Bdd^{k}$
implying that $\rs^k(\la)\in\Part$.
Similarly, $\la\mapsto \rs^k(\la^t)=\cs^k(\la)$
is a bijection, and thus
$\Core^{k+1}\subset \Ksh^{k}$.
In particular $\Core^k\subset\Ksh^{k-1}$.
For $\la\in\Core^k$ we have $\bdy^k(\la)=\bdy^{k-1}(\la)$,
from which it follows that $\la\in\Ksh^k$.
\end{proof}

Since $k\ge2$ remains fixed throughout, we shall often suppress $k$
in the notation, writing $\bdy\la$, $\rs(\la)$, $\cs(\la)$, $\Ksh$,
and so forth.

\begin{remark}
A $k$-shape $\la$ is uniquely determined by its row shape $\rs(\la)$ and column shape $\cs(\la)$.
\end{remark}

\begin{remark}
\label{R:rowsequal}
Consider a partition $\lambda$ with addable
corners $x$ and $y$ in columns $c$ and $c^+$,
respectively.
If $h_{\lambda}(\Left_{\row(x)}(\bdy\lambda))=k$ then
$\rs(\lambda)_{\row(x)}=\rs(\lambda)_{\row(y)}$
since the cell below
$\Left_{\row(x)}(\bdy\lambda)$ is not in
$\bdy\lambda$.
\end{remark}

\begin{remark}
\label{R:rowshift} Suppose for some $c, p\ge1$ and $\mu\in\Ksh$, the
cells $\Top_j(\bdy\mu)$ for $c\le j < c+p$, all lie in the same row.
As $\cs(\mu)$ and $\Int(\mu)$ are partitions, it follows that the
cells $\Bot_j(\bdy\mu)$ lie in the same row (say the $r$-th) for
$c\le j<c+p$. Since $\rs(\mu)$ is a partition, one may deduce that
$\mu_{r-1}\ge \mu_r + p$. In particular, there is a $\mu$-addable
corner in the row of $\Bot_c(\bdy\mu)$ for all columns $c$.
\end{remark}

\subsection{Strings}

Given the $k$-shape vertices, the primary notion to define our order
is a string of cells lying at a diagonal distance $k$ or $k+1$ from
one another.  To be precise, let $b$ and $b'$ be \defit{contiguous}
cells when $|\diag(b)-\diag(b')|\in\{k,k+1\}$.
\begin{remark} \label{R:contig}
Since $\la$-addable cells occur on consecutive diagonals,
a $\la$-addable corner $x$ is contiguous with
at most one $\la$-addable corner above (resp. below) it.
\end{remark}

\begin{definition}\label{D:string}
A \defit{string} of \defit{length} $\ell$ is a skew shape
$\mu/\la$ which consists of cells $\{a_1,\dots,a_\ell\}$,
where $a_{i+1}$ is below $a_i$ and they are contiguous
for each $1\leq i<\ell$.
\end{definition}


Note that all cells in a string $s=\mu/\la$
are $\lambda$-addable and $\mu$-removable.
We thus refer to $\lambda$-addable or $\mu$-removable
strings.  Any string $s=\mu/\la$
can be categorized into one of four types depending on the
elements of $\bdy\la\setminus\bdy\mu$, as described by
the following property.

\begin{property}
\label{P:stringtypes}
For any string $s=\mu/\la=\{a_1,a_2,\dotsc,a_\ell\}$,
let $b_0=\Left_{\row(a_1)}(\bdy\la)$,
$b_\ell=\Bot_{\col(a_\ell)}(\bdy\la)$,
and $b_i = (\row(a_{i+1}),\col(a_i))$ for $1\le i <\ell$.
\begin{equation}
\label{lasetminusmu}
\bdy\la\setminus\bdy\mu =
\begin{cases}
\{b_1,\dotsc,b_{\ell-1}\}
& \text{if $h_\la(b_0)< k$ and $h_\la(b_\ell)<k$}
\\
\{b_0,b_1,\dotsc,b_{\ell-1}\}
& \text{if $h_\la(b_0)=k$ and $h_\la(b_\ell)<k$}
\\
\{b_1,\dotsc,b_{\ell-1},b_\ell\}
& \text{if $h_\la(b_0)<k$ and $h_\la(b_\ell)=k$}
\\
\{b_0,b_1,\dotsc,b_{\ell-1},b_\ell\}
& \text{if $h_\la(b_0)=h_\la(b_\ell)= k$}
\end{cases}
\end{equation}
\end{property}
\begin{proof}
For any $1\le i<\ell$,
$b_i\in\partial \lambda\setminus\bdy\mu$ by definition of contiguous.
Otherwise,
$\Left_{\row(a_1)}(\bdy\la)$ and $\Bot_{\col(a_\ell)}(\bdy\la)$
are the only other cells whose hooks may be $k$-bounded in
$\lambda$ and exceed $k$ in $\mu$.
\end{proof}


\begin{definition}\label{D:stringtypes}
A string $s=\mu/\lambda$ is defined to be one of four types,
{\it cover-type, row-type, column-type}, or {\it cocover-type}
when $\bdy\la\setminus\bdy\mu$ equals the first, second, third,
or fourth set, respectively, given in \eqref{lasetminusmu}.
\end{definition}

It is helpful to depict a string $s=\mu/\la$ by its
\defit{diagram}, defined by the following data: cells
of $s$ are represented by the symbol $\bullet$,
cells of $\bdy\la\setminus\bdy\mu$ a represented by $\circ$,
and cells of $\bdy\mu\cap\bdy\la$ in
the same row (resp. column) as some $\bullet$ or $\circ$
are collectively depicted by
a horizontal (resp. vertical) line segment. The four possible
string diagrams are shown in Figure \ref{F:stringdiags}.
\begin{figure}
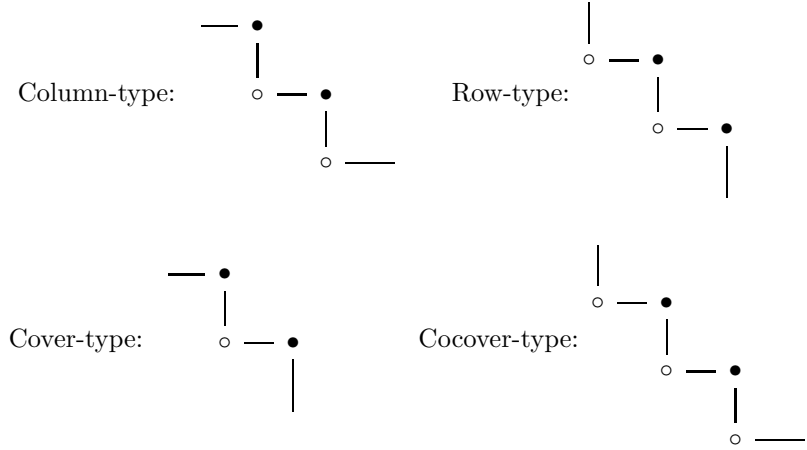

\dgARROWLENGTH=1.1em
$$
\text{Column-type:}
\begin{diagram}
\node{} \arrow{e,-} \node{\bullet} \arrow{s,-}\\ \node[2]{\circ}
\arrow{e,-} \node{\bullet} \arrow{s,-}\\ \node[3]{\circ} \arrow{e,-}
\end{diagram}
\qquad \qquad \text{Row-type:}
\begin{diagram}
\node{} \arrow{s,-}\\ \node{\circ} \arrow{e,-} \node{\bullet}
\arrow{s,-}\\ \node[2]{\circ} \arrow{e,-} \node{\bullet}
\arrow{s,-}\\ \node{}
\end{diagram}
$$
$$
\text{Cover-type:}
\begin{diagram}
\node{} \arrow{e,-} \node{\bullet} \arrow{s,-}\\ \node[2]{\circ}
\arrow{e,-} \node{\bullet} \arrow{s,-}\\ \node{}
\end{diagram}
\qquad \qquad \text{Cocover-type:}
\begin{diagram}
\node{} \arrow{s,-}\\ \node{\circ} \arrow{e,-} \node{\bullet}
\arrow{s,-}\\ \node[2]{\circ} \arrow{e,-} \node{\bullet} \arrow{s,-}
\\ \node[3]{\circ} \arrow{e,-}
\end{diagram}
$$
\caption{Types of string diagrams}
\label{F:stringdiags}
\end{figure}


Given a string $s=\mu/\lambda=\{a_1,\dotsc,a_\ell\}$,
of particular importance are the columns and rows
in its diagram that contain only a $\circ$ or only
$\bullet$.  To precisely specify such rows and columns,
we need some notation.  For a skew shape $D=\mu/\la$, define
$\change_\rs(D)=\rs(\mu)-\rs(\la)\in\Z^\infty$.
The \defit{positively} (resp. \defit{negatively})
\defit{modified rows} of $D$ are those corresponding to
positive (resp. negative) entries in $\change_\rs(D)$.
Similar definitions apply for columns.
It is clear from the Figure~\ref{F:stringdiags} diagrams
that a given string has at most one positively or negatively
modified row and column.  Such rows and columns are
earmarked as follows, given they exist:
\begin{itemize}
\item
$c_{s,u}$ is the unique column negatively modified by $s$.
Equivalently, $c_{s,u}=\col(\Left_{\row(a_1)}(\bdy\lambda))$
if and only if the leftmost column in the diagram of $s$ has a $\circ$

\item $r_{s,d}$ is the unique row negatively modified by $s$.
Equivalently, $r_{s,d}=\row(\Bot_{\col(a_\ell)})$
iff the lowest row in the diagram of $s$ has a $\circ$

\item
$r_{s,u}$ is the unique row positively modified by $s$.
Equivalently, $r_{s,u}=\row(a_1)$ iff
the topmost row in the diagram of $s$ has no $\circ$

\item
$c_{s,d}$ is the unique column positively modified by $s$.
Equivalently, $c_{s,d} = \col(a_\ell)$ if
the rightmost column in the diagram of $s$ has no $\circ$.
\end{itemize}
\noindent
Note that $c_{s,u}<\col(a_1)$ and $r_{s,d}<\row(a_\ell)$ when defined.


\begin{remark}
\label{R:stringchange}
For a $\la$-addable string $s$,
we have the following vector equalities in the
free $\Z$-module $\Z^\infty=\bigoplus_{i\in\Z_{>0}} \Z e_i$
with standard basis $\{e_i\mid i\in \Z_{>0}\}$:
\begin{align}
\label{E:stringcschange}
\change_\cs(s) &= e_{c_{s,d}} - e_{c_{s,u}} \\
\label{E:stringrschange}
\change_\rs(s) &= e_{r_{s,u}} - e_{r_{s,d}}
\end{align}
where by convention $e_i=0$ if the subscript $i$ is not defined
(e.g.  $c_{s,u}$ is not defined
when $\Left_{\row(a_1)}(\bdy\lambda)\not\in\bdy\lambda\setminus\bdy\mu$).
\end{remark}

\subsection{Moves}
Our poset will be defined by taking a $k$-shape $\mu$ to be larger
than $\lambda\in \Pi$ when the skew diagram $\mu/\lambda$ is a
particular succession of strings (called a move).  To this end,
define two strings to be
\defit{translates}
when they are translates of each other in $\Z^2$ by a fixed vector,
and their corresponding modified rows and columns agree in size.
Equivalently,
their diagrams have the property that
$\bullet$'s and $\circ$'s appear in the same relative
positions with respect to each other and the
lengths of each corresponding horizontal and vertical segment
are the same.  We will also refer to cells $a_j$ and
$b_j$ as translates when strings $s_1=\{a_1,\ldots,a_\ell\}$
and $s_2=\{b_1,\ldots,b_\ell\}$ are translates.


\begin{definition} \label{D:defmove}
A \defit{row move} $m$ of rank $r$ and length $\ell$
is a chain of partitions
$\la=\la^0\subset\la^1\subset\dotsm\subset\la^r=\mu$
that meets the following conditions:
\begin{enumerate}
\item \label{I:movestart} $\la\in\Ksh$
\item \label{I:movestrings}
$s_i=\la^i/\la^{i-1}$ is a row-type string consisting of $\ell$ cells
for all $1\le i\le r$
\item \label{I:movetranslates} the strings $s_i$ are translates of each other
\item \label{I:movetopcells} the top cells of $s_1, \dotsc, s_r$ occur
in consecutive columns
from left to right
\item \label{I:moveend} $\mu\in\Ksh$.
\end{enumerate}
We say that $m$ is a row move from $\la$ to $\mu$ and
write $\mu=m*\la$ or $m=\mu/\la$.
A \defit{column move} is the transpose analogue of a row move.
A \defit{move} is a row move or column move.
\end{definition}

\begin{example} \label{X:rowmove}
For $k=5$, a row move of length $1$
and rank $3$ with strings $s_1=\{A\}$, $s_2=\{B\}$,
and $s_3=\{C\}$ is pictured below. The lower case
letters are the cells that are removed
from the $k$-boundary when the
corresponding strings are added.
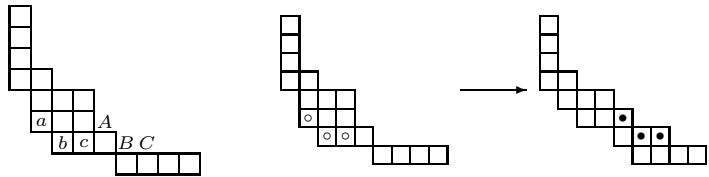
\begin{equation*}
\begin{diagram}
\node{{\scriptsize \tableau[sby]{\\ \\ \\ & \\ \bl&&& \\ \bl&a&&&\bl A
\\\bl&\bl&b&c&&\bl B& \bl C \\ \bl&\bl&\bl&\bl&\bl&&&&}} }
\node{\tiny\tableau[sby]{\\ \\ \\ & \\ \bl&&&\\ \bl&\circ&&\\
\bl&\bl&\circ&\circ& \\ \bl&\bl&\bl&\bl&\bl&&&&}}
\arrow{e,->}
\node{\tiny\tableau[sby]{\\ \\ \\ & \\ \bl&&&\\ \bl&\bl&&&\bullet \\
\bl&\bl&\bl&\bl&&\bullet&\bullet \\ \bl&\bl&\bl&\bl&\bl&&&&}}
\end{diagram}
\end{equation*}
For $k=3$, a row move of length $2$ and rank $2$
with strings $s_1=\{A_1,A_2\}$ and $s_2=\{B_1,B_2\}$ is:
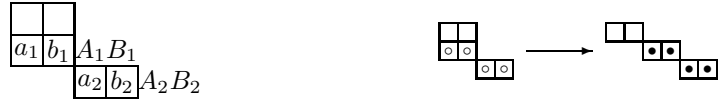
\begin{equation*}
\begin{diagram}
\node{{ \tableau[sby]{&\\ a_1&b_1&\bl A_1&\bl B_1 \\
\bl&\bl&a_2&b_2&\bl A_2 & \bl B_2}}}
\node[2]{\tiny\tableau[sby]{&\\
\circ&\circ \\ \bl&\bl&\circ&\circ}} \arrow{e,->} \node{\tiny\tableau[sby]{& \\
\bl&\bl&\bullet&\bullet \\ \bl&\bl&\bl&\bl&\bullet&\bullet }}
\end{diagram}
\end{equation*}
\end{example}

Note that a row move from $\la$ to $\mu$ merits its name because
$\bdy\mu$ can be viewed as a right-shift of some rows of
$\bdy \la$.  In particular, $|\bdy\mu|=|\bdy\la|$.

\begin{property}\label{C:rowmovecolsize}
If a row move negatively (resp. positively) modifies a column
then it negatively (resp. positively)
modifies all columns of the same size to the right (resp. left).
\end{property}
\begin{proof}
All of the columns positively (resp. negatively)
modified by a row move, are consecutive
and have the same size, by
Definition \ref{D:defmove}\eqref{I:movetranslates},\eqref{I:movetopcells}.
The result follows from Definition \ref{D:defmove}\eqref{I:moveend}.
\end{proof}

A move $m$ is said to be \defit{degenerate}
if $c_{s_r,u}^+=c_{s_1,d}$. Note that a degenerate
move can be of any rank but always has length 1.
The first move in Example~\ref{X:rowmove} is degenerate.

\begin{property}
\label{P:cond5} Condition \eqref{I:moveend}
of Definition~\ref{D:defmove} is equivalent to
\begin{itemize}
\item $\cs(\la)_{c_{s_r,u}}>\cs(\la)_{c_{s_r,u}^+}$ and
$\cs(\la)_{c_{s_1,d}^-}>\cs(\la)_{c_{s_1,d}}$ if $m$ is nondegenerate.
\item $\cs(\la)_{c_{s_r,u}}>\cs(\la)_{c_{s_1,d}}+1$ if $m$ is degenerate.
\end{itemize}
\end{property}
\begin{proof}
The precise column and row modification of a string is pinpointed
in Remark~\ref{R:stringchange} and immediately implies the claim
by definition of $k$-shape.
\end{proof}

\begin{remark}\label{R:movespec}
Consider a $k$-shape $\la$ and a string
$s_1=\lambda^1/\lambda$. If there is  a row
move from $\lambda$ starting as $\lambda\subset\lambda^1$,
then Conditions (3),(4) and (5) of Definition~\ref{D:defmove}
determine $s_2,\ldots,s_r$ (and thus the move).
Note that Property~\ref{C:rowmovecolsize} and Property~\ref{P:cond5} ensures a unique $r$
since $\cs(\la)_{c_{s_r,u}}>\cs(\la)_{c_{s_r,u}^+}$
implies that an extra row type string would not be a translate of
$s_1$.
\end{remark}

\begin{lemma} \label{L:distancestrings}
Suppose $s$ and $t$ are strings in a move $m$
and the cells $x\in s$ and $y\in t$ are translates of each other.
Then $|\diag(x)-\diag(y)| < k-1$.
\end{lemma}
\begin{proof}
Let $m=s_1\cup \cdots \cup s_r$ be a row move from $\la$ to $\mu$
and let $s_j=\{a_1^j, \dots, a_n^j \}$ for $j=1,\ldots,r$.
It suffices to prove the case where $x=a_1^1$ and $y=a_1^r$
are the topmost cells of $s_1$ and $s_r$, respectively.
First suppose that $d(a_1^r)-d(a_1^1) \geq k$.
Then $c_{s_r,u}\geq \col(a_1^1)$ since $a_1^1,a_1^r\in \mu$,
and further, $c_{s_r,u}<\col(a_1^r)$.
Thus $c_{s_r,u}=\col(a_1^j)$ for some $j<r$
since $a_1^1,\ldots,a_1^r$ occur in adjacent columns
by Definition~\ref{D:defmove} of row move.
Moreover, $m$ is a row move implies that $s_j$ and $s_r$
are translates and therefore $\col(a_1^j)=c_{s_r,u}$ of
$\bdy(\la\cup s_1\cup\dotsm\cup s_{j})$
has the same length as $\col(a_1^r)$ in
$\bdy(\la\cup s_1\cup\dotsm\cup s_{r}=\mu)$.
However, column $c_{s_r,u}$ is negatively
modified by $s_r$ implying the contradiction $\mu\not\in \Pi$.
In the case that $d(a_r^1)-d(a_1^1) = k-1$,
the top cell in column $c_{s_r,u}$ of
$\bdy(\la\cup s_1\cup\dotsm\cup s_{r-1})$ is
left-adjacent to $a_1^1$.
However, this column is negatively modified by $s_r$
implying that in $\partial \mu$, it is shorter than
$\col(a_1^1)$.
Again, the assumption that $\mu\in\Ksh$ is contradicted.
\end{proof}

\begin{corollary}
\label{C:rankbound} The rank of a move is at most $k-1$.
\end{corollary}

\begin{property}
\label{P:movestrip} $~$
\begin{enumerate}
\item
If $m$ is a row move where $\mu=m*\la$, then $\mu/\la$ is a horizontal strip
\item If $M$ is a column move where $\mu=M*\la$,
then $\mu/\la$ is a vertical strip
\item Any cell common to a row and a column move
from the same shape $\la$, is a $\la$-addable corner.
\end{enumerate}
\end{property}
\begin{proof}
Consider a row move $m$ from $\la$ to $\mu$ with
strings $s_1,s_2,\dotsc,s_r$ and let $s_1=\{a_1,a_2,\dotsc,a_\ell\}$.
Suppose that $\mu/\la$ is not a horizontal strip.
Since the strings are translates and their topmost
cells occur in consecutive columns by the definition
of move, a violation of the horizontal strip condition must
occur where $a_2$ lies below the top cell $b_1$ of string $s_i$,
for some $i>1$.
Therefore, $|\diag(a_1)-\diag(b_1)|\in \{k-1,k\}$
since the definition of string implies
$|\diag(a_1)-\diag(a_2)|\in \{k,k+1\}$.
However, Lemma~\ref{L:distancestrings} is contradicted
implying $\mu/\la$ is a horizontal strip.
By the transpose argument, we also have that a column move
is a vertical strip.  (1) and (2) imply (3).
\end{proof}

\begin{proposition} \label{P:uniquedecomp}
Let $m$ be a row or column move from $\lambda$ to $\mu$.  Then
the decomposition of $m=\mu/\lambda$ into strings (according
to Definition~ \ref{D:defmove}) is unique.
\end{proposition}
\begin{proof}
Given row move $m$ from $\lambda$ to $\mu$,
Remark~\ref{R:movespec} implies it suffices to
show that the $\lambda$-addable string $s_1$ is uniquely
determined.  By \eqref{E:stringcschange} and
Definition \ref{D:defmove}\eqref{I:movetranslates},
for any $m=\mu/\lambda = \{s_1,\ldots,s_r\}$,
\begin{equation}\label{E:rowmovecschange}
\change_\cs(\mu/\lambda) = \sum_{j=0}^{r-1} (e_{c_{s_1,d}+j}-e_{c_{s_1,u}+j}).
\end{equation}
Since $c_{s_r,u} < \col(a_1)\le \col(a_\ell)=c_{s_1,d}$ there is no
cancellation in this formula, so the rank of $m$ can be read from the
number of consecutive $+1$'s in $\change_\cs(\mu/\lambda)$
(and is independent of $s_1,\ldots,s_r$).
The length of $m$ is then simply $|\mu/\lambda|/r$.
Since the leftmost cell of the horizontal strip $m$ must be
the top cell of the first string $s$ of $m$ and the length
of $s$ is determined, by Remark \ref{R:contig} it follows
that the $\la$-addable
string $s_1$ is determined.
\end{proof}

\subsection{Poset structure on $k$-shapes} \label{SubS:poset}
We endow the set $\Ksh_N$ of $k$-shapes of fixed size $N$,
with the structure of a directed acyclic graph with an edge from $\la$ to $\mu$
if there is a move from $\la$ to $\mu$.
Since a row (resp. column) move from $\la$ to $\mu$ satisfies $\rs(\la)=\rs(\mu)$ and
$\cs(\la) \dom \cs(\mu)$ (resp. $\cs(\la)=\cs(\mu)$ and $\rs(\la)\dom\rs(\mu)$),
this directed graph induces a poset structure on $\Ksh_N$
which is a subposet of
the Cartesian square of the dominance order $\dom$ on
partitions of size $N$.
\begin{proposition}  An element of the $k$-shape poset is maximal
(resp. minimal)
if and only if it is a
$(k+1)$-core (resp. $k$-core).
\end{proposition}

\begin{proof}
Since a $k$-core has no hook sizes of size $k$, it also has no row-type or column-type strings addable.  Thus $k$-cores are minimal elements of the $k$-shape poset.  Now suppose $\la$ is a minimal element of the $k$-shape poset, and suppose $\la$ has a hook of size $k$.  Let us take the rightmost such cell of $\bdy \la$, say $b$.  Then there is a $\la$-addable corner $a_1$ at the end of the row of $b$.  Let $s = \{a_1,a_2,\ldots,a_\ell\}$ be the longest row-type string with top cell $a_1$ (see Lemma \ref{L:makestring}).

Suppose $\cs(\la)_{\col(b)} = \cs(\la)_{\col(b)+1} = \cdots = \cs(\la)_{\col(b)+t} > \cs(\la)_{\col(b)+t+1}$.  Then the bottom cells in $\bdy \la$ of columns $\col(b),\col(b) + 1,\ldots, \col(b) + t+1$ all lie in the same row as $b$, for otherwise such a cell would have a hook-length of size $k$ (or $\la$ would not be a $k$-shape).  Since $\la$ is a $k$-shape, there are $t$ successively addable cells to the right of $a_1$, on the same row as $a_1$.  A similar argument shows that we can in fact find $t+1$ row-type strings $s=s_1, s_2, \ldots, s_{t+1}$ whose cells are on exactly the same set of rows and which have identical diagrams.  We claim that $m = s_1 \cup s_2 \cup \cdots \cup s_{t+1}$ is a row move on $\la$.  Let $\mu = m*\la$.  By construction, $\cs(\mu)_{\col(b)+t} \geq \cs(\mu)_{\col(b)+t+1}$.  It thus suffices to show that $\cs(\mu)_{\col(a_\ell)-1} \geq \cs(\mu)_{\col(a_\ell)}$.  Since $s$ is row-type, and cannot be extended further below, the cell $d = \Bot_{\col(a_\ell)}(\bdy \la)$ has hook length $< k-1$.  Suppose $\cs(\la)_{\col(a_\ell)-1} = \cs(\la)_{\col(a_\ell)}$.  Since $a_\ell$ is $\la$-addable, the bottom cell $c$ of column $\col(a_\ell)-1$ in $\bdy \la$ must be above the bottom of column $\col(a_\ell)$.  But the cell $c'$ directly below $c$ has hook length $h_\la(c') \leq h_\la(d) + 2 <k+1$.  This is a contradiction.

The proof that the $(k+1)$-cores are exactly the maximal elements is similar.
\end{proof}

\begin{example} \label{X:kshapeposet} The graph $\Ksh^2_4$ is pictured below. Only the
cells of the $k$-boundaries are shown. Row moves are indicated by $r$ and column moves by $c$.
\dgARROWLENGTH=2.5em
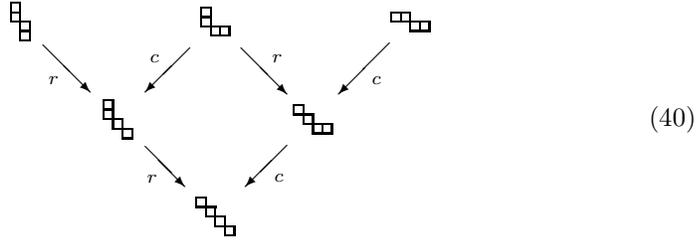
\begin{equation}
\begin{diagram} \label{E:poset2,4}
\node{\tableau[pby]{ \\ \\ \bl& \\ \bl& }} \arrow{se,b}{r} %
\node[2]{\tableau[pby]{ \\ \\ \bl& & }} \arrow{sw,t}{c} \arrow{se,t}{r} %
\node[2]{\tableau[pby]{ & \\ \bl&\bl&&}} \arrow{sw,b}{c} \\
\node[2]{\tableau[pby]{ \\ \\ \bl & \\ \bl&\bl& }} \arrow{se,b}{r} %
\node[2]{\tableau[pby]{\\ \bl&\\ \bl&\bl&&}} \arrow{sw,b}{c} \\
\node[3]{\tableau[pby]{ \\ \bl& \\ \bl&\bl& \\ \bl&\bl&\bl&}}
\end{diagram}
\end{equation}
The graph $\Ksh^3_5$ is pictured below.
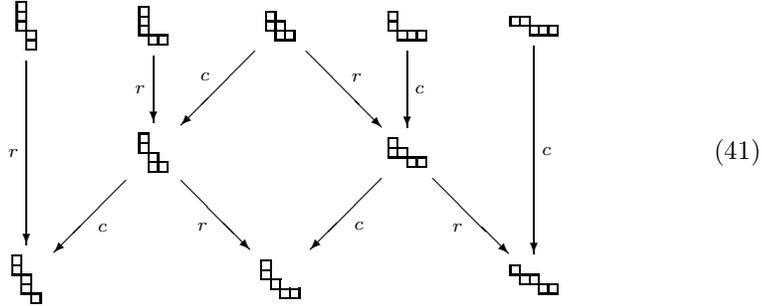
\begin{equation}
\begin{diagram} \label{E:poset3,5}
\node{\tableau[pby]{ \\ \\ \\ \bl& \\ \bl& }} \arrow[2]{s,t}{r} %
\node{\tableau[pby]{ \\ \\ \\ \bl&&}} \arrow{s,t}{r}  %
\node{\tableau[pby]{ \\ & \\ \bl&& }} \arrow{sw,t}{c} \arrow{se,t}{r} %
\node{\tableau[pby]{ \\ \\ \bl &&&}} \arrow{s,b}{c} %
\node{\tableau[pby]{ & \\ \bl&\bl&&& }} \arrow[2]{s,b}{c} \\
\node[2]{\tableau[pby]{\\ \\ \bl& \\ \bl & & }} \arrow{sw,b}{c} \arrow{se,b}{r} %
\node[2]{\tableau[pby]{\\ & \\ \bl& \bl & & }} \arrow{sw,b}{c} \arrow{se,b}{r} \\
\node{\tableau[pby]{ \\ \\ \bl& \\ \bl&\\ \bl&\bl&}} %
\node[2]{\tableau[pby]{\\ \\ \bl& \\ \bl&\bl&&}} %
\node[2]{\tableau[pby]{\\ \bl&& \\ \bl&\bl&\bl&&}}
\end{diagram}
\end{equation}
\end{example}

\subsection{String and move miscellany}
Here we highlight a number of lemmata about
strings that will be needed later.


In the special case that $\mu$ or $\lambda$ is a $k$-shape,
the string $s=\mu/\lambda$ obeys  a number of
explicit properties.

\begin{lemma}\label{L:makestring}
Let  $\lambda\in\Ksh$ and $s=\{a_1,\dots,a_\ell\}$ be a
$\lambda$-addable string.
\begin{enumerate}
\item If $s$ negatively modifies a row, then it can be extended below to a $\la$-addable string that does not have negatively modified rows.
\item If $s$ negatively modifies a column, then it can be extended above to a $\la$-addable string that does not have negatively modified columns.
\end{enumerate}
\end{lemma}
\begin{proof}
Let $s$ negatively modify a row. By Remark \ref{R:rowshift},
there is a $\la$-addable cell $x$ in the row of
$b=\Bot_{\col(a_\ell)}(\bdy\la)$ and we have $h_\la(b)=k$.
Therefore $d(x)-d(a_\ell)=k+1$ and $s\cup\{x\}$ is a $\la$-addable
string that extends $s$ below. The required string exists by induction.
Part (2) is similar.
\end{proof}

\begin{lemma}  \label{L:comparelengths}
Let  $\lambda\in\Ksh$ and $s=\{a_1,\dots,a_\ell\}$ be a
$\lambda$-addable string.
\begin{enumerate}
\item
$\lambda_{\row(a_{j})-1}-\lambda_{\row(a_j)}
\geq \lambda_{\row(a_i)-1}-\lambda_{\row(a_i)}$ for all $i < j$,
\item
$\lambda^t_{\col(a_{j})-1}-\lambda^t_{\col(a_j)} \geq
\lambda^t_{\col(a_i)-1}-\lambda^t_{\col(a_i)}$ for all $j > i$,
\end{enumerate}
with the convention that
$\lambda_{\row(a_\ell)-1}$ (resp. $\lambda^t_{\col(a_1)-1}$) is infinite
if $a_\ell$ (resp. $a_1$) lies in the first row (resp. column) of $\lambda$.
\end{lemma}
\begin{proof}
Part (1) follows from Remark~\ref{R:rowshift} and Part (2) follows by
transposition.
\end{proof}

\begin{lemma} \label{L:distancecontiguouscells}
Let $s=\{a_1,\dots,a_n\}$ and $t=\{b_i,b_{i+1}\}$ be $\la$-addable strings for
some $\la\in\Ksh$. If
$\row(b_i)=\row(a_j)+1$ and
$\row(b_{i+1})=\row(a_{j+1})+1$  (or
$\col(b_i)=\col(a_j)+1$ and
$\col(b_{i+1})=\col(a_{j+1})+1$), then
$|\diag(b_i)-\diag(b_{i+1})|\leq
|\diag(a_j)-\diag(a_{j+1})|$.
This also holds if $s$ and $t$ are $\la$-removable.
%
%
\end{lemma}
\begin{proof}
Note that for any $x$ contiguous to and higher than $x'$,
$\cs(\lambda)_{\col(x)}=\row(x)-\row(x')$
and $|\diag(x)-\diag(x')|=
\cs(\lambda)_{\col(x)}+ \rs(\lambda)_{\row(x')}$.
Thus, $\cs(\lambda)_{\col(b_i)}=\cs(\lambda)_{\col(a_j)}$.
Since $\lambda\in \Pi$ implies that
$\rs(\lambda)_{\row(b_{i+1})}\leq
\rs(\lambda)_{\row(a_{j+1})}$ we then have our claim.
\end{proof}

\begin{lemma} \label{L:addablecornerbelow}
Let $\lambda\in\Ksh$.  Consider a $\lambda$-addable corner $b$
and some $x\not\in\lambda$ in a lower row than $\row(b)$
that is right-adjacent to a cell in $\lambda$.
If $|\diag(b)-\diag(x)|=k+1$ then $x$ is $\lambda$-addable
and if $|\diag(b)-\diag(x)|=k$ then either $x$ or the cell
immediately below $x$ is $\lambda$-addable.
\end{lemma}
\begin{proof}
When $|\diag(b)-\diag(x)|=k+1$, $\Bot_{\col(b)}(\bdy\lambda)$ is
in the row of $x$ and thus Remark~\ref{R:rowshift}
implies that $\row(x)$ has $\lambda$-addable corner
(namely $x$).  If $|\diag(b)-\diag(x)|=k$, then either
$\Bot_{\col(b)}(\bdy\lambda)$ is in the row of $x$
(and as before, $x$ is $\lambda$-addable) or
$\Bot_{\col(b)}(\bdy\lambda)$ is in the row below $x$. 
If $x$ is not $\la$-addable then the latter case holds and the cell immediately
below $x$ is $\la$-addable.
\end{proof}

\begin{lemma} \label{C:samesize}
Let $m$ be a row or column move from $\lambda$ to $\mu$.
For any cells $a,b\in m$ that are translates of each other,
$\rs(\lambda)_{\row(a)}=
\rs(\lambda)_{\row(b)}$,  $\cs(\lambda)_{\col(a)}=
\cs(\lambda)_{\col(b)}$, $\rs(\mu)_{\row(a)}=\rs(\mu)_{\row(b)}$
and  $\cs(\mu)_{\col(a)}=
\cs(\mu)_{\col(b)}$.
\end{lemma}
\begin{proof}
Consider the case that $m$ is a row move (the
column case follows by transposition).  By definition of move,
the strings of $m$
have diagrams which are translates of
each other. Since
Property~\ref{P:movestrip} implies the strings never lie on top
of each other, if cell $a$  is the translate of
cell $b$ then $\cs(\lambda)_{\col(a)}=
\cs(\lambda)_{\col(b)}$ and  $\cs(\mu)_{\col(a)}=
\cs(\mu)_{\col(b)}$.  Since strings in a row move never change the row reading
we have by translation of diagrams that $\rs(\lambda)_{\col(a)}=
\rs(\lambda)_{\col(b)}$ and  $\rs(\mu)_{\col(a)}=
\rs(\mu)_{\col(b)}$.
\end{proof}

Let $b$ be a cell in a skew shape $D$. Define the
\defit{indent} of $b$ in $D$ by $\Ind_D(b)=\col(b)-\col(\Left_{\row(b)}(D))$;
this is
the number of cells strictly to the left of $b$ in its row in $D$.
If $D$ is a horizontal strip and $b\in D$ then
$b$ is $\la$-addable if and only if $\Ind_D(b)=0$.

\begin{lemma} \label{L:addable} Let $\la\in\Ksh$, $m$ a row move from $\la$,
and $s$ a string of $m$. Then $\Ind_m(b)$ is constant for $b \in s$.
In particular, if some cell of $m$ is $\la$-addable,
then so is every cell in its string.
\end{lemma}
\begin{proof} The first assertion follows by induction on $\Ind_m(b)$
using Definition \ref{D:defmove}\eqref{I:movetranslates}. The second
holds by Proposition~\ref{P:uniquedecomp}.
\end{proof}

\section{Equivalence of paths in the $k$-shape poset}
\label{sec:equivalence}

\subsection{Diamond equivalences}

\begin{definition} \label{D:charge}
Given a move $m$, the \defit{charge} of $m$, written $\charge(m)$, is
$0$ if $m$ is a row move and $r\ell$ if $m$ is
a column
move of length $\ell$ and rank $r$. Notice that in the column case,
$r \ell$ is simply the number of cells in the move $m$
when viewed as a skew shape.  The charge of a path $(m_1,\dots,m_n)$ in $\Ksh_N$
is
$\charge(m_1)+\cdots+\charge(m_n)$, the sum of the charges of the moves that
constitute the path.
\end{definition}

Let $\equiv$ be the equivalence relation on directed paths in $\Ksh_N$
generated by the following {\it diamond equivalences}:
\begin{equation} \label{E:elemequiv}
\tM m \equiv \tm M
\end{equation}
where $m,M,\tm,\tM$ are moves (possibly empty) between $k$-shapes
such that the diagram
 \dgARROWLENGTH=2.5em
\begin{equation} \label{E:commudiagram}
\begin{diagram}
\node[2]{\la} \arrow{sw,t}{m} \arrow{se,t}{M} \\ \node{\mu}
\arrow{se,b}{\tilde M} \node[2]{\nu} \arrow{sw,b}{\tilde m} \\
\node[2]{\gamma}
\end{diagram}
\end{equation}
commutes and the charge is the same on both sides of the diamond:
\begin{equation} \label{E:charge}
\charge(m)+\charge(\tM)=\charge(M)+\charge(\tm).
\end{equation}
The commutation is equivalent to the equality $\tM\cup m = \tm\cup
M$ where a move is regarded as a set of cells.  Observe that
the charge is by definition constant
on equivalence classes of paths.

\begin{example}
Continuing Example \ref{X:kshapeposet}, the two paths in $\Ksh^2_4$ from $\la=(3,1,1)$ to
$\nu=(4,3,2,1)$ have charge 2 and 3 respectively, and so are not equivalent.  Thus by Theorem \ref{T:branch}  one has $b^{(k)}_{\mu\la} = 2$, and according to Conjecture \ref{CJ:branch}, we have $b^{(k)}_{\mu\la} = t^2 + t^3$.

\dgARROWLENGTH=2.5em
\begin{equation}
\begin{diagram}
\node[2]{\tableau[pby]{ \\ \\ \bl& & }} \arrow{sw,t}{c} \arrow{se,t}{r} \\
\node{\tableau[pby]{ \\ \\ \bl & \\ \bl&\bl& }} \arrow{se,b}{r} %
\node[2]{\tableau[pby]{\\ \bl&\\ \bl&\bl&&}} \arrow{sw,b}{c} \\
\node[2]{\tableau[pby]{ \\ \bl& \\ \bl&\bl& \\ \bl&\bl&\bl&}}
\end{diagram}
\end{equation}

The two paths in $\Ksh^3_5$ from $\la=(3,2,1)$ to
$\nu=(4,2,1,1)$ are diamond equivalent, both having charge 1.  Thus by Theorem \ref{T:branch}  one has $b^{(k)}_{\mu\la} = 1$, and according to Conjecture \ref{CJ:branch}, we have $b^{(k)}_{\mu\la} = t$.

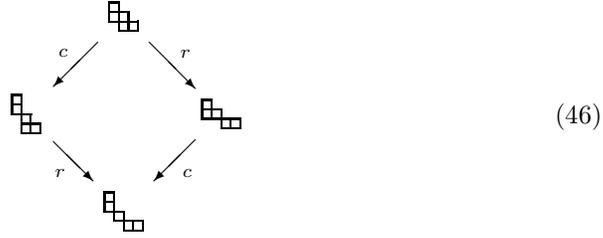
\begin{equation}
\begin{diagram} 
\node[2]{\tableau[pby]{ \\ & \\ \bl&& }} \arrow{sw,t}{c} \arrow{se,t}{r} \\
\node{\tableau[pby]{\\ \\ \bl& \\ \bl & & }}\arrow{se,b}{r} %
\node[2]{\tableau[pby]{\\ & \\ \bl& \bl & & }} \arrow{sw,b}{c}  \\
\node[2]{\tableau[pby]{\\ \\ \bl& \\ \bl&\bl&&}} %
\end{diagram}
\end{equation}
\end{example}

We will describe in more detail in this section when two moves $m$ and
$M$ can obey a diamond equivalence.  We will also see that
the relation $\equiv$ is generated by special diamond
equivalences called {\it elementary equivalences} (see Proposition~\ref{P:diamond}).

\subsection{Elementary equivalences}
We require a few more notions to define elementary equivalence.

Let $m$ and $M$ be moves from $\la\in\Ksh$.  We say that $m$ and $M$
{\it intersect} if they are non-disjoint as sets of cells.  Similarly, we say
that two strings $s$ and $t$ intersect if they have cells in common.
We say that the pair $(m,M)$ is
{\it reasonable} if for every string $s$ and $t$ of $m$ and $M$ respectively
that intersect, we either have $s \subseteq t$ or $t \subseteq s$.

Let $s$ and $t$ be intersecting strings.
Either the highest (resp. lowest) cell of $s\cup t$ is in $s\setminus t$,
or in $s\cap t$, or in $t\setminus s$; in these cases we say that
$s$ continues above (resp. below) $t$, or $s$ and $t$ are matched above (resp. below), or
$t$ continues above (resp. below) $s$.
We say that $m$ continues above (resp. below) $M$,
or $m$ and $M$ are matched above (resp. below),
or $M$ continues above (resp. below) $m$, if the corresponding relation holds for all pairs
of strings $s$ in $m$ and $t$ in $M$ such that $s\cap t\ne \emptyset$.

We say that the disjoint strings $s$ and $t$ are {\it contiguous} if
$s \cup t$ is a string. We say that the moves $m$ and $M$ are not
contiguous if no string of $m$ is contiguous to a string of $M$.

For the sake of clarity, the overall picture is presented first, the
proofs being relegated to Subsections \ref{SS:introwcolmoves},
\ref{SS:introwrowmoves} and \ref{SS:diamond}.

The following lemma asserts that any pair of intersecting strings $s\subset m$
and $t\subset M$ are in the same relative position.
\begin{lemma} \label{L:moves_relative_position} Let $m$ and $M$ be intersecting
$\la$-addable moves for $\la\in\Ksh$. Then
$m$ continues above $M$ (resp. $m$ and $M$ are matched above, resp. $M$ continues above $m$)
if and only if there exist strings $s\subset m$ and $t\subset M$ such that
$s$ continues above $t$ (resp. $s$ and $t$ are matched above, resp. $t$ continues above $s$).
A similar statement holds with the word ``above" replaced by the word ``below".
\end{lemma}

\begin{notation} \label{N:arrow}
For two sets of cells $X$ and $Y$, let $\rightarrow_X(Y)$ (resp.
$\uparrow_X(Y)$) denote the result of shifting to the right (resp.
up), each row (resp. column) of $Y$ by the number of cells of $X$ in
that row (resp. column). Define $\leftarrow_X(Y)$ and
$\downarrow_X(Y)$ analogously.
\end{notation}

\subsection{Mixed elementary equivalence}

\begin{definition} \label{D:rowcolcommute}
A \defit{mixed elementary equivalence} is a relation of the form
\eqref{E:elemequiv}
arising from a
row move $m$ and column move $M$ from some $\la\in\Ksh$, which has
one of the following forms:
\begin{enumerate}
\item $m$ and $M$ do not intersect and $m$ and $M$ are not contiguous.
Then $\tilde m=m$ and
$\tilde M=M$.
\item $m$ and $M$ intersect and
\begin{itemize}
\item[(a)] $m$ continues above and below $M$. Then
\begin{equation*}
\tilde m=\rightarrow_M(m) \qquad\text{and}\qquad
\tilde M=\rightarrow_m(M)
\end{equation*}
\item[(b)] $M$ continues above and below $m$. Then
\begin{equation*}
\tilde m= \uparrow_M(m)\qquad\text{and}\qquad \tilde M=
\uparrow_m(M).
\end{equation*}
\end{itemize}
\end{enumerate}
\end{definition}
\begin{remark}
If the pair $(m,M)$ defines a mixed elementary equivalence
then $m$ and $M$ are reasonable.
\end{remark}

\begin{example} \label{X:rowcol} For $k=4$ the
following diagram defines a mixed elementary equivalence via Case
(2)(a). The black cells indicate those added to the original shape.
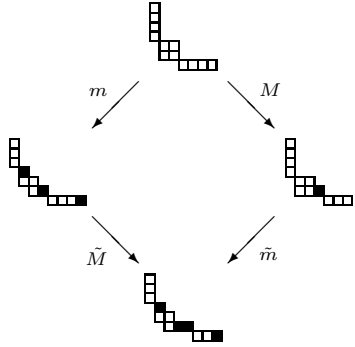
\begin{equation*}
\begin{diagram}
\node[2]{{\tableau[pby]{\\ \\ \\ \\ \bl & & \\ \bl & & \\ \bl&\bl&\bl&&&&\\ }}} \arrow{sw,t}{m} \arrow{se,t}{M} \\
\node{{\tableau[pby]{\\ \\ \\ \bl&\fl \\ \bl&& \\ \bl&\bl& &\fl \\
\bl&\bl&\bl&\bl&&&&\fl }}} \arrow{se,b}{\tilde M}
\node[2]{{\tableau[pby]{\\ \\ \\ \\ \bl & & \\ \bl &&&\fl \\ \bl&\bl&\bl&\bl&&& }}} \arrow{sw,b}{\tilde m} \\
\node[2]{{\tableau[pby]{\\ \\ \\ \bl &\fl \\ \bl &&  \\ \bl &\bl& &
\fl&\fl \\ \bl&\bl&\bl&\bl&\bl&&&\fl }}}
\end{diagram}
\end{equation*}
\end{example}

\begin{proposition}\label{P:relmix} If $(m,M)$ defines a mixed elementary equivalence,
then the prescribed sets of cells $\tm$ and $\tM$
define a diamond equivalence.
\end{proposition}

\subsection{Interfering row moves and perfections}
To define row equivalence we require the notions of interference and
perfections.

Let $m$ and $M$ be row moves from $\la\in\Ksh$ of respective ranks $r$ and $r'$
and lengths $\ell$ and $\ell'$ such that $m\cap M=\emptyset$.

\begin{remark} \label{R:norowrowcontig} Suppose a cell in the string $s$
of $m$ is above and contiguous with a cell in the string $t$ of $M$.
If all the cells of $s$ are not
above all the cells of $t$ then using Property~\ref{P:movestrip} and
Lemma \ref{L:addable} one may deduce the contradiction that $m$ and $M$
intersect.
If the cells of $s$ are above those of $t$, we have a contradiction to
Definition~\ref{D:stringtypes}.
Therefore $m$ and $M$ are not contiguous.
In particular the diagrams of the strings of $m$ and $M$ are
unaffected by the presence of the other move.
\end{remark}

Say that the pair $(m,M)$ is \defit{interfering} if $m \cap M= \emptyset$
and
$\cs(\la\cup m \cup M)$
is not a partition.
Let $m=s_1\cup\dotsm \cup s_r$ and $M=t_1\cup\dotsm\cup t_{r'}$.
We immediately have
\begin{lemma}\label{L:interfere} Suppose $(m,M)$ is interfering.
Say the top cell of $m$ is above the top cell of $M$. Then
\begin{enumerate}
\item \label{I:interfere} $c_{s_1,d}^- = c_{t_{r'},u}$.
In particular $m$ and $M$ are nondegenerate.
\item \label{I:intabovebelow} Every cell of $m$ is above every cell of $M$.
\item \label{I:intoneoff} $\cs(\la)_{c_{s_1},d} = \cs(\la)_{c_{t_{r'},u}} + 1$.
\end{enumerate}
\end{lemma}
\noindent Property (1) tells us that the pair $(m,M)$ can only be interfering
if the last negatively  modified column of $M$
is just before the first positively modified column of $m$
(or similarly with $m$ and $M$ interchanged).

Suppose $(m,M)$ is interfering and the top cell of $m$ is above the
top cell of $M$. A \defit{lower} (resp. upper) \textit{perfection}
of the pair $(m,M)$ is a $k$-shape of the form $\la\cup m \cup M
\cup m_\per$ (resp. $\la\cup m\cup M\cup M_\per$) where $m_\per$
(resp. $M_\per$) is a $(\la\cup m\cup M)$-addable skew shape such
that $m\cup m_\per$ (resp. $M\cup M_\per$) is a row move from
$M*\la$ (resp. $m*\la$) of rank $r$ (resp. $r'$) and length
$\ell+\ell'$ and $M \cup m_\per$ (resp. $m\cup M_\per$) is a row
move from $m*\la$ (resp. $M*\la$) of rank $r+r'$ and length $\ell'$
(resp. $\ell$). We say that $(m,M)$ is lower-perfectible (resp.
upper perfectible) if it admits a lower (resp. upper) perfection. By
Lemma \ref{L:perfection}, the lower (resp. upper) perfection is
unique if it exists.

\begin{example} \label{X:twocompletions} For $k=5$,
row moves $m=\{A\}$ and $M=\{B\}$ from $\la$ are pictured below
together with $\bdy\la$.
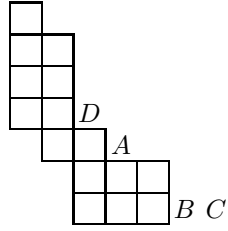
\begin{equation*}
\tableau[sby]{ \\ & \\ & \\ & & \bl D \\ \bl& & &\bl A \\ \bl&\bl& &
& \\ \bl&\bl&&&&\bl B&\bl C}
\end{equation*}
The pair $(m,M)$ is interfering: the skew shape $\bdy(\la\cup m\cup
M)$ is pictured below.
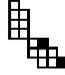
\begin{equation*}
\tableau[pby]{\\ &  \\ & \\ & \\ \bl&\bl &&\fl \\ \bl&\bl& & & \\
\bl&\bl &\bl& & &\fl }
\end{equation*}
The lower and upper perfections both exist, with $m_\per=\{C\}$ and $M_\per=\{D\}$. They are pictured
as the bottom shapes in the left and right diagrams respectively.
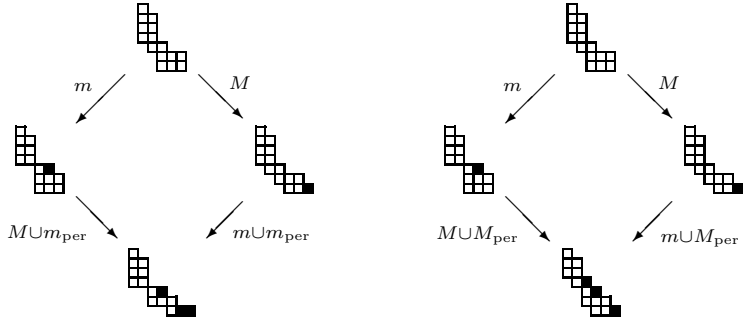
\begin{equation*}
\begin{diagram}
\node[2]{{\tableau[pby]{\\ & \\ & \\ & \\ \bl & & \\ \bl&\bl&&& \\ \bl&\bl&&&}}} \arrow{sw,t}{m} \arrow{se,t}{M} \\
\node{{\tableau[pby]{\\ & \\ & \\ & \\ \bl &\bl &&\fl \\ \bl&\bl&&&
\\ \bl&\bl&&&}}} \arrow{se,b}{M\cup m_\per}
\node[2]{{\tableau[pby]{\\ & \\ & \\ & \\ \bl & & \\ \bl&\bl&&& \\ \bl&\bl&\bl&&&\fl}}} \arrow{sw,b}{m\cup m_\per} \\
\node[2]{{\tableau[pby]{\\ &  \\ & \\  &  \\ \bl&\bl &&\fl \\
\bl&\bl& & & \\ \bl&\bl &\bl&\bl & &\fl&\fl }}}
\end{diagram}\qquad\qquad
\begin{diagram}
\node[2]{{\tableau[pby]{\\ & \\ & \\ & \\ \bl & & \\ \bl&\bl&&& \\ \bl&\bl&&&}}} \arrow{sw,t}{m} \arrow{se,t}{M} \\
\node{{\tableau[pby]{\\ & \\ & \\ & \\ \bl &\bl &&\fl \\ \bl&\bl&&&
\\ \bl&\bl&&&}}} \arrow{se,b}{M\cup M_\per}
\node[2]{{\tableau[pby]{\\ & \\ & \\ & \\ \bl & & \\ \bl&\bl&&& \\ \bl&\bl&\bl&&&\fl}}} \arrow{sw,b}{m\cup M_\per} \\
\node[2]{{\tableau[pby]{\\ &  \\ & \\ \bl& & \fl \\ \bl&\bl &&\fl \\
\bl&\bl& & & \\ \bl&\bl &\bl& & &\fl }}}
\end{diagram}
\end{equation*}
\end{example}

\begin{lemma} \label{L:perfection}
Suppose $(m,M)$ are interfering row moves with the top cell of $m$ above that of $M$.
\begin{enumerate}
\item Suppose a lower perfection $\rho$ exists. Then it is unique: $m_\per$ is such that
$m\cup m_\per$ is the unique move from $\la\cup M$ obtained by extending
each of the strings of $m$ below by $\ell'$ cells, and also $M\cup m_\per$
is the unique move from $\la\cup m$ obtained by adding $r$ more translates
to the right of the strings of $M$.
\item Suppose an upper perfection $\rho$ exists. Then it is unique: $M_\per$ is such that
$M\cup M_\per$ is the unique move from $\la\cup m$ obtained by extending
each of the strings of $M$ above by $\ell$ cells, and also
$m\cup M_\per$ is the unique move from $\la\cup M$ given by adding
$r'$ more translates to the left of the strings of $m$.
\end{enumerate}
\end{lemma}

\subsection{Row elementary equivalence}

\begin{definition}
\label{D:rowrowcommute} A \defit{row elementary equivalence} is a
relation of the form \eqref{E:elemequiv}
arising from two row moves $m$ and $M$ from some $\la\in\Ksh$, which
has one of the following forms:
\begin{enumerate}
\item \label{it:rowequivnoninterf}
$m$ and $M$ do not intersect and $(m,M)$ is non-interfering. Then $\tm= m$ and $\tM =M$.
\item
\label{it:rowequivinterf} $(m,M)$ is interfering (and say the top
cell of $m$ is above the top cell of $M$) %
and $(m,M)$ is lower (resp. upper) perfectible by adding the set of
cells $x=m_\per$ (resp. $x=M_\per$). Then $\tm= m \cup x$ and $\tM=
M \cup x$.
\item \label{it:rowequivmatch}
$m$ and $M$ intersect and are matched above (resp. below).
In this case $\tilde m= m \setminus (m \cap M)$ and $\tilde M= M
\setminus (m \cap M)$.
\item \label{it:rowequivabovebelow}
$m$ and $M$ intersect and $m$ continues above and below $M$. In this case
$\tilde m= \uparrow_{m\cap M}(m)$ and $\tilde M= \uparrow_{m\cap
M}(M)$.
\item  \label{it:rowequivempty} $M=\emptyset$ and there is a row move $m_\per$ from $m*\la$
such that $m\cup m_\per$ is a row move from $\la$.
Then $\tM=m_\per$ and $\tm=m\cup m_\per$.
\end{enumerate}
In cases \eqref{it:rowequivabovebelow} and \eqref{it:rowequivempty}
the roles of $m$ and $M$ may be exchanged. In case
\eqref{it:rowequivinterf}, $(m,M)$ may be both lower and upper
perfectible, in which case both perfections yield row elementary
equivalences. In case \eqref{it:rowequivempty}, $m_{\per}$ can
continue the strings of $m$ above or below.  This case can thus be
considered as a degeneration of case~\eqref{it:rowequivinterf}.
\end{definition}
\begin{remark}
If the pair $(m,M)$ satisfies a row elementary equivalence
then $m$ and $M$ are reasonable.
\end{remark}

\begin{example} \label{X:rowrow}
We give examples of row elementary equivalences.
For case \eqref{it:rowequivinterf} see Example \ref{X:twocompletions}.
On the left we give for $k=4$ a case \eqref{it:rowequivmatch} example
with moves matched above,
and on the right, for $k=5$ an example of case \eqref{it:rowequivabovebelow}.
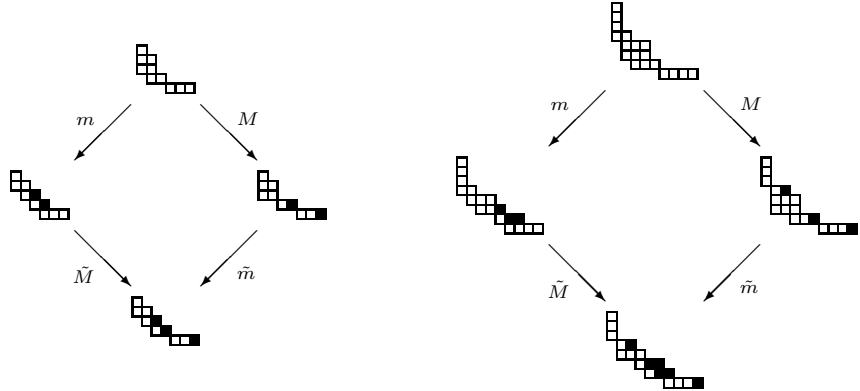
\begin{equation*}
\begin{diagram}
\node[2]{{\tableau[pby]{\\ & \\ &  \\ \bl & & \\ \bl &\bl&\bl& & & }}} \arrow{sw,t}{m} \arrow{se,t}{M} \\
\node{{\tableau[pby]{\\ & \\ \bl&&\fl \\ \bl&\bl&&\fl \\
\bl&\bl&\bl&&& }}} \arrow{se,b}{\tilde M}
\node[2]{{\tableau[pby]{\\ & \\ &  \\ \bl&\bl& &\fl \\ \bl&\bl&\bl&\bl&&&\fl}}} \arrow{sw,b}{\tilde m} \\
\node[2]{{\tableau[pby]{ \\ & \\ \bl&&\fl \\ \bl&\bl&&\fl \\
\bl&\bl&\bl&\bl&&&\fl }}}
\end{diagram}
\qquad\qquad
\begin{diagram}
\node[2]{{\tableau[pby]{\\ \\ \\ & \\ \bl & && \\ \bl & & & \\ \bl&\bl&&&\\ \bl&\bl&\bl&\bl&\bl&&&&\\}}} \arrow{sw,t}{m} \arrow{se,t}{M} \\
\node{{\tableau[pby]{\\ \\ \\ & \\ \bl&&&\\ \bl&\bl &&&\fl\\
\bl&\bl&\bl&\bl&&\fl&\fl\\ \bl&\bl&\bl&\bl&\bl&&&&\\}}}
\arrow{se,b}{\tilde M} \node[2]{{\tableau[pby]{\\ \\ \\ \bl &&\fl \\
\bl &&& \\ \bl & & & \\ \bl&\bl&\bl&&&\fl\\
\bl&\bl&\bl&\bl&\bl&\bl&&&&\fl \\}}} \arrow{sw,b}{\tilde m} \\
\node[2]{{\tableau[pby]{\\ \\ \\ \bl &&\fl \\ \bl &&& \\ \bl
&\bl&\bl &&\fl&\fl \\ \bl&\bl&\bl&\bl&&\fl&\fl \\
\bl&\bl&\bl&\bl&\bl&\bl&&&&\fl \\}}}
\end{diagram}
\end{equation*}

\end{example}

\begin{proposition}\label{P:rel} If $(m,M)$ defines a row elementary equivalence,
then the prescribed sets of cells $\tm$ and $\tM$
define a diamond equivalence.
\end{proposition}

\subsection{Column elementary equivalence}
\label{SS:column_elementary_equivalence}

\begin{definition} \label{D:colcolcommute}
There is an obvious transpose analogue of row elementary equivalence
which we shall call \defit{column elementary equivalence}.
\end{definition}

Since \eqref{E:charge} is obviously satisfied in the case of
a column elementary equivalence, the transpose analogue of Proposition
\ref{P:rel} holds.

\subsection{Diamond equivalences are generated by elementary equivalences}

\begin{lemma} \label{L:mixeddiamond}
Any diamond equivalence $\tilde M m \equiv \tilde m M$ in which
$m$ is a row move and $M$ a column move from some $\la\in\Ksh$,
is a mixed elementary equivalence.
\end{lemma}

\begin{lemma} \label{L:rowdiamond}
 Let $m$ and $M$ be row (resp. column) moves such that $\tilde M m\equiv\tilde m  M$
  is a diamond equivalence.  Then the relation $\tilde M m \equiv \tilde m  M$ can be
generated by row (resp. column) elementary equivalences.
\end{lemma}

We have immediately:

\begin{proposition} \label{P:diamond} The equivalence relations generated respectively
by diamond
equivalences and by elementary equivalences are identical.
\end{proposition}
\begin{proof}
Lemma~\ref{L:mixeddiamond} and Lemma~\ref{L:rowdiamond} imply that
diamond equivalences are generated by elementary
  equivalences.  Since elementary equivalences are diamond equivalences
  by Propositions \ref{P:relmix} and \ref{P:rel},
the proposition follows.
\end{proof}

\subsection{Proving properties of mixed equivalence}
\label{SS:introwcolmoves}
For the rest of this subsection we assume that $m$ and $M$ are respectively
row and column moves from $\la\in\Ksh$.

\begin{property} \label{L:notfinishsame}  Strings of $m$ and $M$ cannot
be matched above or below.
\end{property}
\begin{proof}
By Property~\ref{P:movestrip}(3) the cells of the strings of $m$ and $M$
that meet are $\lambda$-addable.
The lemma then easily from considering the diagrams of $m$ and $M$.
\end{proof}

\begin{property} \label{L:uniqstrint}
Any string of $m$ (resp. $M$) meets at most one string of $M$ (resp. $m$).
\end{property}
\begin{proof}
Suppose there is some string $s=\{a_1,a_2,\ldots\}\subset m$ where $a_i\in
t$ and $a_j\in \bar t$ for distinct column-type strings $t,\bar t\in M$.  Let $i$ and $j$
be such that $j-i$ is minimum where $i<j$.

We first show that $a_i$ is not the bottom cell of $t$.  If this
were the case, the distance between the bottom cell of
$t$ and the bottom cell of $\bar t$ would be larger than $k-1$, which
would contradict Lemma~\ref{L:distancestrings}. Therefore
$a_i$ is not the lowest cell of $t$.

The cells $a_i$ and $a_j$ are $\lambda$-addable by
Property~\ref{P:movestrip}(3) and so is $s$ by Lemma~\ref{L:addable}.
Thus by Remark~\ref{R:contig}, the cell of $t$ contiguous to and below $a_i$
is $a_{i+1}$.  If $a_{i+1}\neq a_j$ we have a contradiction
to the choice of $i$ and $j$.   If $a_{i+1}=a_{j}$ we have the contradiction
that $t$ and $\bar t$ intersect.

Taking transposes, every string of $M$ meets at most one string of $m$.
\end{proof}

\begin{lemma} \label{L:stringint}
Suppose $s=\{a_1,\dotsc,a_\ell\}$ and $t=\{b_1,\dotsc,b_{\ell'}\}$
are strings in $m$ and $M$ respectively such that
$s\cap t\ne\emptyset$.
Then $s\cap t$ is a $\la$-addable string
and there are intervals $A\subset [1,\ell]$ and $B\subset [1,\ell']$ such that
$s\cap t = \{a_j\mid j\in A\} = \{b_j\mid j\in B\}$. Moreover, either $\min(A)=1$
or $\min(B)=1$, and also either $\max(A)=\ell$ or $\max(B)=\ell'$.
\end{lemma}
\begin{proof} Let $x\in s\cap t$. By Property~\ref{P:movestrip} and
Lemma \ref{L:addable} all the cells of $s$ and $t$ are $\la$-addable.
Suppose both $s$ and $t$ contain cells below (resp. above) $x$.
Since cells in strings satisfy a contiguity property,
there are $\la$-addable cells $z\in s$ and $z'\in t$ such that
$z$ and $z'$ are contiguous with and below (resp. above) $x$.
By Remark \ref{R:contig}, $z=z'$. The Lemma follows.
\end{proof}

Call a row-type (resp. column-type) string of $m$ (resp. $M$) \defit{primary}
if it consists of $\la$-addable corners. Write $\Prim(m)$ for the set of
primary strings of $m$; the dependence on $\la$ is suppressed.
The strings of $m$ (resp. $M$) are totally ordered, and this induces an order on the primary strings.
For $s\in\Prim(m)$ with $s\ne\max(\Prim(m))$ (resp. $s\ne \min(\Prim(m))$)
we write $\sucs(s)$ (resp. $\pred(s)$) for the cover (resp. cocover)
of $s$ in $\Prim(m)$.

\begin{remark} \label{R:intprim} By Lemma \ref{L:stringint},
if $s$ is a string in $m$
and $t$ a string in $M$ such that $s\cap t\ne\emptyset$ then
$s\in\Prim(m)$ and $t\in\Prim(M)$.
\end{remark}

\begin{lemma}
\label{L:nextint} Suppose $s$ is a string in $m$ and $t$ a string in $M$
such that $s\cap t\ne\emptyset$.
\begin{enumerate}
\item If $s$ continues below (resp. above) $t$ and $s\ne\min(\Prim(m))$
(resp. $s\ne \max(\Prim(m))$)
then there is a string $t'\in\Prim(M)$ such that $t'>t$ (resp. $t'<t$),
$\pred(s) \cap t' \ne\emptyset$ (resp. $\sucs(s)\cap t'\ne\emptyset$),
and $\pred(s)$ (resp. $\sucs(s)$) continues below (resp. above) $t'$.
\item If $t$ continues above (resp. below) $s$ and $t\ne\min(\Prim(M))$
(resp. $t\ne\max(\Prim(M))$), then there is a string $s'\in\Prim(m)$ such that
$s'>s$ (resp. $s'<s$), $s'\cap \pred(t)\ne\emptyset$ (resp. $s'\cap\sucs(t)\ne\emptyset$),
and $\pred(t)$ (resp. $\sucs(t)$) continues above (resp. below) $s'$.
\end{enumerate}
\end{lemma}
\begin{proof} We prove (1) as (2) is the transpose analogue.
Suppose $s$ continues below $t$ and $s\ne\min(\Prim(m))$.
Let $b$ be the bottom cell in $t$; it is also the bottom cell of $s\cap t$.
By hypothesis the string $s$ has a $\la$-addable cell $b'\not\in t$,
contiguous to and below $b$.  $M$ shortens the row of $b'$ since
$(\row(b'),\col(b))$ is removed by $M$ and $b'$ is not added by $M$ by
Property~\ref{L:uniqstrint}.
Let $c$ and $c'$ be the translates in $\pred(s)$ of the cells $b$ and $b'$ in $s$.
Note that $\row(b')<\row(c')$ since $b'$ and $c'$ are $\la$-addable
and $c' \in \pred(s)$.
Furthermore, by Corollary~\ref{C:samesize},
$\rs(\la)_{\row(c')}=\rs(\la)_{\row(b')}$.
Now, from a previous comment $\rs(M * \lambda)_{\row(b')}
= \rs(\la)_{\row(b')}-1$.  In order for  $M*\la$ to belong to $\Ksh$,
$M$ must also remove the cell $(\row(c'),\col(c))$ without adding the
cell $c'$.  Therefore there is a string $t'>t$ such that
$\pred(s)\cap t' \ne \emptyset$ and such that $\pred(s)$ continues below $t'$.

Suppose $s$ continues above $t$ and $s\ne\max(\Prim(m))$.
Let $b$ be the highest cell in $t$, $b'$ the cell of $s$ below and contiguous with
$b$. Let $c$ and $c'$ be the translates in $\sucs(s)$ of
$b$ and $b'$ in $s$. One may show that $M$ adds a cell to the row of $b$ and removes
none. $M$ must do the same to the row of $c$ since $M*\la\in\Ksh$. The rest of the argument
is similar to the previous case.
\end{proof}

%

\begin{lemma} \label{L:continuesexistscol} Lemma \ref{L:moves_relative_position} holds
for a row move and a column move.
\end{lemma}
\begin{proof}
Suppose $s$ and $s'$ are strings of $m$ and $t$ and $t'$
are strings of $M$
such that $s\cap t\ne\emptyset$, $s'\cap t'\ne\emptyset$, $s$ continues below $t$,
and $s'$ does not continue below $t'$. By Lemma \ref{L:nextint},
$s'>s$. Let $b,b',c,c'$ be the bottom cells of $s,s',t,t'$.
We have that $d(c)<d(b)<d(b') \leq d(c')$.  Since the distance between
$c$ and $b$ is more than $k-1$,
the distance between $t$ and $t'$ is also more than $k-1$.  But this violates
Lemma~\ref{L:distancestrings}.
\end{proof}

\begin{proof}[Proof of Proposition \ref{P:relmix}]

Let $m=\{s_1,s_2,\dotsc,s_r\}$ and $s_1=\{a_1,\dotsc,a_\ell\}$.

(1) The disjointness of $m$ and $M$ implies the commutation of
\eqref{E:commudiagram}, and \eqref{E:charge} holds trivially in this
case. So it suffices to show that $m$ is a row move from $M*\la$;
showing that $M$ is a column move from $m*\la$ is similar.

Since $M\cap m=\emptyset$, $s_1$ is a $(M*\la)$-addable string. The
diagram of the string $s_1$ remains the same in passing from $\la$
to $M*\la$; the only place it could change is in the row of $a_1$
and the column of $a_\ell$, and this could only occur if $a_1$ or
$a_\ell$ were contiguous with a cell of $M$, which is false by
assumption. So $s_1$ is a row-type $(M*\la)$-addable string.
The argument for the other strings of $m$ is similar.

Since $M$ is a column move from $\la$,
$\cs(\la)=\cs(M*\la)$. But then Property~\ref{P:cond5}
implies that $(M*\la)\cup m\in\Ksh$. This proves that $m$ is
a row move from $M*\la$ as required.

(2) We prove case (a) as (b) is similar.
By definition of $\tilde M$, $\tilde M$ contains the same number of strings
as $M$ and the strings of $\tilde M$ are of the same length as those of $M$.
Thus \eqref{E:charge} is satisfied.

By Lemma \ref{L:nextint}, all primary strings of $m$ meet $M$.
In particular,
since the first string $s_1$ of $m$ is always primary, it meets $M$.
The string $s_1$
meets a single string $\hs$ in $M$ and
$s_1\cap\hs =\hs = \{a_p,a_{p+1},\dotsc,a_n\}$ for some $1<p\le n<\ell$
by Lemma \ref{L:stringint} since $s_1$ continues above and below $\hs$.

We now show that $t=\rightarrow_M(s_1)$ is a $(M*\la)$-addable string.
By Property~\ref{P:movestrip} $M$ is a vertical strip and
$$t = \{a_1,a_2,\ldots,a_p^\dagger,a_{p+1}^\dagger,\ldots,a_{n}^\dagger,
a_{n+1},\ldots,a_\ell\}$$ where $a^\dagger$ denotes the cell
right-adjacent to the cell $a$.

Since $a_p$ is the top cell of the column-type $\la$-addable
string $\hs$ and there is a $\la$-addable corner $a_{p-1}$
contiguous to and above $a_p$, by Definition~\ref{D:stringtypes}
we have $\diag(a_p)-\diag(a_{p-1})=k$ and
$\diag(a_p^\dagger)-\diag(a_{p-1})=k+1$. Similarly,
since $a_n$ is the bottom cell of $\hs$,
we have $\diag(a_{n+1}-\diag(a_n)=k+1$ and $\diag(a_{n+1})-\diag(a_n^\dagger)=k$.
Thus the cells of $t$ satisfy the contiguity conditions for a string.

Let $\mu=M*\la$.
For $i<p$ and $i>n$, $a_i$ is $\mu$-addable since
it is $\la$-addable and $a_i\not\in M$ by Property~\ref{L:uniqstrint}.

Let $c$ be the column of $a_{p-1}$.  Since $a_p$ is the top cell of a string in
$M$, there is no cell removed in the row of $a_p$ when going from $\lambda$ to
$\mu$, and thus $\Bot_c(\partial \mu)$ still lies in the row of
$a_p$. By Remark~\ref{R:rowshift} there is a $\mu$-addable corner in the
row
of $a_p$ and it corresponds to $a_p^{\dagger}$.

Now observe that if
$a_{p+1}^\dagger$ is not a $\mu$-addable corner, then there is a
$\mu$-addable corner $e$ below it by Lemma~\ref{L:addablecornerbelow} (since
$a_{p+1}^\dagger$ is a distance $k$ or $k+1$ from the $\mu$-addable corner
$a_p^\dagger$).
Since $M$ is a column move,
there is a $\mu$-removable string with cells $c_p$ and $c_{p+1}$
in the columns of $a_p$ and $a_{p+1}$ which is a translate of string $\hat s$ (the
two strings may coincide).
The distance between $a_p^{\dagger}$ and $e$
is thus larger (by exactly one unit) than the distance between $c_p$ and $c_{p+1}$. Furthermore,
$a_p^{\dagger}$ and $e$ lie in columns immediately to the right
of those of $c_p$ and $c_{p+1}$ respectively.
We then have the contradiction that
the removable string containing $c_p$ and $c_{p+1}$ and the addable string containing
$a_{p}^\dagger$ and $e$
violate Lemma~\ref{L:distancecontiguouscells}.
Therefore $a_{p+1}^\dagger$ is a $\mu$-addable corner
and repeating the previous argument again and again
we get that $a_i^{\dagger}$ is $\mu$-addable for any $p<i\leq
n$.

Therefore $t$ is a $\mu$-addable string.
It is of row-type since the top and bottom of its string diagram
are unaffected by adding $M$ to $\la$ and
coincide with the top and bottom of the diagram of the $\la$-addable row-type string $s$.

Suppose there are $q$ strings of $m$ in the rows of $s_1$, and let
 $t_j=\rightarrow_M(s_j)$ for
$1\le j\le q$.
It follows from the results of
Subsection \ref{SS:introwcolmoves} and the
translation property of strings in row moves, that
$t_j$ is a translate of $t_1$: the top and bottom of $t_j$ agree with
those of $s_j$, and $\rightarrow_M$ right-shifts
the $p$-th through $n$-th cells in $s_j$ to obtain $t_j$, which are the
same positions within the string $s_1$ that are right-shifted
to obtain $t_1$. In particular $t_j$ is a row-type string.

We claim that $t_j$ is
$(M\cup t_1\cup\dotsm\cup t_{j-1})*\lambda$-addable for $1\le j \le q$.
It holds for $j=1$.  For the general case, since $m$ is a
horizontal strip, we have that $\lambda_{\row(a_{p-1})-1}-\lambda_{\row(a_{p-1})} \geq q$.
In $\mu$
we still have $\mu_{\row(a_{p-1})-1}-\mu_{\row(a_{p-1})} \geq q$ since there is
no cell of $M$ in $\row(a_{p})$. By Lemma~\ref{L:comparelengths}
applied to the string $t$ we have that $\mu_{\row(a_{i})-1}-\mu_{\row(a_{i})} \geq q$
for any $p \leq i \leq n$.  This immediately implies that $t_j$ is
$(M\cup t_1\cup\dotsm\cup t_{j-1})*\lambda$-addable for $1\le j \le q$.

The same approach shows that $t_i'=\rightarrow_M(s_i)$ is a $\mu$-addable
string for any other primary string $s_i$ of $m$, and that the strings lying
in the rows of $s_i$ can be right-shifted as prescribed.  Moreover, arguing as in the
proof of Lemma \ref{L:continuesexistscol}, one may show
that $\hs_i=s_i \cap t_i$ consists of the $p$-th through $n$-th cells of $s_i$
(using the same $p$ and $n$ as for $s_1$). It follows that all the strings $\ts_i$ are translates
of each other.

$\tm * M * \la$ is a $k$-shape, because
$M*\la$ is, and because the condition in Property~\ref{P:cond5}
is unchanged in passing from the move $\la\to m*\la$, to the move
$M*\la\to \tm*M*\la$. Therefore $\tm$ is a row move from $M*\la$
with first strings $\ts_1,\dots,\ts_r$.

We must show that $\tM$ is a column move from $m*\la$.
It is a vertical strip, being the difference of
partitions $m*\la$ and $\la\cup M\cup \tm$,
and having at most one cell per row by definition.

It was shown previously that for any string $s'$ of $M$ that meets $m$,
$s'$ is contained in the string of $m$ that it meets.
Note that since strings in a move are translates of each other, we have that if the primary
string
$t=\{a_1,\dots,a_{\ell} \}$ of $m$ is such that there are $n$ cells of
$m$ in the row of $a_1$, then there are also $n$ cells of $m$ in the row
of $a_i$ for all $i$.
It follows that under $\rightarrow_m$,
every string in $M$ is translated directly to the right by some
number of cells (possibly zero). Therefore $\tM$ is the disjoint union of strings
that are translates of each other and which start in consecutive rows.
Since $\tM$ is an $m*\la$-addable vertical strip we deduce that
it is a column move from $m*\la$.
\end{proof}

\subsection{Proving properties of row equivalence}
\label{SS:introwrowmoves}
We state the analogues of results in Subsection \ref{SS:introwcolmoves}
for intersections of row moves $m$ and $M$ from $\la\in\Ksh$.

\begin{property} \label{L:uniqstrintrow}
Every string of $m$ meets at most one string of $M$.
\end{property}

\begin{lemma} \label{L:stringintrow}
Suppose $s=\{a_1,\dotsc,a_\ell\}$ and $t=\{b_1,\dotsc,b_{\ell'}\}$
are strings in $m$ and $M$ respectively such that $s\cap t\ne\emptyset$.
Then $s\cap t$ is a string and there are intervals $A\subset [1,\ell]$ and $B\subset [1,\ell']$ such that
$s\cap t = \{a_j\mid j\in A\} = \{b_j\mid j\in B\}$. Moreover, either $\min(A)=1$ or $\min(B)=1$,
and also either $\max(A)=\ell$ or $\max(B)=\ell'$.
\end{lemma}

\begin{lemma}
\label{L:nextintrow} Suppose $m=s_1\cup s_2\cup\dotsm\cup s_p$ and
$M=t_1\cup t_2\cup\dotsm\cup t_q$ are row moves on $\la\in\Ksh$ with given
string decomposition such that $m\cap M\ne \emptyset$.
\begin{enumerate}
\item The leftmost cell of $m\cap M$ is contained in either $s_1$ or $t_1$.
\item The rightmost cell of $m\cap M$ is contained in either $s_p$ or $t_q$.
\end{enumerate}
\end{lemma}

\begin{lemma} \label{L:continuesexistsrow} Lemma \ref{L:moves_relative_position} holds
for $m$ and $M$ both row moves.
\end{lemma}
%

\begin{lem}\label{L:introw}
Suppose $m = s_1 \cup s_2 \cup \cdots \cup s_p$ and $M = t_1 \cup
t_2 \cup \cdots \cup t_q$ are intersecting moves from $\la\in\Ksh$.
\begin{enumerate}
\item
Suppose that $m$ continues above $M$ but the two are matched below.
Then $p \le q$ and $s_i$ contains $t_i$ and continues above it for $1\le i\le p$.
\item
Suppose that $m$ continues below $M$ but the two are matched above.
Then $p \le q$ and $s_{p-i}$ contains $t_{q-i}$
and continues below it for $0\le i<p$.
\end{enumerate}
\end{lem}
\begin{proof} We prove (1) as (2) is similar. The hypotheses imply that
for some $i$ and $j$ we have $c_{s_i,d}=c_{t_j,d}$. It follows from
Property~\ref{C:rowmovecolsize} that $c_{s_1,d}=c_{t_1,d}$, that is, $s_1$ and $t_1$
intersect. Applying Property~\ref{C:rowmovecolsize} to the upper part of $M$
we conclude that $p\le q$. We have that
$s_i$ meets $t_i$ for $1\le i\le p$ since it is true for $i=1$
and strings in a move are translates. Since $m$ continues above $M$ and they are matched below,
$s_i$ contains $t_i$ and continues above it.
\end{proof}

\begin{proof}[Proof of Lemma \ref{L:perfection}]
We prove (1) as (2) is similar. Let $m_\per$ give rise to the lower
perfection $\la\cup m\cup M\cup m_\per\in\Ksh$. Since $m\cup m_\per$
is a row move from $M*\la$ of rank $r$, it follows that $m_\per$,
viewed as $\la\cup m\cup M$-addable, must negatively modify (by
$-1$) precisely the columns $c_{s_1,u}$ through $c_{s_r,u}$. So
$M\cup m_\per$ is a row move from $m*\la$ which negatively modifies
the $r+r'$ consecutive columns
$c_{t_1,u},\dotsc,c_{t_{r'},u},c_{s_1,d},\dotsc,c_{s_r,d}$.
Therefore $m_\per$ is specified by adjoining to $\la\cup m\cup M$,
translates of $t_1$ in the $r$ columns just after $c_{t_{r'},u}$.
The other claims are clear.
\end{proof}

\begin{proof}[Proof of Proposition \ref{P:rel}]

 Cases \eqref{it:rowequivnoninterf} and \eqref{it:rowequivabovebelow}
are similar to Cases (1) and (2) of mixed equivalence.
Case \eqref{it:rowequivempty} is trivial. Case \eqref{it:rowequivinterf} holds by definition.
So consider Case \eqref{it:rowequivmatch}. We suppose
that $m$ and $M$ are row moves on $\la$ that are matched below, as the ``matched above" case
is similar. If $m$ and $M$ are also matched above then it follows that $m=M$:
intersecting strings must coincide, and Property~\ref{C:rowmovecolsize}
implies that the two moves must modify the same columns. So we may assume that $m$ continues above $M$.
Using the notation of Lemma \ref{L:introw}, we see that $\tM$ decomposes into
strings $t_{p+1},\dotsc,t_q$. These strings neither intersect nor have any cells contiguous with
any of the other strings in $m$ or $M$.
It follows that $\tM$ is a row move from $m*\la$ since $\cs(\tilde M\cup m \cup
\lambda)=\cs(M\cup m \cup\lambda)$ is a partition ($m$ and $M$ are not
interfering).
As a
set of cells, $\tm$ decomposes into strings $u_i:=s_i\setminus t_i$ for $1\le
i\le p$ that are translates of each other.
The top of the diagram of the string $u_1$ coincides with that of $s_1$.
Consider the column $c_{u_1,d}$ in the diagram of $u_1$
as a $M*\la$-addable string. $u_1$ does not remove a cell from this column
since the first string $t_1$ of $M$ already removed such a cell in passing from $\bdy\la$ to $\bdy M*\la$.
Therefore $u_1$ is a row-type $M*\la$-addable string. Similarly it follows that
$\tm$ is a row move from $M*\la$ with strings $u_i$.
\end{proof}

\subsection{Proofs of Lemma~\ref{L:mixeddiamond} and
  Lemma~\ref{L:rowdiamond}}\label{SS:diamond}

\begin{proof}[Proof of Lemma~\ref{L:mixeddiamond}]
By Definition~\ref{D:rowcolcommute} we need to show that if $m$ and
$M$ do not intersect but a cell of $m$ is contiguous to a cell of
$M$ or if $m$ and $M$ intersect and are not reasonable, then $m$ and
$M$ do not define a diamond equivalence.

Suppose there is a diamond equivalence $\tilde m M \equiv \tilde M
m$. By definition we must have $\Delta_{\rs} (\tilde
M)=\Delta_{\rs}(M)$ ($m$ and $\tm$ are row moves and thus do not
change row shapes). As a consequence, $M$ and $\tilde M$ must have
the same rank, and similarly for $m$ and $\tilde m$. The charge
conservation of a diamond equivalence also implies that $M$ and
$\tilde M$ have the same length.

Consider the case where $m$ and $M$ do not intersect but a cell of
$m$ is contiguous to a cell of $M$. Suppose $m$ is above $M$. Then
the bottom cell $a$ of a given string $s$ of $m$ is contiguous to
the top cell $b$ of a given string $t$ of $M$.  Furthermore,
$d(b)-d(a)=k$, for otherwise $s$ and $t$ would not be of row and
column types respectively. Since $\Delta_{\rs}(t)$ has a +1 in
$\row(b)$, there must be a string $\tilde t$ of $\tilde M$ that ends
in $\row(b)$ in order for $\Delta_{\rs}(\tilde M)$ to have a +1 in
$\row(b)$.  But then in $m*\lambda$ the hook-length of the cell in
position $(\row(b),\col(a))$ is $k$ which gives the contradiction
that $\tilde t$ is not a column-type string. Otherwise $m$ is below
$M$, the top cell $a$ of some string $s\subset m$ is contiguous with
the bottom cell $b$ of some string $t\subset M$ with
$d(a)-d(b)=k+1$. Let $r=\row(a)$ and $c=\col(b)$. Then
$h_\la(r,c)=k$. Since $\Delta_\cs(m)=\Delta_\cs(\tm)$, $\tm$ must
remove $(r',c)=\Bot_c(\bdy(M*\la))$ where $r'>r$. Since $M$ is a
column move, $\cs(M*\la)_c=\cs(\la)_c$. Since $m$ contains the cell
$a=(r,\la_r+1)$, $\tm$ must contain the cell $(r',\la_r+1)$ in order
to remove $(r',c)$. But then $M$ contains the cell $a$,
contradicting the disjointness of $m$ and $M$.

Now consider the case where $m$ and $M$ intersect but are not reasonable.
 Suppose there is a
string $s$ of $m$ that meets a string $t$ of $M$, with $s$ continuing below
$t$ but not above it. By Property~\ref{L:notfinishsame},
we know that $t$ finishes above $s$.  Let $t=\{a_1,\dots,a_\ell \}$
and $s$ be such that
$s\cap t= \{a_{i},\dots,a_\ell\}$, and let
$b$ be the cell of $s$  contiguous to and below $a_\ell$
(it  exists by our hypotheses).
Since $M$ is a column move, $\Delta_{\rs}(M)$ has a $-1$
 in $\row(b)$.  Thus
$\Delta_{\rs}(\tilde M)$ must also have a $-1$ in $\row(b)$.  This implies that
there is a string $t'=\{a_1',\dots,a_\ell' \}$ of $\tilde M$ (recall that
$M$ and $\tilde M$ have the same length)
such that
$\Delta_{\rs}(t')$ has a $-1$ in $\row(b)$.  By definition of column moves,
and since $\Delta_{\rs}(M)=\Delta_{\rs}(\tilde M)$
(which implies that $M$ and $\tilde M$ have the same rank), we have that
the upper cells of $t$ and $t'$ must coincide. That is,
$\{a_1',\dots,a_{i-1}' \}=\{a_1,\dots,a_{i-1} \}$.  Note that since $m$ is
a horizontal strip and $\tilde M$ is a vertical strip,
the cells outside $m*\lambda$ catty-corner to $\{a_{i},\dots,a_\ell\}$
are not in $\tilde M m *\lambda$.  Now, the distance between $a_{i-1}$ and
$a_i$ is $k+1$ ($a_i$ is the top cell of a row move).  Thus
from the previous comment and contiguity we have that
$d(a_i')< d(a_i) < d(a_{i+1}') <  d(a_{i+1}) <  \cdots < d(a_\ell')< d(a_\ell)$.
But then we have the contradiction that $a_{\ell}'$ cannot negatively modify
$\row(b)$  since in this row
there is no cell of $\partial (m * \lambda)$ weakly to the
left of $\col(a_\ell)$.
The case where there is a string $s$ of $m$ that
meets a string $t$ of $M$, with $s$ continuing above
$t$ but not below it is similar.
\end{proof}

\begin{proof}[Proof of Lemma~\ref{L:rowdiamond}]

All cases that could produce a diamond equivalence
where $m$ and $M$ do not intersect are covered  by
Definition~\ref{D:rowrowcommute}.  In case (1) there are no strings that could
be added at the same time to $m$ and $M$ to produce moves $\tilde m \neq m$  and
$\tilde M \neq M$.  In case (2), unicity is
guaranteed by Lemma~\ref{L:perfection}.

Suppose we have a diamond
equivalence  $\bar M m \equiv \bar m M$, where $m$ and $M$ are such as in case (3),
and suppose that and $m$ and $M$ are matched below with $m$ continuing above.
As mentioned in the proof of Proposition~\ref{P:rel},
$\tilde m$ decomposes into strings $u_i=s_i \setminus t_{i}$ for $1 \leq
i \leq p$ and $\tilde M$ decomposes into strings $t_{p+1},\dots,t_q$.
We now show that if $q>p$ then $\bar m = \tilde m$ and $\bar M=\tilde M$.
It is obvious that $\tilde m \subseteq \bar m$ and $\tilde M \subseteq \bar M$.
There are two possible options:  either $\bar M$ has more strings than $\tilde
M$ or its strings are extensions of those of $\tilde M$.  Since
$\bar m \setminus \tilde m = \bar M \setminus \tilde M$, in the first option
the extra strings must extend the $u_i$'s below, and in the second option the
extension must form strings to the right of those of $\tilde m$.
The former is impossible since the distance between the bottom cell of any
$u_i$ and the top cell of any of the new strings added is more than $k+1$.  The
latter case is impossible since no new strings can be added to the right of
$\tilde m$ to form a move by Property~\ref{C:rowmovecolsize}.  Thus the only
option is $p=q$.  In this case $\tilde M= \emptyset$, $\tilde m = m \setminus M$
and we have:
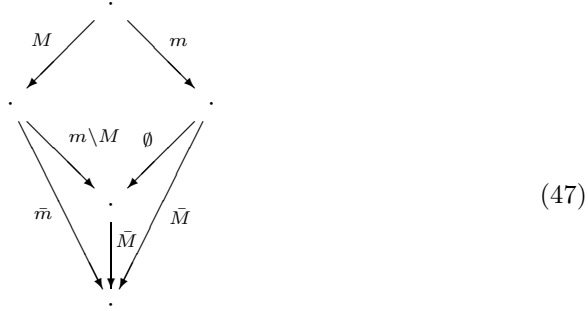
\begin{equation}
\begin{diagram} \label{E:commudiagram2}
\node[2]{\cdot} \arrow{sw,t}{M} \arrow{se,t}{m} \\ \node{\cdot}
\arrow{se,t} {m \setminus M }
\arrow{sse,b}
{\bar m} \node[2]{\cdot} \arrow{sw,t} {\emptyset} \arrow{ssw,b} {\bar M}\\
\node[2]{\cdot} \arrow{s,r,1}{ \!\bar M} \\
\node[2]{\cdot} \\
\end{diagram}
\end{equation}
New strings cannot be added to $m \setminus M$ to form a new move.
Strings to the right would violate Property~\ref{C:rowmovecolsize}.
And strings to the left need to be such that
in $M$ the columns $c_{t_1,u}$ and the one to its left are of the same size
(and thus $M$ could not have been a move).   So $m \setminus M$ can be
extended either below or above (not both since otherwise $\bar M$ could
not be a move).  In this case the triangle in the left of the diagram obeys
a relation of the form (3).  Since the other triangle is trivial (a case (5)
with $m_{\per} = \emptyset$), the diamond equivalence $\bar M m = \bar m M$
 is generated by the elementary ones. The case (3) where $m$ and $M$ are
 matched above is similar.

The only other cases that could produce a diamond equivalence
which are not covered by Definition~\ref{D:rowrowcommute} are those where
$m$ and $M$ are not reasonable, that is,
there are
strings $s$ and $t$ of $m$ and $M$ respectively such that $s\cap t \neq
\emptyset$,
$t \nsubseteq s$ and $s \nsubseteq t$.  Suppose that $t$ continues below $s$.
We show that if there are strings $t_{i},\dots, t_{i+j}$ of $M$ that do not intersect strings
of $m$ then there is no possible diamond equivalence  $\bar M m = \bar m M$.
The strings $t_{i},\dots, t_{i+j}$ need to be to the right of those that meet
strings of $m$ by Property~\ref{C:rowmovecolsize} applied to the positively
modified columns of $m$.  For the diamond equivalence to hold, we need a
$M_{\per}$ that extends the strings $t_{i},\dots, t_{i+j}$ above and that
add extra strings to the right of $m \setminus M$.  But this is impossible
by Property~\ref{C:rowmovecolsize}.  In a similar way, if there are strings of
$m$ that do not intersect strings of $M$ then there is no possible diamond equivalence.
Therefore, we are left with the case where the strings $s_1,\dots,s_p$ of $m$
and $t_1,\dots,t_p$ of $M$ each intersect with one another. In this case we
necessarily have $\bar m=m \setminus M$ and $\bar M = M \setminus m$.  But then
$\bar m M=\bar M m = \mathcal M$ is also a move and we have the situation.
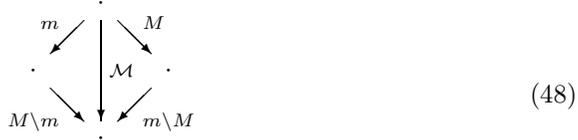
\begin{equation}
\begin{diagram} \label{E:commudiagram3}
\node[2]{\cdot} \arrow{sw,t}{m} \arrow{se,t}{M} \arrow[2]{s,r}{\mathcal M}
\\ \node{\cdot}
\arrow{se,b} {M \setminus m}
 \node[2]{\cdot} \arrow{sw,b} {m \setminus M} \\
\node[2]{\cdot}  \\
\end{diagram}
\end{equation}
In this case both triangles correspond to Case (3) of
Definition~\ref{D:rowrowcommute} and thus this diamond equivalence
is also generated by elementary ones.
\end{proof}

\section{Strips and tableaux for $k$-shapes}
\label{sec:strips}

In this section we introduce a notion of (horizontal) strip and
tableau for $k$-shapes.

\subsection{Strips for cores}
We recall from \cite{LM:cores,LLMS} the notion of weak strip and
weak tableau for cores. Let $\tS_{k+1}$ and $S_{k+1}$ be the affine
and finite symmetric groups and let $\tS^0_{k+1}$ denote the set of
minimal length coset representatives for $\tS_{k+1}/S_{k+1}$.
$\Core^{k+1}$ has a poset structure given by the left weak Bruhat
order transported across the bijection $\tS^0_{k+1}\to \Core^{k+1}$.
Explicitly, $\mu$ covers $\la$ in $\Core^{k+1}$ if $\mu/\la$ is a
nonempty maximal $\la$-addable string. Such a string is always of
cover-type and consists of all $\la$-addable cells whose diagonal
indices have a fixed residue (say $i$) mod $k+1$, and corresponds to
a length-increasing left multiplication by the simple reflection
$s_i\in\tS_{k+1}$. A
\defit{weak strip} in $\Core^{k+1}$ is an interval in the left weak
order whose corresponding skew shape is a horizontal strip; its
\defit{rank} is the height of this interval, which coincides with
the number of distinct residues mod $k+1$ of the diagonal indices of
the cells of the corresponding skew shape. For $\la\subset\mu$ in
$\Core^{k+1}$, a \defit{weak tableau} $T$ of shape $\mu/\la$ is a
chain
$$\la=\la^{(0)}\subset\la^{(1)}\subset\la^{(2)}\subset\dotsm\subset\la^{(N)}=\la$$
in $\Core^{k+1}$ where each interval $\la^{(i-1)}\subset\la^{(i)}$
is a weak strip.  The weight of a weak tableau $T$ is the sequence
of nonnegative integers $\wt(T)$ whose $i$-th member $\wt(T)_i$ is
the rank of $\la^{(i)}/\la^{(i-1)}$. Let $\Weak^{k+1}_{\mu/\la}$ be
the set of weak tableaux of $k+1$-cores of shape $\mu/\la$. The
weight generating function of $\Weak^{k+1}_{\mu/\la}$ is denoted by
$\WS^{(k+1)}_{\mu/\la}[X]$.

\subsection{Strips for $k$-shapes}

\begin{definition} \label{D:strip}
A strip of rank $r$ is a horizontal strip $\mu/\la$ of
$k$-shapes such that $\rs(\mu)/\rs(\la)$ is a horizontal strip
and $\cs(\mu)/\cs(\la)$ is a vertical strip, both of size $r$.
A cover is a strip of rank $1$.
\end{definition}

By the assumption that $\rs(\mu)/\rs(\la)$ is a horizontal strip,
distinct modified rows of $\mu/\la$ do not have the same length (in either $\rs(\la)$ or $\rs(\mu)$). The
modified columns however form {\it groups} which have the same
length in both $\cs(\mu)$ and $\cs(\la)$, where by definition two modified
columns $c, c'$ are in the same group if and only if $\cs(\la)_c =
\cs(\la)_{c'}$.

\begin{proposition} A strip $S=\mu/\lambda$ has rank at most $k$.
\end{proposition}
\begin{proof} Suppose $\mu/\la$ has rank greater than $k$,
that is, $|\cs(\mu)|-|\cs(\la)|>k$. Since $\cs(\mu)/\cs(\la)$ is a
horizontal strip, its cells occur in different columns. Therefore
the $k$-bounded partition $\cs(\mu)$ has more than $k$ columns, a
contradiction.
\end{proof}

\begin{remark} \label{R:small}
Although strips of rank $k$ exist, in the remainder of the
article {\bf we shall only admit strips of rank strictly smaller than $k$}.
For the purposes of this paper, this restriction is not so important:
in Theorem \ref{T:ktab_decomp}, mod the ideal $I_{k-1}$,
monomials with a multiple of $x_i^k$ are killed, and therefore we choose
to leave such tableaux out of the generating function by definition.
Remark~\ref{R:kstrips} will further elaborate on the effects of allowing
strips of rank $k$ in our construction.
\end{remark}

The notion of a strip generalizes that of weak strips for $k$-cores
and $k+1$-cores.

\begin{proposition} \label{P:tabstrip}
Suppose $\mu/\la$ is a strip such that $\mu,\la\in\Core^{k+1}$
(resp. $\mu,\la\in\Core^{k}$).
Then $\mu/\la$ is a weak strip in $\Core^{k+1}$ (resp. $\Core^{k}$).
\end{proposition}
\begin{proof}
It was established in \cite{LM:cores}
that if $\mu,\la\in\Core^{k+1}$,
$\rs(\lambda)/\rs(\mu)$ is a horizontal strip and
$\cs(\lambda)/\cs(\mu)$ is a vertical strip, then
$\mu/\lambda$ is a horizontal strip (Proposition 54 of \cite{LM:cores})
and the cells in $\mu/\lambda$ correspond to one letter in
a $k$-tableau (Theorem 71 of \cite{LM:cores}).   It was further established
in Lemma 9.1 of \cite{LLMS}
that $k$-tableaux and weak tableaux (sequences of
weak strips in $\Core^{k+1}$) are identical.
Therefore $\lambda/\mu$ is a weak strip in
$\Core^{k+1}$.  The same argument works for
$\mu,\lambda \in \Core^{k}$.
\end{proof}

\subsection{Maximal strips and tableaux}

\begin{definition} Let $\la\in\Ksh$ be fixed. Let $\Str_\la\subset\Ksh$ be the
induced subgraph of $\nu\in\Ksh$ such that $\nu/\la$ is a strip. Moves (paths) in $\Str_\la$
are called $\la$-augmentation moves (paths). By abuse of language, if $m$ is a move (path)
from $\mu$ to $\nu$ in $\Str_\la$ we shall say that $m$ is a $\la$-augmentation
move (path) from the strip $\mu/\la$ to the strip $\nu/\la$.
An augmentation of a strip
$S=\mu/\la$ is a strip reachable from $S$ via a $\la$-augmentation path.
A strip $S=\mu/\la$ is maximal if it is maximal in $\Str_\la$, that is,
if it admits no $\la$-augmentation move.
\end{definition}
Diagrammatically, an augmentation move is such that the following diagram commutes
$$
\begin{diagram}
\node{\la} \arrow{s,l}{S}
 \arrow{e,t}{\emptyset} \node{\la} \arrow{s,r}{\tilde S}\\
\node{\mu} \arrow{e,b}{m} \node{\nu}
\end{diagram}
$$
where $S$ and $\tilde S$ are strips and $\emptyset$ denotes the
empty move.

These definitions depend on a fixed $\la\in\Ksh$, which shall
usually be suppressed in the notation. Later we shall consider
augmentations of a given strip $S$, meaning $\la$-augmentations
where $S=\mu/\la$.

Clearly augmentation paths pass through strips of a constant rank.

\begin{definition} Let $\mu\in\Ksh$ be fixed. Let $\Str^\mu\subset\Ksh$ be the
induced subgraph of $\rho\in\Ksh$ such that $\mu/\rho$ is a strip.
A strip $S=\mu/\rho$ is reverse-maximal if $\rho$ is minimal in the graph
$\Str^\mu$ (see Definition~\ref{D:revaugmentation} for more details).
\end{definition}

Let $\mu\supset\la$ with $\la,\mu\in\Ksh$. A \defit{($k$-shape)
tableau} of shape $\mu/\la$ is a sequence $\la = \la^{(0)} \subset
\la^{(1)}\subset \dotsm\subset \la^{(N)}=\mu$ with
$\la^{(i)}\in\Ksh$, such that $\la^{(i)}/\la^{(i-1)}$ is a strip for
all $i$. It is \defit{maximal} (resp. \defit{reverse-maximal}) if
its strips are. The tableau has weight $\wt(T)=(a_1,a_2,\dotsc,a_N)$
where $a_i$ is the rank of the strip $\la^{(i)}/\la^{(i-1)}$ (which
we require to be strictly smaller than $k$ by Remark~\ref{R:small}).
Let
\begin{align}
\label{E:maxtabgf}
  \dk_{\mu/\la}^{(k-1)}[X] &= \sum_{T\in\T_{\mu/\la}^k} x^{\wt(T)} \\
\label{E:rmaxtabgf}
  \dkr_{\mu/\la}^{(k)}[X] &= \sum_{T\in \RT_{\mu/\la}^k} x^{\wt(T)}.
\end{align}
where $\T_{\mu/\la}^k$ (resp. $\RT_{\mu/\la}^k$) denotes
the set of maximal (resp. reverse-maximal) tableaux of shape $\mu/\la$ for $\la,\mu\in\Ksh^k$.

For $k$-cores (resp. $k+1$-cores), the maximal (resp.
reverse-maximal) tableau generating functions reduce to dual $k-1$
(resp. $k$) Schur functions.  The following result is a consequence
of Propositions~\ref{L:stripcore} and \ref{P:reversemaxcore}.
\begin{proposition} \label{L:maxcore} $~$
\begin{enumerate}
\item
For any $\la\in\Core^k$ and $\mu\in\Ksh^k$ such that $\mu/\la$ is a
maximal strip, $\mu\in\Core^k$. In particular, for $\la\in\Ksh^k$,
$\T_\la^k$ is empty unless $\la\in\Core^k$ and in that case
$\T_\la^k = \Weak^k_\la$ and the definition \eqref{E:maxtabgf} of $\dk_{\la}^{(k-1)}[X] $
agrees with the usual definition of the dual $(k-1)$-Schur function
(or affine Schur function or weak Schur function) $\WS_\la^{(k)}[X]$
via weak tableaux.
\item For any $\mu\in\Core^{k+1}$ and $\la\in\Ksh^k$ such that $\mu/\la$
is a reverse-maximal strip, $\la\in\Core^{k+1}$. In particular, for
$\la\in\Core^{k+1}$ and for every weight
$\beta=(\beta_1,\beta_2,\dotsc)$ with $\beta_i\le k-1$ for all $i$,
the set of reverse maximal tableaux of shape $\la$ and weight
$\beta$ is equal to the set of weak $k$-tableaux of shape $\la$ and
weight $\beta$ and thus $\dkr_\la^{(k)}[X] = \WS_\la^{(k)}[X] \mod
I_{k-1}$ where $I_{k-1}$ is defined in \eqref{E:symmquotient}.
\end{enumerate}
\end{proposition}

\begin{corollary}
For $\la \in \Core^k$, we have $\dk^{(k-1)}_\la[X] =
\dkr^{(k-1)}_\la[X]$.
\end{corollary}

Theorem \ref{T:ktab_decomp} is established as follows.

\begin{theorem} \label{T:pushpullbij} For all fixed $\mu,\nu\in\Ksh$,
there is a bijection
\begin{equation} \label{E:tabbranchpushout}
\begin{split}
  \bigsqcup_{\la\in\Ksh} \left(\RT_{\mu/\la} \times \Patheq^k(\la,\nu)\right) &\to \bigsqcup_{\rho\in\Ksh} \left(\T_{\rho/\nu} \times
  \Patheq^k(\mu,\rho)\right) \\
  (S,[\bp])&\mapsto (T,[\bq])
\end{split}
\end{equation}
such that
\begin{align*}
  \wt(S) &= \wt(T).
\end{align*}
\end{theorem}
The map $(S,[\bp])\to (T,[\bq])$ is called the \defit{pushout} and the inverse bijection
$(T,[\bq])\to(S,[\bp])$ is called the \defit{pullback}, in reminiscence of homological diagrams,
as the following diagram ``commutes" for some $\la,\rho$:
$$
\begin{diagram}
\node{\la} \arrow{s,l}{S}
 \arrow{e,t}{[\bp]} \node{\nu} \arrow{s,r}{T}\\
\node{\mu} \arrow{e,b}{[\bq]} \node{\rho}
\end{diagram}
$$
Since tableaux are sequences of strips, we can immediately reduce the pushout bijection to
the case that $S$ and $T$ are both single strips. One might try to
straightforwardly reduce to the case that paths $\bp$ and $\bq$  are single
moves $m$ and $m'$. This does not work: not all pairs $(S,m)$ admit a pushout.
Those that do will be called \defit{compatible}. The bijection
\eqref{E:tabbranchpushout} is defined by combining certain moves
(called augmentation moves) with pushouts of compatible pairs.
The proof of Theorem \ref{T:pushpullbij} will be
completed in \S\ref{S:prooftheo}.

\begin{proof}[Proof of Theorem \ref{T:ktab_decomp}]
Let $\nu$ be the empty $k$-shape. Then
the only possibility for $\la$ is the empty $k$-shape, $S$ runs over
$\RT_\mu^k$, and by Proposition~\ref{L:maxcore}, $\rho$ runs over
$\Core^k$
and $T$ over $\Weak^k_\rho$.  By Theorem~\ref{T:pushpullbij}, we thus
have a bijection between $\RT_\mu^k$ and $\bigsqcup_{\rho\in\Core^k}
\left(\Weak^k_\rho \times  \Patheq^k(\mu,\rho)\right)$.
Theorem \ref{T:ktab_decomp}
follows since it is known that each generating function
$\WS_\rho^{(k-1)}[X]$ is a symmetric function \cite{LM:ktab}.
\end{proof}

\begin{proof}[Proof of Theorem \ref{T:Pieri}]
The $k$-Schur functions satisfy (essentially by definition) the Pieri rule \cite{LM:ktab}
$
h_r[X] \, \ksf^{(k)}_\mu[X] = \sum_{\rho} \ksf^{(k)}_\rho[X]$
where the sum is over weak strips $\rho/\mu$ of $k+1$-cores of rank $r$.

Thus for a fixed $\la \in \Ksh^k$,
\begin{align*}
h_r[X] \,\ksgen_\la^{(k)}[X]
&= \sum_{\mu \in \Core^{k+1}} |\Patheq^k(\mu,\la)| h_r[X]\,\ksf^{(k)}_\mu[X] \\
&= \sum_{\mu \in \Core^{k+1}}
|\Patheq^k(\mu,\la)| \sum_{\text{rank $r$ weak strips $\rho/\mu$}}
\ksf^{(k)}_\rho[X]\\
& =\sum_{\rho \in \Core^{k+1}} \ksf^{(k)}_\rho[X]
\sum_{\text{rank $r$ reverse maximal strips $\rho/\mu$}} |\Patheq^k(\mu,\la)|  \\
&= \sum_{\rho \in \Core^{k+1}} \ksf^{(k)}_\rho[X]
\sum_{\text{rank $r$ maximal strips $\nu/\la$}} |\Patheq^k(\rho,\nu)| \\
&= \sum_{\text{rank $r$ maximal strips $\nu/\la$}} \ksgen_\nu^{(k)}[X].
\end{align*}
In the third equality we used Proposition \ref{L:maxcore}, and in the fourth equality we used Theorem \ref{T:pushpullbij}.
\end{proof}

\begin{remark} \label{R:kstrips}
Suppose that strips of rank $k$ are
allowed.
The results of this paper hold with a few minor changes.\footnote{However, the
concept of lower augmentable corner
which will be introduced in \S\S\ref{subsectionaugmentation}
needs to be slightly modified: we define an
augmentable corner $b$ of a strip $S=\mu/\la$ as
usual, except we disallow the case that $b$ lies in a row of $S$
that already contains $k$ cells.}
For instance, Theorem~\ref{T:pushpullbij} and Theorem~\ref{T:Pieri}
are still valid (with the case $r=k$ being allowed in
Theorem~\ref{T:Pieri}).  However, as the rest of the remark
should make clear, the extension of Theorem~\ref{T:ktab_decomp}
is somewhat more subtle.

When $\nu= \emptyset$ (and thus also $\mu=\emptyset$),
the bijection on which Theorem~\ref{T:pushpullbij} relies,
associates to a reverse-maximal tableaux $S$
a pair $(T,[\bq])$, where $T$ is a maximal tableau of a
given shape $\rho$. If strips of rank $k$ are allowed then
Proposition~\ref{L:maxcore}
is not valid anymore, as adding a strip of rank $k$ on a $k$-core
does not produce a $k$-core.  Therefore, if the weight of $S$
has entries of size $k$, then the pushout of $S$
does not produce a weak
tableau $T$ and Theorem~\ref{T:ktab_decomp} ceases to be valid.
In the following, we will extend  Theorem~\ref{T:ktab_decomp}
to the case when strips of rank $k$ are allowed.

The fact that $\WS_\rho^{(k-1)}[X]$ is a symmetric function is not sufficient
anymore to prove that $\dkr_\mu^{(k)}[X]$ is a symmetric function.
By Theorem~\ref{T:ktab_decomp}, the sum of the terms that do not involve any
power
$x_i^k$ in $\dkr_\mu^{(k)}[X]$ is a symmetric function.  Furthermore, we have that
if $\rs(\mu)_1 < k$ then  $\dkr_\mu^{(k)}[X]$ does not involve any power
$x_i^k$ and thus $\dkr_\mu^{(k)}[X]$ is a symmetric function in that case
(see the proof of Proposition~\ref{P:dualtriang}).
Now if $\mu$ is a $k$-shape such that $\rs(\mu)_1=k$, then by Lemma \ref{L:surfacestrip}, $\mu$ has a unique reverse maximal strip of rank $k$.  In this manner, it is not too difficult to
see that the sum of the terms in $\dkr_\mu^{(k)}[X]$ that involve powers
of $x_i^k$ is equal to $B_k \, \dkr_\la^{(k)}[X]$, where
$B_k m_{\beta}=m_{(k,\beta)}$ and is thus a symmetric function by induction.
This proves that $\dkr_\mu^{(k)}[X]$ is also a symmetric function
if strips of rank $k$ are allowed.

Finally, to complete the extension of
Theorem~\ref{T:ktab_decomp}, let $\pi^{(k)}$ be the projection
onto $\Lambda/I_{k-1}$.  Then for $\mu\in\Ksh^k$, the
cohomology $k$-shape function $\dkr_\mu^{(k)}[X]$ has the
decomposition
\begin{align}
\pi^{(k)}(\dkr_\mu^{(k)}[X]) = \sum_{\rho\in \Core^k} |\Patheq^k(\mu,\rho)|\,
\WS_\rho^{(k-1)}[X]
\end{align}
The result holds trivially from Theorem~\ref{T:ktab_decomp}
since the projection will kill every $x^{\wt(T)}$ such that $T$ has
a strip of rank $k$.

\end{remark}

\subsection{Elementary properties of $\dkr_{\la}^{(k)}[X]$
and $\ksgen_\la^{(k)}[X]$} \label{subsec:elemen}

\begin{proposition}  For $\la \in \Ksh^k$, let
$\la^r$ (resp. $\la^c$) be the unique element of $\Core^{k+1}$ such that
$\rs(\la)=\rs(\la^r)$ (resp. $\cs(\la)=\cs(\la^c)$). Then
one has
\begin{equation} \label{Eq:PropTriang}
\ksgen_{\la}^{(k)}[X] = \ksf_{\la^r}^{(k)}[X] +
\sum_{\rho\in \Core^{k+1};\, \rs(\rho) \, \triangleright \, \rs(\la)}
|\Patheq^k(\rho,\lambda)|\, \ksf_{\rho}^{(k)}[X]
\end{equation}
and, similarly,
\begin{equation}  \label{Eq:PropTriang2}
\ksgen_{\la}^{(k)}[X] = \ksf_{\la^c}^{(k)}[X] +
\sum_{\rho\in \Core^{k+1}; \, \cs(\rho) \, \triangleright \, \cs(\la)}
|\Patheq^k(\rho,\lambda)|\, \ksf_{\rho}^{(k)}[X]
\end{equation}
\end{proposition}
\begin{proof}
We will only prove \eqref{Eq:PropTriang}, since
\eqref{Eq:PropTriang2} follows similarly. As already mentioned at
the beginning of \S\S\ref{SubS:poset}, if  $m$ is a column move
(resp. row move) from $\nu$ to $\mu$, then $\rs(\nu) \triangleright
\rs(\mu)$ (resp. $\rs(\nu) = \rs(\mu)$) in the dominance order on
partitions. Since $\lambda$ is obtained from $\rho$ by a sequence of
moves, it only remains to show that
$|\Patheq^k(\lambda^r,\lambda)|=1$. That is, that there exists a
unique equivalence class of paths in the $k$-shape poset from
$\lambda^r$ to $\lambda$.  Or equivalently, that there exists a
unique equivalence class of paths in the $k$-shape poset from
$\lambda$ to $\lambda^r$. The proof is analogous to the proof that
given $\mu/\lambda$ a strip, there exists a unique equivalence class
of paths in $\Str^\mu$ to the reverse-maximal strip $\mu/\nu$ (see
Proposition~\ref{P:maxstripuniquerev}).
\end{proof}

Let $\mu$ be a $k$-shape.  The {\it surface strip} $\mu/\la$ of $\mu$ is the horizontal strip consisting of the topmost cell of each column of $\mu$.

\begin{lemma}\label{L:surfacestrip}
The surface strip of $\mu$ is the unique reverse maximal strip of $\mu$ with rank $\rs(\mu)_1$.
\end{lemma}
\begin{proof}
It follows from the definitions and the fact that $\mu$ is a
$k$-shape that the skew shape $\Int(\mu)/\Int(\la)$ is the surface
strip of $\Int(\mu)$.  Thus $\rs(\la)$ is obtained from $\rs(\mu)$
by removing the first row, and $\cs(\la)$ is obtained from
$\cs(\mu)$ by reducing the last $\rs(\mu)_1$ columns each by $1$. In
particular, $\mu/\la$ is a strip.  It is clear that the surface
strip $S= \mu/\la$ is reverse maximal.

Let $S' = \mu/\nu$ be another reverse maximal strip with rank $\rs(\mu)_1$.  The modified columns of $S'$ must be exactly the last $\rs(\mu)_1$ columns, and furthermore, $\rs(\nu) = \rs(\la)$.  It follows that $\nu = \la$.
\end{proof}

\begin{proposition} \label{P:dualtriang}
Let $\la \in \Ksh^k$ be such that $\rs(\la)=\nu$.
If we allow strips of rank $k$, then
\begin{equation}
\dkr_{\la}^{(k)}[X] = m_{\nu} + \sum_{\mu \, \triangleleft \, \nu} \tilde
K_{\nu \mu} m_{\mu}
\end{equation}
for some coefficients $\tilde K_{\nu \mu} \in \mathbb Z_{\geq 0}$.
\end{proposition}
\begin{proof}
Let $T\in\T_{\la}^k$, and suppose that $\wt(T)=\mu$.  Since
 $T = \emptyset=\la^{(0)} \subset
\la^{(1)}\subset \dotsm\subset \la^{(N)}=\lambda$
is a sequence of strips, we have in particular that
$\rs(\la^{(i)})/\rs(\la^{(i-1)})$ is a horizontal $\mu_i$-strip
for all $i$.
This gives immediately that $\mu \trianglelefteq \nu$ (think of the triangular
expansion of the homogeneous symmetric functions into Schur functions).
Now, the unique  $T\in\T_{\la}^k$
such that $\wt(T)=\nu$ is obtained by recursively taking the surface strips of $\la$. 
Finally,  $\tilde K_{\nu \mu} \in \mathbb Z_{\geq 0}$ by definition
of $\dkr_{\la}^{(k)}[X]$.
\end{proof}

\begin{proposition}  Let $\la \in \Ksh^k$, and let $\omega: \La \to \La$
be the homomorphism that sends the $r^{th}$ complete symmetric function
to the $r^{th}$ elementary symmetric function.
 Then
\begin{equation} \label{Eq:kschuromega}
\omega(\ksgen_{\la}^{(k)}[X])= \ksgen_{\la'}^{(k)}[X]
\end{equation}
and
\begin{equation} \label{Eq:dualomega}
\omega(i_{k-1}(\dkr_{\la}^{(k)}[X]))= \dkr_{\la'}^{(k)}[X] \mod I_{k-1}
\end{equation}
where $\lambda'$ is the conjugate of $\la$, and where $i_k$ was defined
in \eqref{Eq:dual1}.
\end{proposition}
\begin{proof}
For the proof of \eqref{Eq:kschuromega} we proceed by induction.
The result holds for $k$ large since
in that case $\ksf_{\la}^{(k)}[X]=s_\la[X]$ is a usual Schur functions
and it is known that
$\omega(s_\la[X])=s_{\la'}[X]$.  From \eqref{E:ksf} when $t=1$
we get
\begin{equation} \label{Eq:omega2}
\omega(\ksgen_\la^{(k)}[X]) = \sum_{\rho\in \Core^{k+1}}  |\Patheq^k(\rho,\lambda)|\,
\omega(\ksf^{(k)}_\rho[X])
\end{equation}
Since $\ksf^{(k)}_\rho[X]=\ksgen^{(k+1)}_\rho[X]$ from
\eqref{Eq:kmoins1}, we can suppose by induction that
$\omega(\ksf^{(k)}_\rho[X])=\ksf^{(k)}_{\rho'}[X]$.  We also have
$|\Patheq^k(\rho,\lambda)|=|\Patheq^k(\rho',\lambda')|$ by the
transposition symmetry of the $k$-shape poset, and thus
\eqref{Eq:kschuromega} follows from \eqref{Eq:omega2}.

For the proof of \eqref{Eq:dualomega}, we have from
Theorem~\ref{T:ktab_decomp}
that
\begin{align} \label{Eq:dualomega2}
\omega(i_{k-1}(\dkr_\mu^{(k)}[X])) = \sum_{\rho\in \Core^k} |\Patheq^k(\mu,\rho)|\,
\omega(i_{k-1}(\WS_\rho^{(k-1)}[X]))
\end{align}
The duality \eqref{Eq:dual1}
between $k$-Schur functions and dual $k$-Schur functions
implies that
\begin{equation}
\omega(i_{k-1}(\WS_\rho^{(k-1)}[X]))=\WS_{\rho'}^{(k-1)}[X] \mod I_{k-1}
\end{equation}
given that
$\omega(\ksf^{(k-1)}_\rho[X])=
\ksf^{(k-1)}_{\rho'}[X]$ (see \cite{LM:QC}) and that
$\omega$ is an isometry.  The result then
follows from \eqref{Eq:dualomega2}
since, as we saw earlier, $|\Patheq^k(\mu',\rho')|=|\Patheq^k(\mu,\rho)|$.
\end{proof}

\subsection{Basics on strips}

The remainder of this section deals with the properties of strips
and augmentation moves. Sections \ref{sec:rowpushout} and
\ref{sec:columnpushout} study pushouts involving row and column
moves respectively.

The next results help in checking whether something is a strip.
\begin{property} \label{L:checkstrip}
Let $S = \mu/\la$ be a horizontal strip of $k$-shapes
and $c$ a column which contains a cell of $S$.  Then $\cs(\mu)_c \ge \cs(\la)_c$.
\end{property}
\begin{proof} Let $b\in\bdy\la$ be in column $c$ and $b'$ be the cell just above $b$.
Since $\mu/\la$ is a horizontal strip,
$k \geq h_\la(b)\ge h_\mu(b')$. This implies $b'\in \bdy \mu$ and the
result follows since there is a cell of $S$ in column $c$.
\end{proof}

\begin{property}\label{L:innerstrip}
Suppose $S = \mu/\la$ is a strip.  Then $\bdy \la \setminus \bdy
\mu$ is a horizontal strip.
\end{property}
\begin{proof} The lemma follows from Property~\ref{L:checkstrip}
and the fact that $\mu/\lambda$ is a horizontal strip.
\end{proof}

\begin{lemma}\label{L:coverseq}
Let $\mu/\la$ be a strip and $c$ be such that
$cs(\mu)_c = cs(\la)_c + 1\le \cs(\la)_{c-1}$.
Then there is a cover-type $\la$-addable string $s$ such that $c_{s,d} = c$, $\la \cup s \in \Ksh^k$, and
$\mu/(\la \cup s)$ is a strip.
\end{lemma}
\begin{proof}
Let $b$ be the unique cell in column $c$ of $\mu/\la$.
The hypotheses imply that $b$ is $\la$-addable.
Let $s$ be the maximal $\la$-addable string such that
$s\subset\mu$ and $s$ ends with $b$. Say the top cell $y$ of $s$ is
in row $i$. Let $x=(i,j)=\Left_i(\bdy\la)$.  Let $b'$ be the $\la$-addable
cell in column $j$, if it exists.

The string $s$ is of row-type or cover-type by the hypotheses.
Suppose $s$ is of row-type. Then $h_\la(x)=k$. Then $b'\cup s$
is a $\la$-addable string. Since $\mu/\la$ is a strip
we have $\cs(\mu)_j\ge \cs(\la)_j$. But $x=\Bot_j(\bdy\la)\not\in\bdy\mu$.
Hence $b'\in \mu$, contradicting the maximality of $s$.

Therefore $s$ is of cover-type.
Since $\mu/\la$ is a strip, $\rs(\mu)_i\ge\rs(\la)_i$. Suppose
$\rs(\mu)_i=\rs(\la)_i$, and  let
$\mu/\la$ have $\ell$ cells in row $i$.  By supposition,
there are also $\ell$ cells  of $\bdy\la\setminus\bdy\mu$
in row $i$.  Since $\cs(\lambda) \subseteq \cs(\mu)$ and $\mu/\lambda$
is a horizontal strip, there must then be cells of
$\mu/\lambda$ in columns $j,\dots,j+\ell-1$ that are contiguous
to the $\ell$ cells of $\mu/\lambda$ in row $i$.  In particular,
the $\lambda$-addable corner $b'$ is contiguous to $y$.
Again $b'\cup s$ is a $\la$-addable string, contradicting the
maximality of $s$.

Therefore $\rs(\mu)_i>\rs(\la)_i$, which ensures that
$\rs(\mu)_i \geq \rs(\lambda \cup s)_i$.  Since $\rs(\mu)/\rs(\la)$ is a horizontal strip
we deduce that $\la \cup s$ is a $k$-shape.
It then easily follows that
$\mu/(\la\cup s)$ is a strip.
\end{proof}

\begin{corollary}\label{C:coverseq}
Any strip $S=\mu/\la$ of rank $\rho$ has a
decomposition into $\rho$ cover-type strings. More precisely,
for every sequence $c_1,c_2,\dotsc,c_\rho$ of modified columns of $S$
such that $c_i<c_j$ if $\cs(\la)_{c_i}=\cs(\la)_{c_j}$, there is a chain in $\Ksh$:
$\la=\la^{(0)}\subset\dotsm\subset \la^{(\rho)}=\mu$
such that $t_i = \la^{(i)}/\la^{(i-1)}$ is a $\la^{(i-1)}$-addable cover-type
string with modified column $c_i$.
\end{corollary}
\begin{proof}
Follows by induction from Lemma \ref{L:coverseq}.
\end{proof}

\begin{corollary}\label{C:stripmodifiedrows}
Let $S = \mu/\la$ be a strip. If a column contains a cell in $\bdy\la\setminus\bdy\mu$
then it also contains a cell of $S$.
\end{corollary}
\begin{proof}
Each such cell is a removed cell for one of the cover-type strings that
constitute $S$.
\end{proof}

\begin{lemma}\label{L:stripcomparelengths}
Suppose $S = \mu/\la$ is a strip, and let
$s=\{a_1,\dots,a_{\ell}\} \subseteq S$ be
a $\la$-addable string (with $a_1$ the topmost).  For each $i \in [1,\ell]$
let $r_i$ (resp. $c_i$) be the row (resp. column) of $a_i$.  Then
\begin{enumerate}
\item If $i<j$ then there are at least as many cells of $S$ in row $r_j$ than
 there are in row $r_i$.
\item If $i>j$ then there are at least as many  $\mu$-addable cells in column $c_j$
than there are in column $c_i$.
\end{enumerate}
\end{lemma}
\begin{proof} For (1), let $r_i$ and $r_{i+1}$ violate the first assertion,
and let $b$ be the the rightmost cell of $S$ in row $r_i$.  It is then easy to
see that in $\mu$ we have the contradiction that
the column of $b$ is larger than the column of the first
cell of $S$ in row $r_i$.

For (2), it is enough to consider the case
$c_i$ and $c_{i+1}$.  Since
$a_i$ and $a_{i+1}$ are contiguous, $\Left(\bdy \lambda)_{r_{i+1}}$ lies in
column $c_i$.  Suppose that column $c_{i+1}$ has $p \geq 1$ $\mu$-addable
cells, and let $b=\Bot(\partial \mu)_{c_i^-}$ lie in row $R$.
Then row $R$ is  at  least $p$ rows
above row $r_{i+1}$ since otherwise $\rs(\mu)_R$ would be larger than
$\rs(\lambda)_{r_{i+1}}$ contradicting the fact that
$\rs(\mu)/\rs(\lambda)$ is
 a horizontal strip.  Since $\cs(\mu)_{c_i^-} \geq \cs(\mu)_{c_i}$,
column $c_i$ needs to have at least $p$ $\mu$-addable cells.
\end{proof}

\subsection{Augmentation of strips} \label{subsectionaugmentation}

We first observe the following:
\begin{remark}\label{R:posmodaugmen}
The negatively modified columns (resp. rows) of an
augmentation move of the strip $S$ are positively modified columns (resp. rows)
of $S$.
\end{remark}

\begin{property}\label{L:augmentcolumn}
All augmentation column moves of a strip $S = \mu/\la$ have rank 1.
\end{property}
\begin{proof}
If it were not the case, the modified rows of $m$ (which all have the same length
by definition) would violate the condition that
$\rs(m * \mu)/\rs(\la)$ is a horizontal strip.
\end{proof}

Let $S = \mu/\la$ be a strip. A $\mu$-addable cell $a$ is called
\begin{enumerate}
\item
a \defit{lower augmentable corner} of $S$ if adding $a$ to $\mu$
removes a cell from $\bdy\mu$ in a modified column $c$ of $S$
in the same row as $a$.
\item
an \defit{upper augmentable corner} of $S$ if $a$ does not lie on top of
any cell in $S$ and adding $a$ to $\mu$ removes a cell from
$\bdy \mu$ in a modified row $r$ of $S$ and in the same column as $a$.
\end{enumerate} We say that $a$ is {\it associated} to $c$
(or $r$, respectively).
We call a modified column $c$ of $S$ {\it leading} if the cell $c
\cap S$ (the cell of $S$ in column $c$) is leftmost in its row in $S$.

\begin{lemma}\label{L:augmenmove}
Let $S=\mu/\lambda$ be a strip.  Then any  augmentation move $m$ contains an
augmentable corner of $S$.
\end{lemma}
\begin{proof} If the strip $S$ admits an augmentation row (resp. column) move $m$ then the top
left (resp. bottom right) cell of $m$ is a lower (resp. upper) augmentable corner of $S$.
\end{proof}

\begin{definition}
A \defit{completion row move} is one in which all strings start in the same row.
It is maximal if the first string cannot be extended below.
A \defit{quasi-completion column
  move} is a column
 augmentation move from a strip $S$ that contains
no lower augmentable corner.
A \defit{completion column move} is a quasi-completion move from a strip $S$ that contains
no upper augmentable corner below its unique
(by Property~\ref{L:augmentcolumn}) string \footnote{The reason for distinguishing
between completion and quasi-completion column moves will
only become apparent in \S\ref{SS:Sectionmaxcomp}
(Lemma~\ref{L:maxcompleteequiv}).}.  A completion
column move or a quasi-completion column move is maximal if its string
cannot be extended above.
A completion move is a completion row/column move.
\end{definition}

The definition of completion move is transpose-asymmetric since strips are.
Our main result for augmentations of strips is the following.
Its proof occupies the remainder of the section.

\begin{prop} \label{P:maxstripunique} Let $S=\mu/\la$ be a strip.
\begin{enumerate}
\item $S$ has a unique maximal augmentation $S'\in\Str_\la$.
\item There is one equivalence class of paths
in $\Str_\la$ from $S$ to $S'$.
\item The unique equivalence class of paths in $\Str_\la$ from $S$ to $S'$
has a representative consisting entirely of maximal completion moves.
\end{enumerate}
\end{prop}

Let $m = s_1 \cup s_2 \cup \cdots \cup s_r$ be an augmentation row
move from $S$.  Then $c_{s_i,u}$ is a modified column of $S$ for each
$i \in [1,r]$.  Since $m$ is a move and $m * \mu\in\Ksh$,
the columns $\{c_{s_i,u} \mid i \in [1,r]\}$ are part of a group
of modified columns of $S$ and must be the rightmost $r$ columns in
this group, by Property~\ref{C:rowmovecolsize}.

\begin{lemma}\label{L:extendstring}
Let $s$ be a $\la$-addable
row-type (resp. column-type)
string that cannot be extended below (resp. above).
Then $\cs(\la)_{c_{s,d}}<\cs(\la)_{c_{s,d}^-}$
(resp. $\rs(\la)_{r_{s,u}}<\rs(\la)_{r_{s,u}^-}$).
\end{lemma}
\begin{proof}
Let $s=\{a_1,a_2,\dotsc,a_\ell\}$ be a row-type string and
suppose $\cs(\la)_c = \cs(\la)_{c^-}$ where $c=c_{s,d}$.
We have $\la_{c^-}>\la_c$
since $a_{\ell}$ is $\la$-addable, and thus
 the cell immediately to the left of $b=bot_{c}(\bdy\la)$
 is not in $\bdy\la$.
This implies that
$h_\la(b)\ge k-1$
so that $h_{\la\cup s}(b)=k$ given that $s$ is a row-type string.
By Remark \ref{R:rowshift}, there is a $\la$-addable corner
at the end of the row of $b$ that is contiguous with $a_\ell$,
so that $s$ can be extended below, a contradiction. The column-type
case is similar.
\end{proof}

\begin{lemma}\label{L:extendcompletion}
Let $m = s_1 \cup s_2 \cup \cdots \cup s_r$
be a non-maximal completion row move
from $\la\in\Ksh$ and let $t_1 =
s_1 \cup \{a_{\ell+1},a_{\ell+2},\ldots,a_{\ell+\ell'}\}$ be the
maximal row-type string which extends $s_1$ below.  Then there
is a unique completion row move $n = t_1 \cup t_2 \cup \cdots \cup
t_r$ from $\la$.
\end{lemma}
\begin{proof}
By Proposition~\ref{P:uniquedecomp}, if $n$ exists, it is determined by $t_1$.
We show that there are strings $t_2, t_3, \ldots, t_r$ that can be
added to $\la \cup t_1$.  Let $s_i = \{a_1^{(i)},\ldots,a_\ell^{(i)}\}$
and $R=\row(a_{\ell+1})$.
Since the cells $a_\ell^{(1)},\dots,a_\ell^{(r)}$ lie in the same row, we have by
Lemma~\ref{L:comparelengths} that there is room for $a_i^{(1)},\dots,a_i^{(r)}$
in the row of $a_{i}=a_i^{(1)}$ for all $i= \ell+1,
\dots,\ell+\ell'$.  These cells obviously
all lie in columns of $\bdy\la$ of the same length.
Thus the result holds since
Lemma~\ref{L:extendstring} allows us to
conclude that $n*\la\in\Ksh$.
\end{proof}

\begin{lemma}\label{L:augmentcorner}
Suppose $a$ is a lower augmentable corner in row $R$ of the strip $S = \mu/\la$,
associated to the column $c$.
\begin{enumerate}
\item $\Bot_c(\bdy\mu)=\Bot_c(\bdy\la)=(R,c)$ and $h_\mu(R,c)=k$.
\item
$c$ is a leading column.
\item
Suppose that $c'>c$ is a leading column such that $\cs(\mu)_{c'}=\cs(\mu)_c$.
Then there is a lower augmentable corner $a'$ which is associated to $c'$.
\item
Let $r$ be the number of cells of $S$ in the row of $\Top_c(\mu)$. Then
$\la_{R^-}\ge \la_R+r$.
\end{enumerate}
\end{lemma}

\begin{proof}
(1) and (2) are straightforward using the fact that $c$ is a modified column of $S$.
For (3), let $b=\Bot_{c'}(\partial\mu)$. Then $h_\mu(b)\ge h_\mu(\Bot_{c}(\partial\mu))=k$ by (1).
Thus the addable corner at the end of the row of $b$ (assured to
exist by Remark~\ref{R:rowshift}) must be lower augmentable. (4) is implied by Remark~\ref{R:rowshift}.
\end{proof}

\begin{lemma}\label{L:completionstrip}
Suppose $m$ is a non-maximal completion row move from a strip $S=\mu/\la$.
Let $t_1 = s_1 \cup \{a_{\ell+1},a_{\ell+2},\ldots,a_{\ell+\ell'}\}$ be the maximal
row-type string which extends the first string $s_1=\{a_1,\dotsc,a_\ell\}$ of $m$ below. Then the completion row
move $n$ from $\mu$ of Lemma \ref{L:extendcompletion} is a maximal
completion row move from $S$.
\end{lemma}
\begin{proof}
We use the notation of Lemma \ref{L:extendcompletion}.
We first show that $n*S$ is a horizontal strip.   Since
$a_{\ell+1}=a_{\ell+1}^{(1)}$ is a lower augmentable corner of the strip $m*S$,
by Lemma~\ref{L:augmentcorner}(4), the cells
$\{a^{(i)}_{\ell+1} \mid i \in [1,r]\}$ lie above cells of $\la$.
Now suppose that the cells $\{a^{(i)}_{\ell+j} \mid i \in [1,r]\}$ do not lie
on any cell of $S$.  Let $R$ be the row of
$\{a^{(i)}_{\ell+j+1} \mid i \in [1,r]\}$, and let $c$ the column of
$a^{(r)}_{\ell+j}$.  Since $\Bot_c(\partial \mu)$ lies in row $R$ and
since there are no cells of $S$ in column $c$
by supposition, we have that $\Left_{R-1}(\partial \lambda)$
is strictly to the right of column $c$ by
Corollary~\ref{C:stripmodifiedrows}.  Therefore, there are at least
$r$ extra cells of $\partial \mu$ in row $R$ to the left of $\Left_{R-1}(\partial \lambda)$.
This implies that $\lambda_{R-1}-\mu_R
\geq r$ since $\rs(\mu)/\rs(\lambda)$ is a horizontal strip.  This proves that
$\{a^{(i)}_{\ell+j+1} \mid i \in [1,r]\}$ also do not lie on on any cell of
$S$ and we get by induction that $S$ is a horizontal strip.
Since $n$ is a row move, $\rs(n*\mu)/\rs(\la)=\rs(\mu)/\rs(\la)$
is a horizontal strip.   Finally, since $n$ removes
cells in the same columns of $\bdy \mu$ as $m$ does and $n*S$ is a horizontal
strip, $\cs(n*\mu)/\cs(\la)$ is
a vertical strip. Hence $n*S$ is a strip in $\Ksh$.
\end{proof}

\begin{lemma}\label{L:upperaugmentcore}  Let $\la\in\Core^k$
and $S = \mu/\la$ a strip
with no lower augmentable
corners.  Suppose $a$ is a $\mu$-addable corner such that adding
$a$ to the shape $\mu$ removes a cell from $\bdy \mu$ in a modified
row $r$ of $S$.  Then $a$ is an upper augmentable corner.
\end{lemma}
\begin{proof}
We must show that $a$ does not lie on top of any cell in $S$.
Suppose otherwise. Let $b=\Bot_{\col(a)}(\bdy\mu)$.  We have
that $\col(a)$ is not a modified column of $S$, for otherwise
$\row(b)$ has a lower
augmentable corner for $S$, a contradiction.

Let $b'$ be the cell immediately below $b$.
Since $\col(a)$ is not a modified column but it
contains a cell in $S$, we must have $b' \in \bdy \la - \bdy
\mu$.  Furthermore, Property~\ref{L:innerstrip} implies that $b'=\Left(\bdy\la)_{\row(b')}$.
Since $\rs(\mu)/\rs(\la)$ is a horizontal strip we have
$\rs(\la)_{\row(b')} \geq \rs(\mu)_{\row(b)}$ and hence that
$h_\la(b') \geq h_\mu(b) =
k$.  But $b' \in \bdy \la$ implies $h_\la(b') = k$,
contradicting that $\la\in\Core^k$.
\end{proof}

\begin{example}
The $k$-core condition in Lemma \ref{L:upperaugmentcore} is
necessary. For $k=4$ consider
$$
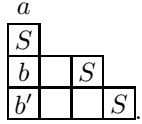
\tableau[sby]{\bl a \\S\\b&&S\\b'&&&S\\}.
$$
\end{example}

\begin{lemma}\label{L:completionaugment}
Let $a$ be a lower augmentable corner of a strip $S=\mu/\la$ associated to column $c$.
Let $S$ contain $r$ cells in the row containing the cell $c \cap S$.
Suppose that $a$ is chosen rightmost amongst augmentable corners
associated to columns of the same size in $\bdy \mu$. Let $t_1 =
\{a = a_{1},a_{2},\ldots,a_{\ell'}\}$ be the maximal row type
string which extends $a$ below.  Then there is a maximal
completion row move $n$ from $S$ which has rank $r$ and
initial string $t_1$.
\end{lemma}
\begin{proof}
We apply the construction in Lemma \ref{L:completionstrip} with $m$ an empty move.
$m$ is not maximal since there is a lower augmentable
corner $a$ in some row $R$, which can be extended to a row-type string by
Lemma \ref{L:makestring}. The move $m$ has rank $r$ since $r$ cells can be added
to row $R$ of $\la$ by Lemma~\ref{L:augmentcorner}(4). The choice of $a$ guarantees that
the negatively modified columns of $n$ have the same size and that the
monotonicity of column sizes is preserved. The
argument in Lemma \ref{L:completionstrip} completes the proof.
\end{proof}

\begin{lem}\label{L:anyaugment}
Let $S=\mu/\la$ be a strip with
$t > 1$ lower augmentable corners and $m$ an augmentation row move from $S$.
Then there is a maximal completion row move $M$ from $S$ such that $(m,
M)$ admits an elementary equivalence $\tm M \equiv \tM m$ in $\Str_\la$,
$\tm$ contains $t-1$ lower augmentable corners, and $\tM$ is a maximal completion row move.
\end{lem}
\begin{proof}
Let $a$ be the rightmost lower augmentable corner of $S$ inside $m$
(it exists by Lemma~\ref{L:augmenmove}).
Then define $M$ to be the move from $S$ that arises from $a$ as described
in Lemma~\ref{L:completionaugment}.
The strings of $m$ and $M$ containing $a$
both start at $a$. By Lemma \ref{L:moves_relative_position},
 $m$ and $M$ are matched above.

If $m$ and $M$ are matched below, it follows from the proof of
Proposition~\ref{P:rel} that $m=M$.  This is a contradiction since
$M$ contains only one augmentable corner of $S$.  Therefore $M$ continues
below $m$ and
the pair $(m,M)$ is a
Case \eqref{it:rowequivmatch} of an elementary row equivalence:
$\tilde m$ contains all the lower augmentable corners of
$m$ apart from $a$; $\tilde M$ contains a lower part of $M$.  The
claimed properties follow immediately.
\end{proof}

\begin{example} Column completions behave somewhat differently: it is not
always possible to choose the maximal extension of an upper augmentable
corner, e. g.,
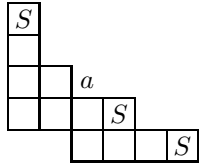
$$
\tableau[sby]{S\\\\&&\bl a\\&&&S\\\bl&\bl&&&&S\\}
$$
with $k = 5$.
\end{example}

\begin{lemma} \label{L:everywhere}
Let $S=\mu/\lambda$ be a strip
and let $c=\Left_{R^-}(\bdy\lambda)$ and $b=\Left_R(\bdy\mu)$ for some row $R$.
Suppose that $\rs(\mu)_R= \rs(\lambda)_{R^-}$.
Then $c$ and $b$ lie in the same column if one of the following conditions
is satisfied.
\begin{enumerate}
\item $h_\mu(b)<k-1$.
\item There is a cell of $S$ in the column of $b$ and
$h_\mu(b)=k-1$.
\item There is no upper augmentable corner of $S$ in the column of $b$, there is no
lower augmentable corner of $S$ in row $R$,
row $R$ is modified by $S$ and $h_\mu(b)=k$.
\end{enumerate}
\end{lemma}
\begin{proof}
$c$ cannot be to the left of $b$ by Property \ref{L:innerstrip}.
Suppose that $c$ is to the right of $b$. Let $c'$ be the cell left-adjacent to $c$.

Case (1). Since $h_\mu(b)<k-1$ and $\rs(\mu)_R= \rs(\lambda)_{R^-}$,
we have the contradiction that $h_\la(c')\le 2+h_\mu(b)\le k$.

Case (2). Since there is a cell
of $S$ in the column of $b$, any column of $\la$ to the right
of $b$ is shorter than the column of $c$ in $\mu$.
Given that $h_\mu(b)=k-1$ and $\rs(\mu)_R= \rs(\lambda)_{R^-}$,
we have the contradiction that $h_\la(c')\le 1+h_\mu(b)=k$.

Case (3). By hypothesis $h_\mu(b)=k$.
If there is a cell of $S$ in $\col(b)$ then $\col(b)$ is a modified column of $S$ and we have the
contradiction that there is a lower augmentable corner of $S$ in row $R$
($\la_{R^-} > \mu_{R}$ by hypothesis).
Otherwise we get the contradiction
that there is an upper augmentable corner of $S$ associated to row $R$ in $\col(b)$.
\end{proof}

\begin{lemma}\label{L:columncompletionaugment}
Let $S = \mu/\la$ be a strip without lower augmentable corners
and let $a$ be an upper augmentable corner of $S$. Let $s =
\{a_1,\ldots,a_\ell = a\}$ be the maximal extension of $a$ above,
subject to the condition that the $a_i$ do not lie on top of cells
of $S$.  Then $m = s$ is a quasi-completion column move from $S$.
\end{lemma}

\begin{proof}
Let $\row(a_1)=R$, $b=(R,c)=\Left_R(\bdy\mu)$ and $d=\Left_{R^-}(\bdy\lambda)$.
Suppose first that $s$ is the maximal extension of $a$ without being
constrained by not lying on top of $S$.  By Lemmata
\ref{L:makestring} and \ref{L:extendstring}, $s$ is a column type
string and $m*\mu$ is a $k$-shape.
It suffices to show that
$\rs(m*\mu)/\rs(\la)$ is a horizontal strip. Since $S$ is a strip,
$\rs(\la)_{R^-}\le \rs(\mu)_{R^-}=\rs(m*\mu)_{R^-}$,
so it remains to show that $\rs(m*\mu)_R \leq \rs(\la)_{R^-}$.
The only way this would fail is if $\rs(\mu)_R = \rs(\la)_{R^-}$.
Since $s$ is maximal, we have $h_b(\mu)<k-1$.  Thus from
Lemma~\ref{L:everywhere}(1),  we have that $b$ and $d$ lie in the same
column.  Since $\rs(\mu)_R = \rs(\la)_{R^-}$ this gives the contradiction
that $a_1$ lies over a cell of $S$.

Now suppose that $s$ is blocked from extending further by the
constraint of not lying on top of $S$.  Consider first the case
that $h_\mu(b) = k-1$, and observe that, as
in the previous case, if $m*\mu$
fails to be a $k$-shape or $\rs(m*\mu)/\rs(\la)$ fails to be a
horizontal strip, then  $\rs(\mu)_R = \rs(\la)_{R^-}$
($S$ is a strip and thus
 $\rs(\mu)_{R-} \geq \rs(\lambda)_{R-} \geq \rs(\mu)_R$).
Lemma~\ref{L:everywhere}(2) then implies that $b$ and $d$ lie in the same
column and the result follows from the argument given in the previous case.
Finally, consider the case that $h_\mu(b) = k$.  Column $c$ is not
a modified column of $S$ since $a_1$ cannot be a lower augmentable
corner.  Thus the cell $b'$ below $b$ is in $\partial\la$
and so is the cell below $a_1$.
This gives the contradiction $h_{b'}(\la)>h_{b}(\mu)=k$.
%
%
\end{proof}

\subsection{Maximal strips for cores}

Recall that a strip $S$ is maximal if it does not admit any augmentation move.

\begin{prop}\label{P:maxstrip}
A strip is maximal if and only if it has no augmentable corners.
\end{prop}
\begin{proof} By Lemma~\ref{L:augmenmove}, if
the strip $S$ admits an augmentation move then $S$ has an augmentable corner.
Conversely, if $S$ has an augmentable corner, then $S$ admits a maximal completion move by
Lemmata~\ref{L:completionaugment}
and \ref{L:columncompletionaugment}.
\end{proof}

\begin{lemma}
Let $S = \mu/\la$ be a maximal strip and let $c, c'$ be two
modified columns such that $\cs(\la)_c = \cs(\la)_{c'}$.
Then the cells $S \cap c$ and $S \cap c'$ are on the same row.
\end{lemma}
\begin{proof}
Suppose otherwise. We may assume that $c'=c+1$.
Let $b'=bot_{c'}(\bdy\la)$ and $b$ be the cell just below $bot_c(\bdy\la)$.
Then
$$h_\mu(b') \ge h_\la(b') + 1 \geq h_\la(b) - 1 \ge k\,.
$$
Since $c'$ is a modified column, $b'\in\bdy\mu$, that is, $h_\mu(b') = k$.
But then there must be a lower augmentable corner for $S$ at the end of
the row of $b'$, contradicting Proposition \ref{P:maxstrip}.

\end{proof}


\begin{proposition}\label{L:covercore}
Suppose $S = \mu/\la$ is a maximal cover and $\la\in\Core^k$.
Then $\mu\in\Core^k$.
\end{proposition}
\begin{proof}
It suffices to check $h_\mu(x)$ for cells $x$ in the modified row
or column, such that $h_\mu(x)=h_\la(x)+1$.
For the modified row $r$, let $b=\Left_r(\bdy\la)$.
Then $h_\la(b)<k-1$, for otherwise $S$ is not maximal.
All cells to the left of $b$ have $h_\la > k$.
Similar reasoning applies to the modified column.
\end{proof}


\begin{proposition}\label{L:stripcore}
Suppose $S = \mu/\la$ is a maximal strip
and $\la\in\Core^k$.
Then $\mu\in\Core^k$.
\end{proposition}
\begin{proof} By Proposition~\ref{L:covercore} it suffices to
show that $S$ can be expressed as a sequence of maximal covers.
Construct a sequence of covers for $S$ using Lemma \ref{L:coverseq}.
By Proposition \ref{P:maxstrip}, $S$ has no augmentable corner.  We
claim that this implies that the successive covers constructed have
no augmentable corners which would then imply their maximality.  Note
that a modified row or column of one of these covers is immediately
also one of $S$.

Let $C = \nu/\kappa$ be such a cover.  For lower augmentable
corners, this is clear since such augmentable corners are
augmentable corners of $S$. For an upper augmentable corner $a \notin S$
of $C$, we apply Lemma \ref{L:upperaugmentcore} which implies that
$a$ is an upper augmentable corner of $S$.
\end{proof}

\subsection{Equivalence of maximal augmentation paths}
Let $S=\mu/\la$ be a strip. Suppose $m$ and $M$ are distinct augmentation
moves from $S$. We say that the pair $(m,M)$ defines an augmentation equivalence
if there is an elementary equivalence of the form $\tM m \equiv \tm M$ such that
$\tm$ and $\tM$ are augmentation moves from the strips $M*S$ and $m*S$
respectively. Note that given the elementary equivalence,
$\tm$ and $\tM$ are augmentation moves
if and only if $\tM*m*S$ (or $\tm*M*S$) is a strip.

\begin{lem}\label{L:maxrowcolcommute}
Suppose $m$ and $M$ are respectively a maximal completion row move
and an augmentation column move
from a strip $S$. Then $(m,M)$ defines an augmentation equivalence. Moreover,
\begin{enumerate}
\item If $m\cap M =\emptyset$ then no cell of $m$ is contiguous to a cell of $M$.
\item If $m\cap M\ne \emptyset$ then $m$ continues above and below $M$.
\end{enumerate}
\end{lem}
\begin{proof} Let $S=\mu/\la$.
For (1) the non-contiguity follows from the maximality of $m$.
The other assertions follow easily in this case.

So let $m\cap M\ne\emptyset$. By Property~\ref{L:augmentcolumn},
$M$ consists of a single $\mu$-addable column-type string $t$.
By Property~\ref{P:movestrip} and Lemma \ref{L:addable}
the first string $s$ of $m$ must be the unique string that meets $M$.
We claim that $m$ continues above and below $M$.
Consider the highest cell $x\in m\cap M$. Suppose $x$ is the highest cell in $s$.
Let $b=\Left_{\row(x)}(\bdy\mu)$. By Definition~\ref{D:stringtypes} $h_\mu(b)=k$.
$x$ is a lower augmentable corner of $S$ so that $b$ lies in a modified column of $S$.
By Property~\ref{L:notfinishsame}, $M$ and $m$ cannot be matched
above and thus we get the contradiction that $M$ needs to
continue above $m$ with an element in the column of $b$ lying on top of $S$.
Therefore $m$ continues above $M$.
Now consider the lowest cell $y\in m\cap M$. Suppose $y$ is the lowest cell
in $s$. By Definition~\ref{D:stringtypes},
$h_\mu(\Bot_{\col(y)}(\bdy\mu))\le k-1$. Again by Property~\ref{L:notfinishsame},
$M$ and $m$ cannot be matched below and thus $y$ is not the lowest cell of
$t$, which gives $h_\mu(\Bot_{\col(y)}(\bdy\mu))= k-1$.
But then $s$ can be extended below, contradicting the maximality of $m$.
Therefore $m$ continues above and below $M$. It is straightforward to check that
in this case, the resulting elementary equivalence $\tM m \equiv \tm M$, when
applied to $S$, ends at a strip.
\end{proof}

\begin{lem}\label{L:maxmoverowcommute}
Suppose $m$ and $M$ are distinct
maximal completion row moves from the strip $S=\mu/\la$. Then $(m,M)$ defines an augmentation
equivalence. Moreover, $m\cap M=\emptyset$ and exactly one
of the following holds:
\begin{enumerate}
\item $m$ and $M$ do not interfere.
\item $(m,M)$ is interfering and lower-perfectible with added cells
$m_\per$ such that $m\cup m_\per$ is a maximal completion row move
from the strip $M*S$ and $M\cup m_\per$ is a maximal completion row
move from $m*S$.
\item The same as (2) with the roles of $m$ and $M$ interchanged.
\end{enumerate}
\end{lem}
\begin{proof}
Suppose that $m \cap M \ne \emptyset$ and $m \neq M$. Then by
maximality we may without loss of generality assume that $m$
continues above $M$ but $(m,M)$ is matched below.  But $m$ must
contain a cell in a modified column associated to $M$,
contradicting the assumption that $S\cup m$ is a strip.

Therefore $m \cap M = \emptyset$. We may assume that $m$ and $M$ interfere,
and that $m$ is above $M$. Let $c$ be the column such that
$\cs(\mu\cup m \cup M)_c = \cs(\mu\cup m\cup M)_{c^-}+1$.
Then $m$ adds the cell atop column $c$ of $\bdy\mu$ and $M$ removes the cell $(r,c^-)=\Bot_{c^-}(\bdy\mu)$.
Let $x=\Left_r(\bdy\mu)$ and $y=\Bot_c(\bdy\mu)=(r',c)$. By Definition~\ref{D:stringtypes}
$h_\mu(x)=k$ and $h_\mu(y)\le k-1$.

Suppose $r>r'$. We have $\rs(\mu)_{r'} \ge \rs(\mu)_r$ and $\cs(\mu)_{\col(x)}=\cs(\mu)_{c^-}=\cs(\mu)_c+1$,
so that $h_\mu(y) \ge h_\mu(x) - 1 = k-1$.
Therefore $h_\mu(y)=k-1$. By Remark \ref{R:rowshift} there is a $\mu$-addable cell in row $r$,
which is below and contiguous with the cell of $m$ in column $c$. This contradicts the maximality of $m$.
Therefore $r=r'$ and $y=(r,c)$.

Let $\nu=\mu\cup M$. We have $h_\nu(y)=k-1$. Since
the negatively modified columns of $M$ and the positively modified columns of
$m$ have their lowest $k$-bounded cell in the same row
and $\rs(\nu)_{r^-}\ge\rs(\nu)_r$, we deduce that $\nu_{r^-}-\nu_r  \ge \rk(m)+\rk(M)$.
Using this and the maximality of $M$, by Lemma \ref{L:extendcompletion}
we may deduce that viewing $m$ as $\nu$-addable,
each of its strings can be maximally extended below to contain a cell in each
of the rows of $M$ by Lemma~\ref{L:comparelengths}.
Call the added cells $m_\per$. It is straightforward
to verify the remaining assertions.
\end{proof}

%
%

\begin{lem}\label{L:maxmovecolcommute}
Suppose $m$ and $M$ are distinct
maximal quasi-completion column moves for the strip $S$.  Then $(m,M)$
defines an augmentation equivalence.  Moreover, exactly one of the
following holds:
\begin{enumerate}
\item $m \cap M = \emptyset$ and $m$ and $M$ do not interfere.
\item $m \cap M \neq \emptyset$ and either $m\subset M$ or $M\subset m$.
\end{enumerate}
\end{lem}
\begin{proof} Suppose that $m\cap M=\emptyset$ and there is interference.
Recall that $m$ and $M$ are of rank 1 and without loss of generality we can
suppose that $M$ is above $m$.  Then the highest cell of $m$ is in a row $R$
such that $R^-$ is a positively modified row of $S$ by
Remark~\ref{R:posmodaugmen} (since $R^-$
is a negatively
modified row of $M$),
and such that
$\rs(m * \mu)_R=\rs(m * \mu)_{R^-}$.  We thus have the contradiction
that $\rs(m * \mu)/\rs(\lambda)$ is not a horizontal strip.

If $m \cap M \neq \emptyset$, then by maximality they finish at the same point
above.  Given that both are of rank 1, we deduce that
$m \subset M$ or  $M \subset m$.
\end{proof}

We now prove Proposition \ref{P:maxstripunique}.
\begin{proof} Let $S$ be a strip such that
\begin{equation}\label{E:stripsat}
\text{the result holds for any proper augmentation of $S$.}
\end{equation}

Let $(m_1,m_2,\ldots,m_x)$ and $(M_1,M_2,\ldots,M_y)$ be distinct
augmentation paths from $S$ to maximal strips. If $m_1=M_1$ we are done by induction.
So suppose $m_1\ne M_1$. If $m_1$ and
$M_1$ are maximal completion moves then by Lemmata
\ref{L:maxrowcolcommute}, \ref{L:maxmoverowcommute} and
\ref{L:maxmovecolcommute}, the pair $(m_1,M_1)$  defines an augmentation equivalence
$\tM_1 m_1 \equiv \tm_1 M_1$.
By \eqref{E:stripsat} there are equivalences of augmentation paths of the form
\begin{equation*}
m_x \dotsm m_2 m_1 \equiv \dotsm \tM_1 m_1 \equiv \dotsm \tm_1 M_1 \equiv M_y \dotsm M_2 M_1.
\end{equation*}
It thus suffices to show that any augmentation path $(m_1,m_2,\ldots,m_x)$ ending
at a maximal strip, is equivalent to one which begins with a maximal completion
move. If $m_1$ is a non-maximal completion row move, then Lemma
\ref{L:completionstrip} implies that $m_1 \subset m$ where $m$ is a
maximal completion row move with the same lower augmentable corner.
But then $(m \backslash m_1)(m_1) \equiv m$ is a row equivalence and
using \eqref{E:stripsat} we deduce that $(m_1,m_2,\ldots,m_x)$ is
equivalent to a path beginning with $m$.
 A similar argument
works for the column case.

We may thus assume that $m_1$ is
a non-completion augmentation row or column move.  In the case of the
non-completion augmentation row move, the argument
is completed by
Lemma \ref{L:anyaugment}.  In the case of the
non-completion augmentation column move, $S$ either contains some lower augmentable
corners or some upper augmentable corners above the one associated to $m_1$.
In the former case, let $M$ be the maximal completion row move
associated to a lower augmentable corner $a$ of $S$ such as described in
Lemma~\ref{L:completionaugment}.  By Lemma~\ref{L:maxrowcolcommute},
the argument is completed in that case.  In the latter case,
let $M$ be the maximal completion column move
associated to the highest upper augmentable corner $a$ of $S$ such as described in
Lemma~\ref{L:columncompletionaugment}.   The lemma then follows
from Lemma~\ref{L:maxmovecolcommute}.

\end{proof}

\subsection{Canonical maximization of a strip}
\label{SS:maxstripalg}

Let $S=\mu/\la$ be a strip. The following algorithm \texttt{MaximizeStrip} produces an augmentation path
$\bq=(\mu=\mu^0\rightarrow \mu^1\rightarrow\dotsm\rightarrow\mu^M=\rho)$ in $\Str_\la$ ending at a
maximal strip $\rho/\la$. This path is comprised of maximal completion moves; the existence of such a path
is asserted by Proposition \ref{P:maxstripunique}(3).

\begin{tabbing}
xxxx\=xxxx\=xxxx\=xxxx\=\kill
\textbf{proc} \texttt{MaximizeStrip}($\mu,\la$): \\
\>  \text{\textbf{local} $\rho := \mu, q := (\mu)$} \\
\>  \textbf{while True}: \\
\> \>   \textbf{if} the strip $\rho/\la$ has a lower augmentable corner: \\
\>\>\>      let $x$ be the rightmost one \\
\>\>\>            let $s$ be the maximal $\rho$-addable string extending $x$ below \\
\>\>\>            $\rho := \rho \cup s$ \\
\>\>\>            append $\rho$ to $q$ \\
\>\>\>            \textbf{continue} \\
\>\>    \textbf{if} the strip $\rho/\la$ has an upper augmentable corner: \\
\>\>\>      let $x$ be the rightmost one \\
 \>\>\>            let $s$ be the maximal $\rho$-addable string extending $x$ above, \\
\>\>\>\>               subject to not having a cell atop $\rho/\la$ \\
 \>\>\>            $\rho := \rho \cup s$ \\
 \>\>\>            append $\rho$ to $q$ \\
 \>\>\>            \textbf{continue} \\
\>\>        \textbf{break} \\
\>    \textbf{return} $q$
\end{tabbing}
The path $\bq$ is initialized to be the path of length zero starting and ending at $\mu$
and the current strip $\rho/\la$ is initialized to be $\mu/\la$.
Whenever the current strip $\rho/\la$ has a lower augmentable corner,
the algorithm appends a completion row move to $\bq$ by Lemma \ref{L:completionaugment}
and applies the move to $\rho$. Whenever the current strip $\rho/\la$ has no lower augmentable corner but an upper augmentable one,
the algorithm appends a completion column move $m$ to $\bq$ by Lemma \ref{L:columncompletionaugment}
and applies the move to $\rho$. When $\rho/\la$ has no augmentable corners, by Proposition \ref{P:maxstrip}
the algorithm terminates with $\rho/\la$ a maximal strip and returns the current path $\bq$.

\begin{example} \label{X:maximizestrip} Let $k=4$, $\la=(6,6,4,4,2,2,1)$, and $\mu=(7,6,4,4,2,2,2)$.
Calling \texttt{MaximizeStrip} with the strip $\mu/\la$, the output path $\bq=(\mu=\mu^0,\mu^1,\mu^2=\rho)$ is given below.
The boxes of $\bdy\la\cap \bdy\mu^i$ are black and the rest belong to
the strip $\mu^i/\la$.
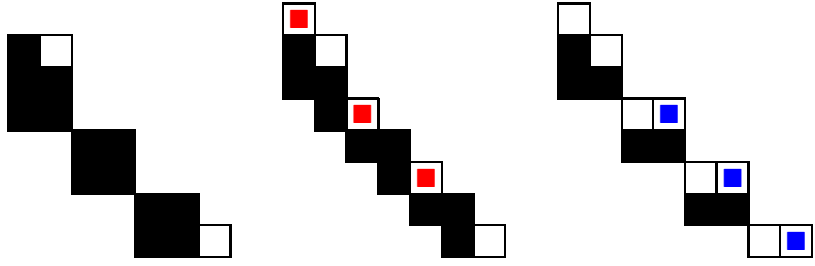
\begin{equation*}
\tableau[sby]{\bl \\ \fl& \\ \fl&\fl \\ \fl&\fl \\ \bl&\bl&\fl&\fl \\
 \bl&\bl&\fl&\fl\\ \bl&\bl&\bl&\bl&\fl&\fl \\ \bl&\bl&\bl&\bl&\fl &\fl&}
 \qquad
\tableau[sby]{{\color{red}\blacksquare} \\ \fl& \\ \fl&\fl \\ \bl&\fl&{\color{red}\blacksquare} \\ \bl&\bl&\fl&\fl \\
 \bl&\bl&\bl&\fl&{\color{red}\blacksquare} \\ \bl&\bl&\bl&\bl&\fl&\fl \\ \bl&\bl&\bl&\bl&\bl &\fl&}
 \qquad
\tableau[sby]{ \\ \fl& \\ \fl&\fl \\ \bl&\bl& & {\color{blue}\blacksquare} \\ \bl&\bl&\fl&\fl \\
 \bl&\bl&\bl&\bl&&{\color{blue}\blacksquare} \\ \bl&\bl&\bl&\bl&\fl&\fl \\ \bl&\bl&\bl&\bl&\bl &\bl&&{\color{blue}\blacksquare}}
\end{equation*}
$\mu^0/\la$ has no lower augmentable corner but has a unique upper augmentable one, namely,
the lowest red cell in $\mu^1$. $\mu^1/\la$ has a unique lower augmentable corner, the
highest cell colored blue in $\mu^2$. $\mu^2/\la$ is maximal.
\end{example}

\section{Pushout of strips and row moves}
\label{sec:rowpushout} Let $(S,m)$ be an initial pair where
$S=\mu/\la$ is a strip and $m=\nu/\la$ is a nonempty row move.

We say that $(S,m)$ is compatible if it is reasonable, not
contiguous, and is either (1) non-interfering, or (2) is interfering
but is also pushout-perfectible; these notions are defined below.
For compatible pairs $(S,m)$ we define an output $k$-shape
$\eta\in\Ksh$ (see Subsections \ref{SS:rowpushnointerfere} and
\ref{SS:rowpushinterfere} for cases (1) and (2) respectively). This
given, we define the pushout
\begin{equation} \label{E:pushoutdef}
  \push(S,m) = (\tS,\tm)
\end{equation}
which produces a final pair $(\tS,\tm)$ where $\tS=\eta/\nu$ is a
strip and $\tm=\eta/\mu$ is a move (possibly empty). This is
depicted by the following diagram.
\begin{equation} \label{E:pushdiag}
\begin{diagram}
\node{\la} \arrow{s,t}{S} \arrow{e,t}{m} \node{\nu} \arrow{s,b,..}{\tS} \\
\node{\mu} \arrow{e,b,..}{\tm} \node{\eta}
\end{diagram}
\end{equation}

If $S$ is a maximal strip then $(S,m)$ is compatible (Proposition
\ref{L:maximalinterfere}).

\begin{property}\label{L:Smmod}
Let $(S,m)$ be an initial pair. Then a modified column $c$ of $S$
cannot be a negatively modified column of $m$.
\end{property}
\begin{proof}
Suppose otherwise. Let $c$ be the leftmost modified column of
$S=\mu/\la$ that is negatively modified by $m$.  We have that $c$ is
also the leftmost negatively modified column of $m$ since otherwise
$\cs(\mu)$ would not be a partition. By the previous comment,
$b=\Bot_c(\bdy\la)$ is leftmost in its row in $\bdy\la$ and
$h_\la(b) = k$. But looking at $S$ we see that $h_\la(b) < k$, a
contradiction.
\end{proof}

\subsection{Reasonableness}
We say that the pair $(S,m)$ is \defit{reasonable} if for every
string $s$ of $m$, either $s \cap S = \emptyset$ or $s \subset S$.
In other words, every string of $m$ which intersects $S$ must be
contained in $S$.

Suppose the string $s$ of $m$ satisfies $s\subset S$. We say that
$S$ \defit{matches $s$ below} if $c_{s,d}$ is a modified column of
$S$ and otherwise say that $S$ \defit{continues below} $s$.

\begin{lem}\label{L:initialstrings}
Suppose $(S,m)$ is reasonable where $m = s_1 \cup s_2 \cup \cdots
\cup s_r$.  If $S$ matches $s_i$ below then $S$ matches $s_j$ below
for each $j \leq i$.  If $S$ continues below $s_i$ then $S$
continues below $s_j$ for each $s_j$ on the same rows as $s_i$
satisfying $j \leq i$.
\end{lem}
\begin{proof}
The first assertion follows directly from the assumption that
$\cs(\mu)$ is a partition. In the case of the second assertion, we
have that $\Bot_{c}(\bdy\la)$ belongs to the same row for every
column $c$ corresponding to such $s_j$'s. Given that column $c$ is
not a modified column of $S$ there is a cell of $\partial \lambda
\setminus \partial \mu$ in column $c$.  Given that $\mu/\lambda$
needs to be a skew diagram, the assertion follows.
\end{proof}

\begin{lem}
Suppose $(S,m)$ is reasonable. Then every modified column $c$ of $S$
which contains a cell in $m$ is a positively modified column of $m$.
\end{lem}

\begin{lemma}\label{L:Smleftcolumns}
Suppose $(S,m)$ is reasonable and $s$ is a string of $m$.
Then $s\subset S$ if and only if $S$ contains a cell in column $c_{s,u}$.
\end{lemma}
\begin{proof} The forward direction is immediate from Corollary \ref{C:stripmodifiedrows}.
For the converse, suppose $c_{s,u}$ contains a cell $x$ of $S$.
Let $d$ be the number of strings of $m$ that are in the same rows as $s$
and are equal to $s$ or to its left. Then $S$ contains the $d-1$
cells to the left of $x$.
It follows from Property~\ref{L:Smmod} that
$S$ contains $d$ cells in the row of $\Bot_{\col(x)}(\bdy\la)$.
This puts the top cell of $s$ into $S$. By reasonableness $s\subset S$.
\end{proof}

\begin{proposition}\label{L:Smreasonable}
Suppose $(S,m)$ is an initial pair with $S=\mu/\la$ maximal. Then $(S,m)$ is
reasonable.
\end{proposition}
\begin{proof} Let $s=\{a_1,a_2,\dotsc,a_\ell\}$ be a string of $m$
and suppose $a_i\in S$. As in Corollary \ref{C:coverseq},
we choose the unique decomposition of $S$ into cover-type strings
such that the bottom cell of $t_j$ is the $j$-th modified column of $S$ for all $j$
(going from left to right) and $t_j$ is taken to be maximal
given $t_1,\dotsc,t_{j-1}$.

Suppose $a_i$ is in the string $t$ of $S$.
It suffices to show that (1) $a_{i-1}\in t$ if $i>1$ and (2) $a_{i+1}\in t$ if $i<\ell$.
We prove (2) as (1) is similar.

The proof proceeds by induction on the indent $\Ind_m(s)$ of $s$ in $m$.
Suppose first that $\Ind_m(s)=0$, that is, $s$ is $\la$-addable.
We have $a_{i+1}\in S$, for otherwise it would be a lower augmentable corner of $S$
which would contradict the maximality of $S$ by Proposition \ref{P:maxstrip}.
By the choice of the decomposition of $S$, $a_{i+1}$ and $a_i$ are both in $t$.

Now suppose the Lemma holds for all strings $s'$ of $m$
with $\Ind_m(s')<d$. Let $s'=\{b_1,b_2,\dotsc,b_\ell\}$ be the string of $m$ preceding $s$.
Since $d>0$, $b_j$ is just left of $a_j$ for all $j$. Since $a_i\in S$ it follows that $b_i\in S$.
By induction the cover-type string $t'$ of $S$ containing $b_i$ contains $s'$.
So $\col(b_i)=\col(a_i)^-$ is not a modified column of $S$.
This implies that $\col(a_i)$ is also not a modified column of $S$.
Due to the decomposition of $S$ into covers, this means that $t$ has
a cell below $a_i$, that is, $a_{i+1}\in t$.
\end{proof}

\subsection{Contiguity}
Suppose $(S,m)$ is reasonable where $S=\mu/\la$ and $m$ is a move from $\la$ to $\nu$.
We say that $(S,m)$ is \defit{contiguous} if
there is a cell $b \in \bdy \mu \cap \bdy \nu$ which is not present
in $\bdy (\mu \cup \nu)$; $b$ is called a \defit{disappearing cell}.

\begin{example} The following strip and move (indicated by $S$ and $m$ respectively)
are contiguous for $k=6$ with disappearing cell $b$.
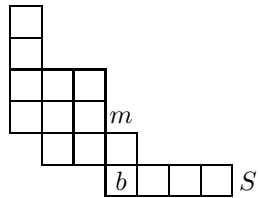
\begin{equation*}
\begin{diagram}
\tableau[sby]{ \\ \\ && \\ &&& \bl m \\ \bl&&& \\ \bl&\bl&\bl&b&&&&\bl S }
\end{diagram}
\end{equation*}
\end{example}

\begin{lem}\label{L:disappearing}
Suppose $b=(r,c)$ is a disappearing cell. Then
\begin{enumerate}
\item
Column $c$ is positively modified by $m$ and contains no cells of $S$,
\item
Row $r$ is modified by $S$.
\item
Column $c$ contains an upper augmentable corner for $S$.
\end{enumerate}
\end{lem}
\begin{proof} Let $n_m$ and $n_S$ be respectively the number of cells of $m$ and $S$ in row $r$.
If column $c$ contains a cell of both $S$ and $m$ then we have the contradiction
$h_{\mu\cup\nu}(b)=1+h_\la(b) + \max(n_m,n_S) = \max(h_\mu(b),h_\nu(b))\le k$.
A similar contradiction is reached if column $c$ contains neither a cell of $m$ nor one of $S$.

Suppose column $c$ contains a cell in $S$; it is the cell $x$ atop the column
$c$ of $\bdy\la$.  Then $n_m>n_S$ and
$h_\nu(b)=k$ since $b$ is a disappearing cell.  Let $y$ be the rightmost cell
of $m$ in row $r$, and observe that $y \not \in S$.
The move $m$ removes the cell
$b^*$ just left of $b$ and thus $\cs(\bdy\la)_{c^-}=\cs(\bdy\la)_c$.
By Property~\ref{C:rowmovecolsize} $m$ cannot negatively modify column $c^-$.
Therefore $m$ has a cell $x^*$ in column $c^-$ just left of $x$ that belongs
to the same string of $m$ as $y$.
Since $S$ is $\la$-addable, $x^*\in S$.   But by reasonableness of $(S,m)$,
since $y \not \in S$
we get the contradiction that $x^* \not \in S$.

Therefore column $c$ contains a cell of $m$ (namely, $x$).
We have $n_S>n_m$ and $h_\mu(b)=k$,
and $x$ has no cell of $m$ contiguous to and below it. Item (1) follows.

If $r$ is not a modified row of $S$ then $S$ removes the $n_S$ cells just left of $b$
and $S$ contains $n_S$ cells just left of $x$. Since $m$ is $\la$-addable
$m$ also contains the $n_S$ cells just left of $x$.
But then $m$ doesn't modify some of these columns (since $n_S>n_m$)
while it modifies column $c$,
contradicting Property~\ref{C:rowmovecolsize}. This proves (2).

It follows that $x\in m$ is an upper augmentable corner for
$S$, proving (3).
\end{proof}

By Proposition~\ref{L:Smreasonable}, we have the following corollary.
\begin{cor} \label{C:rowmaxnotcontig}
Suppose $S$ is a maximal strip and $m$ a row move.  Then $(S,m)$ is
non-contiguous.
\end{cor}

\subsection{Interference of strips and row moves}

Suppose that $(S,m)$ is reasonable and non-contiguous. Then
$$
\cs(\mu) -\cs(\la) + \cs(m*\la) = \change_\cs(S) +
\change_\cs(m) + \cs(\la).
$$
Recalling Notation \ref{N:arrow} let
\begin{align}\label{E:mprime}
m' &= \bigcup \,\{\,\text{strings $s$}  \subset m \mid \text{$s$ and
$S$ are
  not matched below}\} \\
\label{E:mplusrow}
  m^+&=\,\uparrow_S\!(m').
\end{align}
Define the vector $\change_\cs(m')$ by considering only the modified
columns of strings in $m'$. We say that $(S,m)$ is
\defit{non-interfering} if $\cs(\la)+\change_\cs(S)+\change_\cs(m')$ is
a partition, and \defit{interfering} otherwise.

\begin{remark} \label{R:interference}
$(S,m)$ is interfering if and only if $S$ and the last string $s$ of
$m$ are not matched below, $S$ modifies column $c^+$, and
$\cs(\la)_c = \cs(\la)_{c^+}+1$, where $c=c_{s,u}$.
\end{remark}

\begin{example} \label{X:interf}
With $k=7$ the pair $(S,m)$ is interfering: there is a violation of the $k$-shape property
in $m^+*\mu$. The set of cells $m'$ is comprised of the second and third strings of $m$.
In passing from $m'$ to $m^+$ the third string has been bumped up a row.
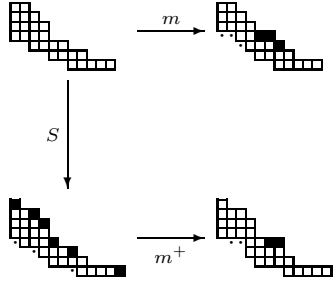
\begin{equation*}
\begin{diagram}
\node{\tableau[pby]{\bl \\ & \\ && \\ &&& \\ &&& \\ \bl&\bl&&&& \\ \bl&\bl&\bl&\bl&&&& \\ \bl&\bl&\bl&\bl&\bl&\bl&&&&&&\bl}}
\arrow{e,t}{m} \arrow{s,t}{S}
\node{\tableau[pby]{\bl \\ & \\ && \\ &&& \\ \bl \cdot &\bl \cdot &&&\fl&\fl \\ \bl&\bl&\bl\cdot&&&&\fl \\ \bl&\bl&\bl&\bl&&&& \\ \bl&\bl&\bl&\bl&\bl&\bl&&&&&&\bl}} \\
\node{\tableau[pby]{\fl \\ &&\fl \\ &&&\fl \\ &&& \\ \bl\cdot &&&&\fl \\ \bl&\bl&\bl\cdot&&&&\fl \\ \bl&\bl&\bl&\bl&&&& \\ \bl&\bl&\bl&\bl&\bl&\bl&\bl\cdot&&&&&\fl}}
\arrow{e,b}{m^+}
\node{\tableau[pby]{ \\ && \\ &&& \\ &&& \\ \bl&\bl\cdot&\bl\cdot&&&\fl&\fl \\ \bl&\bl&\bl&&&& \\ \bl&\bl&\bl&\bl&&&& \\ \bl&\bl&\bl&\bl&\bl&\bl&\bl&&&&&}}
\end{diagram}
\end{equation*}
\end{example}

\begin{lem}\label{L:pushbump}
The set of cells $m^+$ satisfies all the conditions for a move from
$\mu$ except that $(m^+)*\mu$ may not be a $k$-shape.  Furthermore,
we have $\cs((m^+)*\mu) = \cs(\la) +\change_\cs(S) +
\change_\cs(m')$.
\end{lem}
\begin{proof}
Let $m^+ = t_1 \cup t_2 \cup \cdots \cup t_\rho$ where each $t_i$ is
either a string in $m' \setminus S$, or a string in $m' \cap
S$ shifted upwards.  We assume that the $t_i$ are ordered from
left to right, as is the convention for row moves.  It is clear that
the $t_i$ are weak translates of each other in the correct columns.  In order
to prove the lemma, we will
show
that they are successive
row type addable strings that are translates of the strings of $m$.
 We
proceed by induction on $i$.

First suppose that $t_i \subset m'$ was not bumped up.  By Lemma
\ref{L:Smleftcolumns}, $c_{s,u}$ contains no cells of $S$. By
non-contiguity and Corollary \ref{C:stripmodifiedrows}, column $c_{s,d}$ is
identical in $\la$ and $\mu$, and also in $\nu$ and $\mu \cup
\nu$.  Thus $t_i$ is a row type string of $\mu \cup t_1 \cup
\cdots \cup t_{i-1}$ equal to $s_i$.

Now suppose that $t_i$ was bumped up from $s_i \in m' \cap S$.  By
Lemma \ref{L:initialstrings}, it suffices to check the case that
$s_i$ is $\la$-addable.  First we show that $t_i$ is addable.
This is clear if $s_i$ is not equal to $s_1$, for
$s_{i-1}$ is higher than $s_i$.  Lemma \ref{L:continuebelow} deals
with the case $s_i=s_1$. The diagram of $t_i$ is a translate of that of $s_i$
by Lemma \ref{L:Smleftcolumns} and the assumption that $S$
continues below $s_i$, which ensures that the their modified columns agree in size.
\end{proof}

\begin{lem} \label{L:continuebelow}
Suppose $S$ continues below the first string
$s_1 = \{a_1,a_2,\ldots,a_\ell\}$ of  $m$.  For
each $i \in [1,\ell]$ let $c_i$ be the column containing $a_i$. Then
there is an addable corner of $\mu$ in column $c_i$.
\end{lem}
\begin{proof}
Consider the case $i = \ell$ and set $c=c_{\ell}$.
We prove the equivalent statement that  column $c^-$ either intersects $S$ or satisfies
$(\la)_{c^-} \geq (\la)_{c} + 2$.
Since $m$ is a move,
we have $\cs(\la)_{c^-} > \cs(\la)_{c}$.  Assume there is no cell of $S$ in
column $c^-$. Then
Corollary \ref{C:stripmodifiedrows}  and ``continuing below''
imply that the bottom
of $c^-$ in $\bdy \la$ starts higher than that of $c$. This
implies $(\la)_{c^-} \geq (\la)_{c} + 2$.

The general case then follows from Lemma~\ref{L:stripcomparelengths} since
$s_1$ is $\la$-addable and there is a $\mu$-addable corner in column $c=c_{\ell}$.
\end{proof}

\subsection{Row-type pushout: non-interfering case}
\label{SS:rowpushnointerfere} Let $(S,m)$ be reasonable,
non-contiguous, and non-interfering. Then by definition we declare
$(S,m)$ to be compatible, set $\eta=(m^+)*\mu$, let $(\tS,\tm)$ be
as in \eqref{E:pushdiag}, and define the pushout of $(S,m)$ by
\eqref{E:pushoutdef}. By Lemma \ref{L:pushbump} and Proposition
\ref{L:etastrip}, $\tm$ is a (possibly empty) row move and $\tS$ is
a strip.

\begin{proposition}\label{L:etastrip}
Suppose $(S,m)$ is non-interfering.  Then $\eta/\nu$ is a strip.
\end{proposition}
\begin{proof}
It is immediate that $\eta/\nu$ is a horizontal strip.  We have
$\rs(\eta)/\rs(\nu) = \rs(\mu)/\rs(\la)$, which is a horizontal strip by
assumption.  Also \begin{align*} \cs(\eta) -\cs(\nu) &=
\change_\cs(S) + \change_\cs(m') - \change_\cs(m) \\ &=
\change_\cs(S) - \change_\cs(m \setminus m')\end{align*}
Observe that $m \setminus m'$ corresponds to the strings $s$ of $m$ such that
$s$ and $S$ are matched below.  Therefore  $\cs(\eta) -\cs(\nu)$
is a
0-1 vector since
the positively modified columns of $m \setminus m'$
cancel out with some modified columns of $S$, and the negatively
modified columns of $m \setminus m'$ do not coincide with modified
columns of $S$ by Property~\ref{L:Smmod}.
\end{proof}

\subsection{Row-type pushout: interfering case}
\label{SS:rowpushinterfere} Suppose $(S,m)$ is reasonable,
non-contiguous, and interfering. We say that $(S,m)$ is
\textit{pushout-perfectible} if there is a set of cells $m_\comp$
outside $(m^+)*\mu$ such that if
\begin{align} \label{E:etanotation}
\eta = ((m^+)*\mu) \cup m_\comp
\end{align}
then $\eta/\nu$ is a strip and $\eta/\mu$ is a row move from $\mu$
with the same initial string as $m^+$. By
Proposition~\ref{P:uniquedecomp}, $m_\comp$ is unique if it exists.

In the case that $(S,m)$ is pushout-perfectible, then by definition
we declare $(S,m)$ to be compatible. With $\eta$ as in
\eqref{E:etanotation} we define $(\tS,\tm)$ and the pushout of
$(S,m)$ by \eqref{E:pushdiag} and \eqref{E:pushoutdef}. By
definition $\tS$ is a strip and $\tm$ is a row move.

\begin{example} Continuing Example \ref{X:interf}, the cells of
$m_\comp$ are darkened as added to $m^+*\mu$:
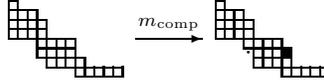
\begin{equation*}
\begin{diagram}
\node{\tableau[pby]{ \\ && \\ &&& \\ &&& \\ \bl&\bl&\bl&&&& \\ \bl&\bl&\bl&&&& \\ \bl&\bl&\bl&\bl&&&& \\ \bl&\bl&\bl&\bl&\bl&\bl&\bl&&&&&}}
\arrow{e,t}{m_\comp}
\node{\tableau[pby]{ \\ && \\ &&& \\ &&& \\ \bl&\bl&\bl&&&& \\ \bl&\bl&\bl&\bl\cdot&&&&\fl \\ \bl&\bl&\bl&\bl&&&& \\ \bl&\bl&\bl&\bl&\bl&\bl&\bl&&&&&}}
\end{diagram}
\end{equation*}
\end{example}

\begin{proposition}\label{L:maximalinterfere}
Suppose $(S,m)$ is interfering with $m$ a row move and $S$ maximal.
Then $(S,m)$ is pushout-perfectible (and hence compatible).
Furthermore the strings of $m_\comp$ lie on the same rows as the
final string of $m$ and no column contains both cells of $m_\comp$
and $S$.
\end{proposition}
\begin{proof} $(S,m)$ is reasonable and
non-contiguous by Proposition \ref{L:Smreasonable} and Corollary
\ref{C:rowmaxnotcontig}, so it makes sense to consider interference.

By Remark~\ref{R:interference}, $(S,m)$ interferes only if there is
a modified column $c$ of $S$ such that column $c^-$ is the rightmost
negatively modified column of $m$ and $\cs(\la)_c =\cs(\la)_{c^-} -
1$.  Let $b=(r,c)=\Bot_c(\bdy\la)$ and $b'=\Bot_{c^-}(\bdy\la)$. We
have $h_\mu(b) \leq k$ since $c$ is a modified column of $S$. If
$\row(b)<\row(b')$ then $h_\mu(b) \ge h_\la(b')= k$ and thus
$h_\mu(b)=k$.  This means that $S$ has a lower augmentable corner in
row $r$, contradicting Proposition \ref{P:maxstrip} and maximality.
Therefore $b$ and $b'$ are in row $r$, and this row corresponds to
the row of the top cell of the last string of $m$. Now suppose that
$c^+$ is also a modified column of $S$  with
$\cs(\la)_c=\cs(\la)_{c^+}$, and let $\bar b=\Bot_{c^+}(\bdy\la)$.
By the same argument we get that $b$ and $\bar b$ lie in the same
row. Continuing in this way, we get that all modified columns $d$ of
$S$ such that $\cs(\la)_d=\cs(\la)_c$ occupy the same rows. If there
are $\ell$ of them and $\ell'$ cells of $m$ in row $r$, we have
established that $\lambda_{r^-}-\lambda_{r} \ge \ell+ \ell'$.
Therefore $\rho=m^+ * \mu$ is such that $\rho_{r^-}-\rho_{r} \ge
\ell$ since exactly $\ell'$ cells of $m^+ \cup S$ lie in that row by
hypothesis (otherwise column $c$ would not be a modified column of
$S$). By Lemma~\ref{L:extendcompletion}, any row $R$ below row $r$
that contains a cell of the last string of $m$ is also such that
$\lambda_{R^-}-\lambda_{R} \ge \ell+ \ell'$.   Furthermore, for any
such row $R$ we also have $\rho_{R^-}-\rho_{R} \ge \ell$ since
again exactly $\ell'$ cells of $m^+ \cup S$ lie in that row by
hypothesis (otherwise there would be an upper augmentable corner
associated to a given row $R$, contradicting maximality).

We have established that $\ell$ cells can be added to the right of every cell
of the last string of $m$ in $\rho$, and from our proof, these cells do not lie
above cells of $S$.  Let $m_{\comp}$ be the union of those cells.
Defining $\eta =\rho\cup m_\comp$,
it is clear that $\eta/\nu$ is a
horizontal strip.  We have $\rs(\eta) = \rs(\mu)$ so
$\rs(\eta)/\rs(\nu)$ is a horizontal strip.  Finally, one checks
that $\cs(\eta)$ is a partition and $\cs(\eta)/\cs(\nu)$ a
horizontal strip in the same manner as in Proposition~\ref{L:etastrip}.
\end{proof}

\subsection{Alternative description of pushouts (row moves)}
Suppose $m = s_1
\cup s_2 \cup \cdots$ is a row move such that
$\change_{\cs}(s_1)$ affects
columns $c$ and $c + d$.  If $\alpha$ is not a partition, we suppose
that $\alpha_i + 1 = \alpha_{i+1} = \alpha_{i+2} = \cdots =
\alpha_{i+a}
> \alpha_{i+a+1}$.  Then the perfection of $\alpha$ with respect to
$m$ is the vector
$$
\per_{m}(\alpha) = \begin{cases} \alpha + \sum_{j = 1}^{a} (e_{i+j+d}
- e_{i+j}) & \mbox{if $\alpha$ is not a partition} \\
\alpha & \mbox{if $\alpha$ is a partition}\end{cases}
$$
Here $e_j$ denotes the unit vector with a 1 in the $j$-th position
and 0's elsewhere.

Let $(S = \mu/\lambda, m = \nu/\lambda)$ be any initial pair where
$m = s_1 \cup \cdots \cup s_r$.  Let $m'$ be the collection of cells
obtained from $m$ by removing $s_i$ whenever the positively modified
column of $s_i$ is a modified column of $S$.  It is easy to see that
$m'$ is of the form $s_j \cup s_{j+1} \cup \cdots \cup s_r$.  The
{\it expected column shape} $\ecs(S,m)$ of $(S,m)$ is defined to be
$$
\ecs(S,m) = \per_{m}(\cs(\la) +\change_\cs(S) + \change_\cs(m')).
$$

\begin{prop}\label{P:pushcrit}
Let $(S = \mu/\lambda, m = \nu/\lambda)$ be an initial pair where
$m\neq \emptyset$ is a row move. Suppose there exists a $k$-shape $\eta$
so that
\begin{enumerate}
\item $\cs(\eta) = \ecs(S,m)$
\item $\eta/\mu$ is either empty or a row-move whose string diagrams
are translates of those of $m$
\item $\nu \subset \eta$.
\end{enumerate}
Then $(S,m)$ is compatible and $\push(S,m) = (\eta/\nu,\eta/\mu)$.
In particular, $(\eta/\nu)$ is a strip.
\end{prop}
\begin{proof}  It is easy to see that $\eta/\mu$ decomposes into row type strings as
$m'' \cup m_\comp$ where $\cs(m'' * \mu) = \cs(\mu)+
\change_\cs(m')$.  Since $m''$ modifies the same columns as $m'$,
and the two have the same diagrams we conclude that each string of
$m''$ is either a string in $m'$ or a string in $m'$ shifted up one
cell.  But $m''$ is a collection of strings on $\mu$, so the strings
in $m'$ must be reasonable with respect to $S$.

We now claim that $m_\comp \cap m = \emptyset$.  Suppose otherwise.
Let $a$ be the rightmost cell in the intersection $m_\comp \cap m$,
lying in a string $s \in (m \setminus m')$ and a string $t \in
m_\comp$.  If $a$ is not the rightmost cell in $s$ we let $b$ be the
cell immediately right of $a$ in $s$. Now $s$ and $t$ have the same
diagram so we deduce that the cell $b'$ after $a$ in $t$ is either
equal to $b$ or immediately to the left of $b$.  In either case,
this contradicts the assumption that $a$ is rightmost.  Thus $a$ is
in the positively modified column $c$ of $s$.  But by the original
assumptions $c$ is also a modified column of $S$.  This contradicts
the fact that $m_\comp \cap S = \emptyset$ and we conclude $m_\comp
\cap m = \emptyset$.  Now we apply (3)
to see that all strings $m
\setminus m'$ must have already been contained in $\mu$ -- thus
$(S,m)$ is reasonable.


Suppose $(S,m)$ is contiguous.  By Lemma \ref{L:disappearing}, this
means there is a disappearing cell $b$, and $b$ is in a column $c$
which does not contain cells of $S$ but does contain cells of $m$ (it is in
fact a positively modified column of $m$).
By reasonableness, the column $c$ thus contains cells of $m'$ and in
particular is not in a modified column of $m_\comp$.  Thus
$\cs(\lambda)_c = \ecs(S,m)_c - 1$.  However, the disappearance
implies $\cs(\eta)_c = \cs(\lambda)_c$, a contradiction.

To show that $\eta/\nu$ is a horizontal strip, we only need to show
that no cell of $m_{\comp}$ lies above a cell of $S$.  Suppose $x$
is the leftmost cell in $m_{\comp}$ that lies above a cell of $S$,
and let $r$ be the row of $x$.  Let $s \in m_{\comp}$ be the string
that contains $x$ and let $c$ be the column of the cell removed in
row $r$ when adding $s$. Since $\rs_r(\mu) \leq \rs_{r^-}(\lambda)$
we have that $c$ is weakly to the right of $\Left_{r^-}(\partial
\la)$ and thus the cell in row $r^-$ and column $c$ belongs to
$\partial\lambda \setminus \partial \mu$.  Hence, by
Corollary~\ref{C:stripmodifiedrows}, there is a cell of $S$ in
column $c$. First assume that $x$ is not the highest cell in $s$,
and let $y$ be above $x$ in $s$. Then $y$ is in column $c$ and we
either have the contradiction that $y$ lies above a cell of $S$ or
that $m_{\comp} \cap S \neq \emptyset$. Now assume that $x$ is the
highest cell in its string. This time we have the contradiction that
$c$ is not a modified column of $S$.


Since $\rs(\eta)=\rs(\mu)$ and $\rs(\nu)=\rs(\lambda)$ we have that
$\rs(\eta)/\rs(\nu)$ is a horizontal strip.  Finally,
by supposition the negatively modified columns of $m_{\comp}$ are
positively modified columns of $S$ and the lowest cell of each string of
$m_{\comp}$ modifies positively its column.
Since $\eta/\nu$
is a horizontal strip, we have that $\cs(\eta)/\cs(\nu)$
is a vertical strip.
\end{proof}

\begin{lem}\label{L:mcompcol}
Let $(S,m)$ be a compatible initial pair with $m$ a row move.
Then the set of strings $m_{\comp}$
and the row move $m$ do not share any columns.
\end{lem}
\begin{proof} In the proof of Proposition \ref{P:pushcrit} it was shown that
no cell of $m_\comp$ can lie above a cell of $S$. Therefore the cells of $m_\comp$
lie above cells of $\la$. Suppose $m_\comp$ and $m$ share columns.
Then they intersect, and must do so in $m\setminus m'\subset S$, a contradiction.
\end{proof}

\section{Pushout of strips and column moves}
\label{sec:columnpushout}
In this section we consider initial pairs $(S =
\mu/\lambda, m = \nu/\lambda)$ consisting of a strip and a column
move.

We define $(S,m)$ to be compatible if it is reasonable,
non-contiguous, normal, and either (1) it is non-interfering or (2)
it is interfering but is pushout-perfectible; these notions are
defined below. As for row moves, in each of the above cases we
specify an output $k$-shape $\eta\in\Ksh$ and define the pushout of
$(S,m)$ and the final pair $(\tS,\tm)$ as in \eqref{E:pushoutdef}
\eqref{E:pushdiag}.

We omit proofs which are essentially the same in the row and column
cases.

\subsection{Reasonableness}
We say that the pair $(S,m)$ is {\it reasonable} if for every string
$s \subset m$, either $s \cap S = \emptyset$, or $s \subset S$. If
$s \subset m$ is contained inside $S$, we say that $S$
\defit{matches $s$ above} if $r_{s,u}$ is a modified row of $S$.
Otherwise we say that $S$ {\it continues above} $s$.

\begin{lem}\label{L:Sminitial}
Let $(S,m)$ be any initial pair.  If a modified row of $S$
contains a cell of $m$, then that row intersects the initial string of $m$.  If
a modified row of $S$ is a negatively modified row of $m$,
then $S$ intersects the initial string $s \subset m$.  In particular, if $(S,m)$ is reasonable, only the
initial string $s \subset m$ can be matched above.
\end{lem}
\begin{proof}
Follows immediately from the definition of column moves and the
fact that $\rs(\mu)/\rs(\lambda)$ is a
horizontal strip.
\end{proof}

\begin{lem}
Suppose $(S,m)$ is reasonable. Then every modified row $r$ of $S$
which contains a cell in $m$ is a positively modified row of $m$.
\end{lem}

\begin{lem}\label{L:Smrow}
Suppose $(S,m)$ is reasonable.  If $s \nsubseteq S$ then $S$ does
not contain a cell in row $r_{s,d}$.
\end{lem}

\begin{proposition} \label{P:Smcolreasonable}
Let $S$ be a maximal strip and $m$ a column move. Then $(S,m)$ is
reasonable.
\end{proposition}

\subsection{Normality}
Let $s \subset m$ be the initial string of $m$. We say that $(S,m)$
is \defit{normal}, if it is reasonable and in the case that $S$
continues above $s$ then (a) none of the modified rows of $S$
contain cells of $s$ (and by Lemma \ref{L:Sminitial}, none of the
modified rows of $S$ contain cells of $m$) and (b) the negatively
modified row of $s$ is not a modified row of $S$.

\begin{proposition}\label{L:Smreasonablenormal}
Let $S$ be a maximal strip and $m$ any column move.  Then $(S,m)$ is
normal.
\end{proposition}
\begin{proof} The pair $(S,m)$ is reasonable by Proposition
\ref{P:Smcolreasonable}. Suppose $S$ continues above $s = \{a_1,
a_2, \dotsc, a_\ell\}$, where the cells are indexed by decreasing
diagonal index. By Lemma~\ref{L:Sminitial}, normality cannot be
violated if $s$ is not the initial string of $m$, so we suppose $s$
is the initial string of $m$. Since $S$ does not match $s$ above by
definition, the row $r_\ell$ containing $a_\ell$ is not a modified
row of $S$.  The claim is thus trivial if $\ell= 1$ so we assume
$\ell > 1$.  Suppose $r_\ell$ contains $p \geq 1$ cells of $S$,
implying that the $p$ leftmost cells of $r_\ell$ are moved when
going from $\bdy \lambda$ to $\bdy \mu$ (and none of the columns of
these $p$ cells are modified columns of $S$).
  It follows from Property~\ref{L:innerstrip} and
$\rs(\lambda)_{r_\ell} < \rs(\lambda)_{r_\ell^-}$ that
$\lambda_{r_\ell^-} \geq \lambda_{r_\ell} + p + 1$.  In particular,
there is an addable corner $b^*$ on row $r_\ell$ of $\mu$.

It is easy to see that the row $r_{\ell-1}$ containing $a_{\ell-1}$
contains at least $p$ cells of $S$, with equality if and only if
$r_{\ell-1}$ is not a modified row of $S$.  If $r_{\ell-1}$ is a
modified row of $S$, then the addable corner $b^*$ will be an upper
augmentable corner for $S$, contradicting maximality and
Proposition \ref{P:maxstrip}.  So $r_{\ell-1}$ is not a modified row
of $S$.

Since $\lambda_{r_\ell^-} \geq \lambda_{r_\ell} + p + 1$, we get by Lemma~\ref{L:comparelengths}
that $\lambda_{r_{\ell-1}^-} \geq \lambda_{r_{\ell-1}} + p + 1$, so that
row $r_{\ell-1}$ of $\mu$
also has an addable corner. Continuing as before we see that $(S,m)$ is normal.
\end{proof}

\begin{lem}\label{L:normal}
Suppose $(S,m)$ is normal and let $s \subset m$ be any string such
that $S$ continues above $s$. Then $S$ contains the same number of
cells in each row $r$ containing a cell of $s$, and also the same
number of cells in the negatively modified row of $s$. Furthermore,
if $s$ is the initial string of $m$, then each such row $r$ has a
$\mu$-addable corner.
\end{lem}
\begin{proof}
The first statement follows easily from the definition of normality.
The last statement is proven as in  Proposition~\ref{L:Smreasonablenormal}.
\end{proof}

\subsection{Contiguity}
Suppose $(S,m)$ is reasonable.  We say that $(S,m)$ is contiguous if
there is a cell $b \in \bdy \mu \cap \bdy \nu$ which is not present
in $\bdy (\mu \cup \nu)$.  Call such a $b$ a {\it disappearing cell}.

\begin{lem}\label{L:disappearingcol}
Suppose $b=(r,c)$ is a disappearing cell. Then
\begin{enumerate}
\item
Row $r$ is a positively modified row of $m$ and does not contain cells
of $S$,
\item
Column $c$ is a modified column of $S$.
\item
Row $r$ contains a lower augmentable corner for $S$.
\end{enumerate}
\end{lem}
\begin{proof}
Suppose row $r$ contains $p\ge1$ cells of $S$.
Then $m$ must contain cells in column $c$. Let $b'$ be the cell below
$\Bot_c(\bdy\nu)$, $R=\row(b')$, and $h=\cs(\nu)_c$.
We have the sequence of inequalities:
\begin{align*}
\rs(\lambda)_R - \rs(\mu)_r &=  (h+ \rs(\lambda)_R) - (h-1+\rs(\mu)_r)-1  \\
  & \leq k+1 - (h-1+\rs(\mu)_r)-1 \\
  & \leq k+1 -h_{\nu \cup \mu}(b) -1 \\
  & \leq -1
\end{align*}
which contradicts the fact that
$\rs(\mu)/\rs(\lambda)$ is a horizontal strip.
Thus row $r$ contains no cells of $S$ and contains exactly one
cell $x\in m$.  Also column $c$ exactly one cell $a\in S$ and no cells of $m$.
Thus $h_\la(b) = k - 1$.

Suppose $r$ is not a positively modified row of $m$.  Then $m$
contains a cell $a^*$ in the column of the cell $b^*$ immediately left of $b$, and we
have $h_\la(b^*) = k$.  But $a^*$ is in the same row as $a$, so
$a^* \in S$ as well.  But by reasonableness, $x\in S$, a contradiction.  This proves (1).

Suppose $c$ is not a modified column of $S$. Then there is a cell $x'\in S$
contiguous to and below $a\in S$. Considering hook lengths we conclude
that $x'$ is just below $x$, $S$ removes the cell $b'$ just below $b$,
$h_\la(b')=k$, and $\la_r=\la_{r^-}$. But since $m$ is $\la$-addable
it follows that $x'\in m$. But then row $r^-$ must be positively modified by $m$,
contradicting $h_\la(b')=k$. This proves (2) and that $x$ is $\la$-addable.

For (3), the cell $x$ is a lower augmentable corner for $S$.
\end{proof}

\begin{cor} \label{C:maxcolnotcontig}
Suppose $S$ is a maximal strip and $m$ any move.  Then $(S,m)$ is
not contiguous.
\end{cor}

\subsection{Interference of strips and column moves}
Suppose that $(S,m)$ is normal and non-contiguous. Define
$\change_\rs(S) = \rs(\mu) -\rs(\la)$ and $\change_\rs(m) =
\rs(m*\la)-\rs(\la)$. Similarly define $\change_\rs(s)$ for a
column-type string $s$.  Thus
$$
\rs(\mu) -\rs(\la) + \rs(m*\la) = \change_\rs(S) + \change_{\rs}(m)
+ \rs(\la).
$$
Recalling Notation \ref{N:arrow} let
\begin{align}\label{E:mprimecol}
m' &= \{\text{strings } s \subset m \mid \text{$s$ and $S$ are not
matched above}\} \subset m \\
\label{L:mpluscol}
  m^+&=\,\rightarrow_S\!(m').
\end{align}
By Lemma \ref{L:Sminitial}, the set $m'$ is obtained from $m$ by
possibly removing the initial string of $m$. Define the vector
$\change_{\rs}(m')$ by considering only the modified rows of strings
inside $m'$.

If $m'\ne\emptyset$ we say that $(S,m)$ is \defit{non-interfering}
if $\rs(\la) +\change_\rs(S) + \change_\rs(m')$ is a partition and
\defit{interfering} otherwise. If $m'=\emptyset$ we say that $(S,m)$
is non-interfering if $\rs(\mu)/\rs(\nu)$ is a horizontal strip and
interfering otherwise (observe that $\rs(\la) +\change_\rs(S) +
\change(m')=\rs(\la) +\change_\rs(S)=\rs(\mu)$ is always a partition
in that case). The latter case is referred to as \defit{special
interference}.

\begin{lem}\label{L:pushbumpcol}
The set of cells $m^+$ satisfies all the conditions for a move on
$\mu$ except that $(m^+)*\mu$ may not be a $k$-shape.  Furthermore,
we have $\rs((m^+)*\mu) = \rs(\la) +\change_\rs(S) +
\change_\rs(m')$.  In particular, $(m^+)*\mu$ is always a $k$-shape
when $(S,m)$ is non-interfering.
\end{lem}
\begin{proof}
The proof is similar to Lemma \ref{L:pushbump}, except that we now
use Lemmata \ref{L:normal} and \ref{L:continuebelowcol}.
\end{proof}

\begin{lem} \label{L:continuebelowcol}
Suppose $S$ continues above the first string
$s_1 = \{a_1,a_2,\ldots,a_\ell\}$ of  $m$.  For
each $i \in [1,\ell]$ let $r_i$ be the row containing $a_i$. Then
there is an addable corner of $\mu$ in row $r_i$.  Moreover, the
addable corner of $\mu$ in row $r_i$ does not lie above a cell of $S$.
\end{lem}
\begin{proof}
Consider the case $i = \ell$ and set $r=r_{\ell}$.
Since $S$ does not match $s_1$ above by definition,
the row $r_\ell$ containing $a_\ell$ is not a modified row of
$S$ by normality.
Suppose
$r_\ell$ contains $p \geq 1$ cells of $S$, implying that the $p$
leftmost cells of $r_\ell$ are moved when going from $\bdy \lambda$
to $\bdy \mu$ (and none of the columns of these $p$ cells are modified columns of $S$).
  It follows from Property~\ref{L:innerstrip} and
$\rs(\lambda)_{r_\ell} < \rs(\lambda)_{r_\ell^-}$ that
$\lambda_{r_\ell^-} \geq \lambda_{r_\ell} + p + 1$.  In particular,
there is an addable corner in row $r_\ell$ of $\mu$ and it does not
lie above a cell of $S$.

Since $s_1 = \{a_1,a_2,\ldots,a_\ell\}$ is $\la$-addable,
 Lemma~\ref{L:comparelengths} ensures that
$\lambda_{r_i^-} \geq \lambda_{r_i} + p + 1$ for all $i$.
There are exactly $p$ cells of $S$ in
row $r_i$ for all $i$ by Lemma~\ref{L:normal}.  So again
there is an addable corner in row $r_i$ of $\mu$ and it does not
lie above a cell of $S$.
\end{proof}

\subsection{Column-type pushout: non-interfering case}
\label{SS:columnpushnointerfere} Suppose $(S,m)$ is normal,
non-contiguous, and non-interfering. In this case, by definition
$(S,m)$ is declared to be compatible where we set $\eta=(m^+)*\mu$
and define $(\tS,\tm)$ by \eqref{E:pushdiag} and the pushout of
$(S,m)$ by \eqref{E:pushoutdef}. $\tm$ is a (possibly empty) column
move and $\tS$ is a strip by Lemma \ref{L:pushbumpcol} and
Proposition \ref{L:etastripcol}.

\begin{proposition}\label{L:etastripcol}
Suppose $(S,m)$ is normal, non-contiguous, and non-interfering.
Then $\eta/\nu$ is a strip.
\end{proposition}
\begin{proof}
That $\eta/\nu$ is a horizontal strip is not difficult
(Lemma~\ref{L:continuebelowcol} ensures that the cells of $m^+$
do not lie above cells of $S$).   We also
have $\cs(\eta)/\cs(\nu) = \cs(\mu)/\cs(\lambda)$.

If $m'=\emptyset$ we have by definition that if $(S,m)$ is non-interfering
then $\rs(\eta)/\rs(\nu)=\rs(\mu)/\rs(\nu)$ is a horizontal strip.  Thus
$\eta/\nu$ is a strip.

Suppose $m'$ is not empty and that $\rs(\eta)=\rs(\la) +\change_\rs(S) +
  \change_{\rs}(m')$ is a partition.
We must prove that $\rs(\nu)_{r^-} \geq \rs(\eta)_r \geq \rs(\nu)_r$ for each row
$r$.
Recall that modified rows of $m'$ are modified rows of $m$ and that the only
string that may possibly be in $m\setminus m'$ is the initial one.
Therefore
the second inequality follows from the fact that
if $m\setminus m'$ is not empty then
the positively modified row of the initial string of
$m$ is a modified row of $S$.

To prove the first inequality, observe that
$$
\rs(\nu)_{r^-}-\rs(\eta)_r=\rs(\lambda)_{r^-}-\rs(\mu)_r+\Delta_{\rs}(m)_{r^-}
-\Delta_{\rs}(m')_{r} \, ,
$$
with $\rs(\lambda)_{r^-}-\rs(\mu)_r \geq 0$ since $S$ is a strip.

Suppose that $m=m'$.
Then the first inequality can
only fail if $r^-$ is the uppermost negatively modified row of $m$ or if
$r$ is the positively modified row of the initial string of $m$.

Suppose that $m$ is non-degenerate.
Let $r^-$ be the uppermost negatively modified row of $m$.  By
normality we have $\rs(\lambda)_{r^-}=\rs(\mu)_{r^-}$.  Therefore if the first
inequality fails, we have $\rs_{r^-}(\mu)=\rs_{r^-}(\lambda) = \rs_r(\mu)$
which is a contradiction since $\eta$ would not then be a $k$-shape
($r^-$ is a negatively modified row of $m'$ that is not a modified row of $S$
by normality).  Let $r$ be the positively modified row of the initial string
of $m$. By
normality we have $\rs(\lambda)_{r}=\rs(\mu)_{r}$.  Therefore if the first
inequality fails, we have $\rs(\lambda)_{r^-}=\rs(\mu)_r = \rs(\lambda)_r$
which is a contradiction since $\lambda$ would not then be a $k$-shape
($r^-$ is a negatively modified row of $m$).

Suppose that $m$ is degenerate, and let $r^-$ be the uppermost negatively
modified row of $m$ (and thus $r$ is the positively modified row of the initial string
of $m$).  By
normality we have $\rs(\lambda)_{r^-}=\rs(\mu)_{r^-}$
and $\rs(\lambda)_{r}=\rs(\mu)_{r}$.  Therefore if the first
inequality fails, we have $\rs(\lambda)_{r^-}\leq \rs(\mu)_r +1= \rs(\lambda)_r+1$
which is a contradiction since $\lambda$ would not then be a $k$-shape
($\rs(\lambda)_{r^-}-\rs(\lambda)_r \geq 2$ since
$m$ is degenerate).

Finally, suppose that $m$ and $m'$ are distinct.  The only case to consider
that was not considered in
the case $m=m'$ is when $r^-$ is the negatively modified row of the first
string
of $m$.  By hypothesis $m'$ is not empty and so $r$ is also a negatively
modified row of $m'$.  The first inequality then follows immediately.
\end{proof}

\subsection{Column-type pushout: interfering case}
\label{SS:columnpushinterfere} Suppose $(S,m)$ is normal,
non-contiguous, and interfering. We say that $(S,m)$ is \textit{
pushout-perfectible} if there is a set of cells $m_\comp$ outside
$(m^+)*\mu$ so that if
\begin{align} \label{E:etanotationcol}
  \eta = ((m^+)*\mu) \cup m_\comp
\end{align}
then $\eta/\nu$ is a strip and $\eta/\mu$ is a column move from
$\mu$ whose strings have the same diagram as those of $m$.  Since
$\rs(\eta)/\rs(\nu)$ is a horizontal strip, $m_\comp$ can only be a
single column-type string and will thus be unique if it exists.

In the case that $(S,m)$ is pushout-perfectible, by definition we
declare $(S,m)$ to be compatible where $\eta$ is specified by
\eqref{E:etanotationcol} and define $(\tS,\tm)$ by
\eqref{E:pushdiag} and the pushout of $(S,m)$ by
\eqref{E:pushoutdef}. By definition, $\tm$ is a column move and
$\tS$ is a strip.

\begin{example} This is an example of special interference for $k=3$.
In $\mu=S*\la$ and $\nu=m*\la$ the new cells added to $\la$ are shaded.
In the lower right $k$-shape the cells of $m_\comp$ are shaded.
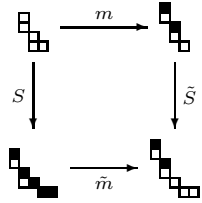
\begin{equation*}
\begin{diagram}
\node{{\tableau[pby]{\bl \\ \\ \\  \bl &  \\ \bl &  & }}} \arrow{s,t}{S} \arrow{e,t}{m}
\node{{\tableau[pby]{\fl \\ \\ \bl  & \fl \\  \bl &  \\ \bl & \bl & }}} \arrow{s,b}{\tS} \\
\node{{\tableau[pby]{\bl \\ \fl \\ \\ \bl & \fl  \\ \bl& & \fl \\ \bl&\bl&\bl&\fl&\fl}}} \arrow{e,b}{\tm}
\node{{\tableau[pby]{\fl \\  \\ \bl & \fl \\ \bl &   \\ \bl&\bl &  \\ \bl&\bl&\bl&&}}}
\end{diagram}
\end{equation*}
\end{example}

\begin{lemma} \label{L:specialint}
Let $(S,m)$ be such that $m'$ is empty.
 Then there is (special)
interference iff $\rs(\mu)_{r^-}=\rs(\mu)_{r}$, where $r^-$ is the negatively modified
row of $m$ ($m$ is necessarily of rank 1).
\end{lemma}
\begin{proof} Let $m'=\emptyset$. In this case there is interference iff
$\rs(\mu)/\rs(\nu)$ is not a horizontal strip. We have
$$
\rs(\nu)_{r^-}-\rs(\mu)_r=\rs(\lambda)_{r^-}-\rs(\mu)_r+\Delta_{\rs}(m)_{r^-} \, ,
$$
with $\rs(\lambda)_{r^-}-\rs(\mu)_r \geq 0$ since $S$ is a strip.
The inequality $\rs(\nu)_{r^-}-\rs(\mu)_r \geq 0$ can thus only fail when
$r^-$ is the negatively modified row
of $m$.  In that case, by normality we have
$\rs(\lambda)_{r^-}=\rs(\mu)_{r^-}$.  Therefore we obtain
$$
\rs(\nu)_{r^-}-\rs(\mu)_r=\rs(\mu)_{r^-}-\rs(\mu)_r-1 < 0 \iff \rs(\mu)_{r^-}=\rs(\mu)_r
$$
and the lemma follows.
\end{proof}

\begin{proposition}\label{L:maximalinterferecol}
Suppose $(S,m)$ is interfering, $S$ is maximal, and $m$ is a column
move. Then $(S,m)$ is pushout-perfectible (and hence compatible).
Moreover $m_\comp$ consists of a single string lying in the same
columns as the last string of $m$.
\end{proposition}
\begin{proof} By Proposition \ref{L:Smreasonablenormal}
and Corollary \ref{C:maxcolnotcontig}, $(S,m)$ is normal and not
contiguous, so that it makes sense to refer to interference.

Suppose $m^{+}$ is non-empty so that $\rs(\la) +\change_\rs(S) + \change_\rs(m')$ is not a partition.
Let $r^-$ be the negatively modified row of the final string $s$ of
$m$. We may assume that $r^-$ is not a modified row of $S$, for
otherwise $s$ would have to be initial, and Lemma~\ref{L:Smrow}
would imply that $s \subseteq S$ and thus that $s$ is matched above,
implying that $(S,m)$ does not interfere.
Also, $r$ must be a modified row of $S$ for interference to occur.
Note that since $\rs(\la) +\change_\rs(S) + \change_\rs(m')$ is not
a partition we have $\rs(\mu)_{r^-}=\rs(\mu)_{r}$ and thus
$\rs(\lambda)_{r^-}=\rs(\mu)_{r}$ ($r^-$ is not a modified row of $S$).

We claim that $\eta$ has an addable corner directly above the first
cell $a$ of $s$. Since $h_\la(\Left_{r^-}(\bdy\la))=k$ by the
definition of a move, we have from $\rs(\lambda)_{r^-}=\rs(\mu)_{r}$
that $h=h_\mu(\Left_r(\bdy\mu))$ is $k-1$ or $k$. In either case
(using Lemma~\ref{L:everywhere}(3) when $h=k$) we see that
$\Left_{r^-}(\bdy\la)$ lies in the same column as
$\Left_r(\bdy\mu)$. We then have immediately that there is an
addable corner directly above the first cell $a$ of $s$. Since $s$
is $\la$-addable, we obtain from Lemma~\ref{L:stripcomparelengths}
that there is a $\mu$-addable corner above every cell of $m$. The
rest of the proof that $m^+ \cup m_{\comp}$ is a column move is
analogous to the proof of Lemma \ref{L:columncompletionaugment}.

Suppose $m^+$ is empty and
$\rs(\eta)/\rs(\nu)=\rs(\mu)/\rs(\nu)$ is not a horizontal strip.
Recall that only the initial string of $m$ can disappear and thus
$m$ is of rank 1.  From Lemma~\ref{L:specialint},
the negatively modified row of $m$ is in a row $r^-$
such that row $r$ is a modified row of $S$ with
$\rs(\la)_{r^-}=\rs(\mu)_{r^-}=\rs(\mu)_{r}$ (recall that
$\rs(\lambda)_{r^-}=\rs(\mu)_{r^-}$ by normality).
Again  $\eta$ has an addable
corner directly above the first cell $a$ of $s$ by Lemma~\ref{L:everywhere}(3).
The rest of the proof that $m_{\comp}$ is a column move
is as in the non-empty case.

Since the cells of $m_{\comp}$ lie above cells of $m$ we have that
$\eta/\nu$ is a horizontal strip.  Obviously $\cs(\eta)/\cs(\nu)=
\cs(\mu)/\cs(\la)$ is a vertical strip.  We thus only have to prove
that $\rs(\eta)/\rs(\nu)$ is a horizontal strip. If $m'$ is
non-empty there is interference only if $\rs(\la) +\change_\rs(S) +
\change_\rs(m')$ is not a partition.   Following the proof of
Proposition~\ref{L:etastripcol} we have that $\rs(\nu)_i \geq
\rs(\eta)_{i+1}$ for all $i$ except possibly when $i=R$ is the
highest positively modified row of $m$. In that case
$\rs(\eta)/\rs(\nu)$ is a horizontal strip since the positively
modified row $R^+$ of $m_{\comp}$ lies in the row above row $R$ and
$\rs(\mu)/\rs(\lambda)$ is a horizontal strip by definition (that
is, given $\lambda_R \geq \mu_{R^+}$, $\eta_{R^+}=\mu_{R^+}+1$ and
$\nu_R=\lambda_R+1$, we have $\nu_R\geq \eta_{R^+}$). If $m'$ is
empty, then by Lemma~\ref{L:specialint} there is interference iff
$\rs(\mu)_i=\rs(\mu)_{i+1}$, where $i=r^-$ is the negatively
modified row of $m=s$.  In that case,  given $\eta_{r}=\mu_{r}-1$
and $\nu_{r^-}=\la_{r^-}-1$, the fact that $\la_{r^-} \geq \mu_{r}$
guarantees that $\nu_{r^-} \geq \eta_{r}$.  So we only have to check
what happens at the positively modified row of $m_{\comp}$. The
result follows just as in the $m' \neq \emptyset$ case.
\end{proof}

\begin{lemma} \label{L:liesabove}
Suppose $(S,m)$ is pushout-perfectible. If any cell of the string
$s=m_{\comp}$ lies above a cell of $S$ then $s$ lies in the same
columns as the final string of $m$.
\end{lemma}
\begin{proof} Let $a$ be a cell of $s$ that lies above a cell of $S$.
From the definition of $s=m_{\comp}$, there is a cell of the final string $t$ of
$m$ in the row below that of $a$.  Hence, since $S$ is a horizontal strip,
$a$ also lies above a cell of $t$.  Finally, since $s$ and $t$ are translates
the lemma follows.
\end{proof}

\begin{lem}\label{L:upperbelow}
Suppose $(S,m)$ is pushout-perfectible. If the string $s=m_\comp$
does not lie in the same column as the last string of $m$ then the
first cell $a$ of $s$ is an upper augmentable corner of $S$.
\end{lem}
\begin{proof}
By the definition of interference,
$S$ modifies the row $r$ above the highest negatively modified
row of $m$. From the hypotheses, $h_\mu(r,\col(a))=k$, so that
the cell $a$ is contiguous and above a cell of $S$ in row $r$.
By Lemma \ref{L:liesabove} $a$ does not lie above a cell of $S$
and is therefore an upper augmentable corner of $S$.
\end{proof}
\subsection{Alternative description of pushouts (column moves)}
Let $(S,m)=(\mu/\lambda,\nu/\lambda)$ be any initial pair where $m =
s_1 \cup \cdots \cup s_r$ is a column move.
 Let $m'$ be the collection of cells
obtained from $m$ by removing $s_i$ whenever the positively modified
row of $s_i$ is a modified row of $S$.  It is easy to see that
$m'$ is of the form $s_1 \cup s_{2} \cup \cdots \cup s_r$
or $s_2 \cup \cdots \cup s_r$.
Suppose that $\change(s_1)$ affects
rows $c$ and $c + d$.  If $\alpha$ is not a partition, we suppose
that $\alpha_i + 1 = \alpha_{i+1} > \alpha_{i+2}$.
We say that there is interference
if $\alpha$ is not a partition or if
$m'$ is empty and $\alpha_i=\alpha_{i+1}$, where $i=c$ is the negatively
modified row of $s_1$.
Then the perfection of $\alpha$ with respect to
$(S,m)$ is the vector
$$
\per_{S,m}(\alpha) = \begin{cases} \alpha + e_{i+d+1}
- e_{i+1} & \mbox{if there is interference}
 \\
\alpha & \mbox{otherwise}\end{cases}
$$
The
{\it expected row shape} $\ers(S,m)$ of $(S,m)$ is defined to be
$$
\ers(S,m) = \per_{S,m}(\rs(\la) +\change_\rs(S) + \change_\rs(m')).
$$

\begin{prop}\label{P:pushcritcol}
Let $(S = \mu/\lambda, m = \nu/\lambda)$ be an initial pair with $m$
a nonempty column move. Suppose there exists a $k$-shape $\eta$ such
that
\begin{enumerate}
\item $\rs(\eta) = \ers(S,m)$.
\item $\eta/\mu$ is either empty or a column move
whose strings are translates of those of $m$.
\item $\nu \subset \eta$.
\end{enumerate}
Then $(S,m)$ is compatible and $\push(S,m) = (\eta/\nu,\eta/\mu)$.
In particular $\eta/\nu$ is a strip.
\end{prop}
\begin{proof}
It is easy to see that $\eta/\mu$ decomposes into column type
strings as $m'' \cup m_\comp$ where $\cs(m'' * \mu) = \cs(\mu)+
\change_\cs(m')$. The proof of reasonableness and non-contiguity of
$(S,m)$ is similar to the proofs in Proposition \ref{P:pushcrit}
with Lemma~ \ref{L:disappearing} replaced by
Lemma~\ref{L:disappearingcol}.

To prove normality, suppose the first string $s$ of $m$ is continued
above by $S$. Then $s \in m'$ and
since $m_{\rm comp}$ cannot affect the rows affected by the strings of $m'$,
there must exist a string of $\eta/\mu$ that is the rightward shift
of $s$. This implies that $S$ contains the same number
of boxes in each row $r$ containing a box of $s$ and also in the negatively
modified row of $s$.  Therefore, none of the rows containing a box of $s$
is a modified row of $S$ (given that the uppermost row containing a box of $s$
is by hypothesis not a modified row of $S$), and also the negatively modified
row of $s$ is not a modified row of $S$.

If $(S,m)$ is non-interfering then $\eta/\nu$ is a strip by
Proposition~\ref{L:etastripcol}.

If $(S,m)$ is interfering then $m_{\comp}$ is a single string $t$.
Suppose a cell $x$ of $t$ lies above a cell $y$ of $S$.  Since $S$
is a horizontal strip  $y$ is $\lambda$-addable, and thus by
hypothesis $y$ is also a cell of $m$ ($t$ is a translate of the
strings of $m$ and it starts one row above the final string of $m$).
Therefore $\eta/\nu$ is a horizontal strip. Obviously
$\cs(\eta)/\cs(\nu)=\cs(\mu)/\cs(\lambda)$ is a vertical strip so it
only remains to show that $\rs(\eta)/\rs(\nu)$ is a horizontal
strip. This is done as in the proof of
Proposition~\ref{L:maximalinterferecol}.
\end{proof}

\section{Pushout sequences}
\label{sec:pushaug}

Consider an initial pair $(S,\bp)$ consisting of a strip $\mu/\la$ for $\la,\mu\in\Ksh$ and a path $\bp$ from $\la$ to $\nu\in\Ksh$.
A \textit{pushout sequence} from $(S,\bp)$ is a sequence of augmentation moves and pushouts
which produces a final pair $(\tS,\bq)$ consisting of a \textit{maximal} strip $\tS=\eta/\nu$ and a path $\bq$ from $\mu$ to $\eta$
for some $\eta\in\Ksh$:
\begin{align} \label{E:pushseq}
\begin{diagram}
\node{\la} \arrow{e,t}{\bp} \arrow{s,t}{S} \node{\nu} \arrow{s,b}{\tS} \\
\node{\mu} \arrow{e,b}{\bq} \node{\eta}
\end{diagram}
\end{align}
More precisely, a pushout sequence is defined by a diagram of the form
\begin{align}\label{E:pushseqdiag}
\begin{diagram}
  \node{\la^0} \arrow{e,t}{m^1} \arrow{s,t}{S^0} \node{\la^1} \arrow{e,t}{m^2}   \arrow{s,t}{S^1}  \node{\dotsm} \arrow{e,t}{}
\node{\la^{L-1}} \arrow{e,t}{m^L} \arrow{s,t}{S^{L-1}} \node{\la^L} \arrow{s,t}{S^L} \\
  \node{\mu^0} \arrow{e,b}{n^1} \node{\mu^1} \arrow{e,b}{n^2}\node{\dotsm} \arrow{e,t}{} \node{\mu^{L-1}} \arrow{e,b}{n^L} \node{\mu^L}
\end{diagram}
\end{align}
where $\la^0=\la$, $S=S^0$, the top row of \eqref{E:pushseqdiag}
consists of the path $\bp$ (possibly with empty moves interspersed),
each $S^i$ is a strip with $\tS=S^L$ maximal, the $n^i$ are
(possibly empty) moves, the bottom row of \eqref{E:pushseqdiag} is
the path $\bq$, and for each $1\le i\le L$, the diagram
\begin{align}
\begin{diagram}
\node{\la^{i-1}} \arrow{e,t}{m^i} \arrow{s,t}{S^{i-1}} \node{\la^i} \arrow{s,b}{S^i} \\
\node{\mu^{i-1}} \arrow{e,b}{n^i} \node{\mu^i}
\end{diagram}
\end{align}
defines an augmentation move if $m^i$ is empty, or the pushout of a compatible pair if $m^i$ is not empty.

The main technical work in this paper is to establish the following existence and uniqueness
properties of pushout sequences.

\begin{prop} \label{P:canonicalpushseq} Each initial pair $(S,\bp)$
admits a canonical pushout sequence, which repeatedly maximizes the
current strip and pushes out the resulting maximal strip with the
next move, and ends with maximization.
\end{prop}

We prove Proposition \ref{P:canonicalpushseq} in Subsection
\ref{SS:canonical pushout sequence} by giving an algorithm which
computes the canonical pushout sequence.

\begin{prop} \label{P:pushseqequiv}
Pushout sequences take equivalent paths to equivalent paths. That
is, if $(S,\bp)$ and $(S,\bp')$ are initial pairs with $\bp\equiv
\bp'$ and there are pushout sequences from $(S,\bp)$ and $(S,\bp')$
that produce the final pairs $(\tS,\bq)$ and $(\tS',\bq')$
respectively, then $\tS=\tS'$ and $\bq\equiv\bq'$.
\end{prop}

It follows that pushout sequences define a map
$(S,[\bp])\to(\tS,[\bq])$ where $(S,\bp)$ is an initial pair and
$(\tS,\bq)$ is a final pair with $\tS$ maximal, fitting the diagram
\eqref{E:pushseq}.

The special case $\bp'=\bp$ of Proposition \ref{P:pushseqequiv} is
proved in Subsection \ref{SS:Sectionmaxcomp}. The general case is
proved in Section \ref{sec:commutingcube}.

\subsection{Canonical pushout sequence}
\label{SS:canonical pushout sequence} The following algorithm
\texttt{PushoutSequence} produces a canonical pushout sequence from
$(S=\mu/\la,\bp)$. It suffices to produce the path $\bq$, as the
output strip $\tS$ is defined by the last elements of $\bp$ and
$\bq$. We may assume that $\bp=(\la=\la^0,\la^1,\dotsc,\la^L)$ has
no empty moves and $m^i$ is the move from $\la^{i-1}$ to $\la^i$.
Let
$$\texttt{PushoutCompatiblePair}(\rho,\la^{i-1},\la^i)$$ compute the
following pushout and return $\eta$
\begin{align}
\begin{diagram}
\node{\la^{i-1}} \arrow{e,t}{m^i} \arrow{s,t}{} \node{\la^i} \arrow{s,b,..}{} \\
\node{\rho} \arrow{e,b,..}{} \node{\eta}
\end{diagram}
\end{align}
as specified in Subsections \ref{SS:rowpushnointerfere} and \ref{SS:rowpushinterfere} if $m^i$ is a row move and
\ref{SS:columnpushnointerfere} and \ref{SS:columnpushinterfere} if $m^i$ is a column move.

\begin{tabbing}
xxxx\=xxxx\=xxxx\=xxxx\=\kill
\textbf{proc} \texttt{PushoutSequence}($\mu,\la,p$): \\
\>  \text{\textbf{local} $q := (\mu)$, $q'$} \\
\>  $\rho := \mu$ \\
\>  \textbf{for} $i$ \textbf{from} $1$ to \textbf{length}(p): \\
\> \>    $q' = \texttt{MaximizeStrip}(\rho,\la^{i-1})$ \\
\> \>    extend $q$ by $q'$ \\
\> \>    $\rho := \mathbf{last}(q')$ \\
\> \>    $\rho := \texttt{PushoutCompatiblePair}(\rho,\la^{i-1},\la^i)$ \\
\> \>    append $\rho$ to $q$ \\
\>   $q' := \texttt{MaximizeStrip}(\rho,\la)$ \\
\>   extend $q$ by $q'$ \\
\>    \textbf{return} $q$
\end{tabbing}
This procedure builds up a path $q$, implemented as a list of
shapes. The variable $q$ is initialized to be the list with a single
item $\mu$. For each move $m^i$ in $p$, the current strip is
maximized. By Propositions \ref{L:maximalinterfere} and
\ref{L:maximalinterferecol}, the resulting initial pair is
compatible and hence its pushout with the current move is
well-defined. The output strip (given by the last shapes in $q$ and
$p$ respectively) is maximal due to the last invocation of
\texttt{MaximizeStrip}. The ``extension" step takes the path $q$,
given as a list of $k$-shapes, and extends it by the path $q'$. Note
that the last element of $q$ equals the first element of $q'$.

\subsection{Pushout sequences from $(S,p)$ are equivalent} \label{SS:Sectionmaxcomp}
In this subsection we prove the following result, which is the
$\bp=\bp'$ case of Proposition \ref{P:pushseqequiv}.

\begin{proposition} \label{P:commuaugmen}
Let $S=\mu/\la$ be a strip and $\bp$ a path in $\Ksh$ from $\la$ to
$\nu$. Then any two pushout sequences from $(S,\bp)$ produce the
same strip and equivalent paths.
\end{proposition}

We shall reduce the proof of Proposition \ref{P:commuaugmen} to that
of Proposition \ref{L:commutingaugment} and then use the rest of the
subsection to prove the latter.

Consider the setup of Proposition \ref{P:commuaugmen}. By induction
on the number of moves in $\bp$ we may assume that $\bp=m$ is a
single move. We may assume that one of the pushout sequences to be
compared, is the canonical one, which first passes from the strip
$S$ to its maximization $S_{\max}$ by the augmentation path $\br$,
then does the pushout $\push(S_{\max},m)=(S'',\tm)$, and finally
maximizes the resulting strip via the augmentation path $\brt$,
resulting in the maximal strip $\tS$ and the path $\bq=\brt\tm\br$.

Consider any other pushout sequence from $(S,m)$, which produces
$(\tS',\bq')$, say. Suppose the first operation in this pushout
sequence is an augmentation move $m'$. The move $m'$ is the first in
the output path $\bq'$; let the path $\bqt$ be the rest of $\bq'$.
Let $\bt$ be any augmentation path from $m'\cup S$ to its
maximization. By Proposition \ref{P:maxstripunique}, this
maximization is equal to $S_{\max}$ and $\bt m' \equiv \br$. We have
\begin{align}
  \bq' = \bqt m' \equiv \brt \tm \bt m' \equiv \brt \tm \br = \bq,
\end{align}
which holds by induction since $\bqt$ and $\brt\tm \bt$ are
equivalent, being produced from the same pair $(m'\cup S,m)$ by
pushout sequences, with $m'\cup S$ closer to maximal than $S$.

We may therefore assume that the first operation in the pushout
sequence producing $(\tS',\bq')$, is a pushout, and in particular
that $(S,m)$ is compatible. Let $\push(S,m)=(S',M)$. Writing
$\bq'=\bqt M$, $\bqt$ is an augmentation path that maximizes $S'$
and produces $\tS'$.

We may also assume that $S$ is not already maximal, for otherwise
there is only one way to begin the pushout sequence from $(S,m)$.
Then $\br$ is nonempty; let its first move be $x$ and $\br'$ the
remainder of $\br$. Since $x$ is a move in the canonical
maximization of $S$, it is a maximal completion move that augments
$S$.

We apply Proposition \ref{L:commutingaugment}, using the label
$S'\cup \tx$ for the front right upward arrow. Let $\mathbf{y}$ be
an augmentation path that maximizes the strip $S'\cup\tx$. We have
\begin{align*}
  \bq' = \bqt M \equiv \mathbf{y}\tx M \equiv \mathbf{y} \tM x
  \equiv \brt \tm \br' x = \bq.
\end{align*}
The first equivalence holds by Proposition \ref{P:maxstripunique}
since both $\bqt$ and $\mathbf{y}\tx$ are maximizations of $S'$. The
second holds by the equivalence of the top face of
\eqref{E:cubeaugment} in Proposition \ref{L:commutingaugment}. The
third equivalence holds by induction since $\mathbf{y}\tM$ and $\brt
\tm \br'$ are equivalent, being the paths produced by two pushout
sequences from $(S\cup x, m)$ with $S\cup x$ closer to maximal than
$S$.

Thus we have reduced the proof of Proposition \ref{P:commuaugmen} to
that of Proposition \ref{L:commutingaugment}.

\begin{proposition}\label{L:commutingaugment}
Let $(S = \mu/\lambda, m = \nu/\lambda)$ be a compatible initial
pair with $\push(S,m) = (S',M)$ and let $x = \kappa/\mu$ be a
maximal completion move that augments $S$.  Then we have the
commuting cube
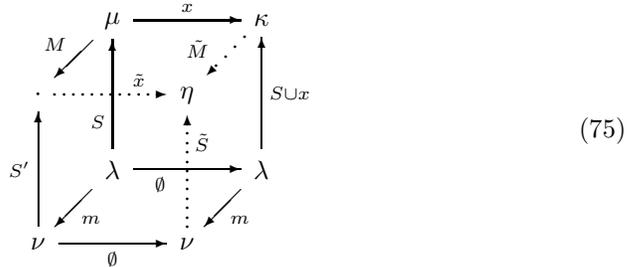
\begin{equation}\label{E:cubeaugment}
\begin{diagram} \node[2]{\mu}  \arrow{sw,t}{M} \arrow[2]{e,t}{x}
\node[2]{\kappa} \arrow{sw,t,3,..}{\tilde M} \\
\node{\cdot} \arrow[2]{e,t,3,..}{\tilde x}
\node[2]{\eta}\\
\node[2]{\lambda} \arrow[2]{n,l,1}{S} \arrow[2]{e,b,1}{\emptyset}
\arrow{sw,b}{m} \node[2]{\lambda} \arrow{sw,b}{m}
\arrow[2]{n,r}{S \cup x} \\
\node{\nu} \arrow[2]{e,b}{\emptyset} \arrow[2]{n,l}{S'}
\node[2]{\nu} \arrow[2]{n,r,3,..}{\tilde S}
\end{diagram}
\end{equation}
in which vertical edges are strips and other edges are moves, the
left and right faces are pushouts, the front and back faces are
augmentations, and the top face is an elementary equivalence.
\end{proposition}

\begin{lem}\label{L:maxcomplete}
Let $x$ be a maximal completion row move and $m$ a row move on
$\lambda$ such that $x$ and $m$ interfere and $x$ is above $m$. Then
$(x,m)$ is lower-perfectible.
\end{lem}
\begin{proof}
Let $a=(r,c)$ be the lowest cell of the initial string $t$ of $x$.
Then $c^-$ is a negatively modified column of $m$ and
$\cs(\lambda)_c = \cs(\lambda)_{c^-} - 1$.  We claim that
$b=\Bot_c(\bdy\la)$ and $b^-=\Bot_{c^-}(\bdy\la)$ are on the same
row. This follows from the estimate $h_\la(b) \geq h_\la(b^-) - 1 =
k - 1$ and the assumption that $t$ is maximal. It follows that there
is a $(m*\la)$-addable corner in the row containing $b^-$ and $b$ by
Remark~\ref{R:rowshift}. The rest of the proof is similar to Lemma
\ref{L:extendcompletion}.
\end{proof}

\begin{lem}\label{L:maxcompletecol}
Let $x$ be a maximal completion column move and $m$ a column move
from $\lambda$ such that $x$ and $m$ interfere and $x$ is below $m$.
Then $(x,m)$ is the transpose analogue of a lower-perfectible
interfering pair of row moves.
\end{lem}
\begin{proof}
The proof is similar to that of Lemma~\ref{L:maxcomplete}.
\end{proof}

\begin{lem}\label{L:maxcompleteequiv}
Let $(S = \mu/\lambda, m = \nu/\lambda)$ be a compatible initial
pair, $\push(S,m) = (S',M)$ and let $x = \eta/\mu$ be a maximal
completion move for $S$. Then $x$ and $M$ define an elementary
equivalence.
\end{lem}
\begin{proof}
By the definition of pushout, if $m$ is a row (resp. column) move
then $M$ is either a row (resp. column) move or empty. In some cases
it will be shown that $(S\cup x,m)$ is compatible. In that case we
write $\push(S\cup x,m)=(\tS,\tM)$.

$~$

\noindent{\it I) $m$ is a row move and $x$ a maximal completion row
move.}

Suppose first that $x$ and $M$ intersect.  Then the first string $s
\in x$ must intersect a string $t$ of $M$.  Since $x$ is maximal,
$M$ cannot continue below $x$.  Let us suppose that $M$ continues
above $x$.  Let $b'$ be the first cell in $s$; it is a lower
augmentable corner for some modified column $c$ of $S$ containing a
cell $a$. It follows that $t$ must contain a cell $b$ in column $c$
(on top of $a$). But $b \in S'$ as well and $S'$ is a horizontal
strip so $a \in m$.  It is clear that $t$ must be part of $m_\comp$,
but this contradicts Lemma \ref{L:mcompcol}.  Thus if $x$ and $M$
intersect, they satisfy an elementary row equivalence.

Now suppose that $x$ and $M$ interfere.  If $x$ is above $M$ then by
Lemma \ref{L:maxcomplete} a lower perfection $M_{\per}$ exists. Note
that we can then easily check that $(S \cup x, m)$ is compatible.
Then $\tilde M=M \cup M_{\per}$ and $\tilde S=S' \cup x \cup
M_{\per}$ (see \eqref{E:cubeaugment}). If $M$ occurs above $x$, let
$c$ be the leftmost positively modified column of $M$. By definition
of pushout, $c$ is not a positively modified column of $S$. And for
$x$ to be a completion move, $c^-$ needs to be a positively modified
column of $S$. Since $\cs(\mu)_c = \cs(\mu)_{c^-} - 1$ we thus  have
$\cs(\lambda)_c = \cs(\lambda)_{c^-}$. Therefore for $c$ to be a
positively modified column of $M$, all the negatively modified
columns of $x$ had to be positively modified columns of $m$.
Therefore, the strings of $m$ that are contiguous to strings of $x$
are all continued above and below in $S\cup x$.  Since $(S,m)$ is
compatible, this gives that $(S \cup x, m)$ is compatible with
$\tilde M=M \cup M_{\per}$ and  $\tilde S = S' \cup x \cup
M_{\per}$, where $M_{\per}$ is given by the strings of $m$ that are
contiguous to strings of $x$ pushed above one cell. It is then easy
to see that $(x,M)$ is upper perfectible by $M_{\per}$.


If $M$ is empty then it is easy to check that $(S\cup x, m)$ is
compatible (with pushout $(\tS,\tM)$, say) such that either $\tM$ is
empty or $\tM$ is a row move from $x*\mu$ and $\tx:=\tM \cup x$ is a
row move from $\mu$ with $\tM$ extending the strings of $x$ above.
Either way we obtain an elementary equivalence $\tx M \equiv \tM x$.

\noindent{\it II) $m$ is a row move and $x$ is a maximal completion
column move.}

Let $M$ and $x$ be intersecting moves. We show that $x$ continues
above and below $M$, so that $M$ and $x$ satisfy an elementary
equivalence. Since a row and a column move cannot be matched above
and cannot be matched below, it suffices to show that $M$ does not
continue above $x$ and does not continue below $x$. Suppose that $M$
continues above $x$. Let $b$ be the highest cell of $M \cap x$ and
$s$ the string in $M$ containing $b$. Let $a$ be the cell above and
contiguous to $b$ in $s$. By the maximality of $x$, $a$ has to lie
above a cell of $S$. So the string $s$ of $M$ was pushed above
during the pushout. Hence the cell below $b$ is also in $S$. This
contradicts the fact that $x$ is a completion move. Suppose that $M$
continues below $x$. Let $b$ be the lowest cell of $x$ and let $s$
be the string of $M$ containing $b$. The cell $b$ is an upper
augmentable corner for some modified row $R$ of $S$. Since $M$
continues below $x$, the string $s$ contains a cell $b'$ in row $R$.
The string $s$ of $M$ cannot have come from pushing above a string
of $m$ since the cell below $b$ is not in $S$ by definition of upper
augmentable corner. If $s \subset m$ then all cells of $S$ in row
$R$ are also in $m$ and thus row $R$ cannot be a modified row of $S$
by reasonableness. Thus $s\subset m_{\rm comp}$. Since the cell to
the left of $b'$ lies in $\mu$, it cannot also lie in $m_{\rm
comp}$. So all the cells outside of $\la$ and to the left of $b'$
lie in $m \cap S$.  But $(S,m)$ is reasonable, so row $R$ would not
be a positively modified row of $S$, a contradiction.

Suppose $M$ and $x$ do not intersect and are contiguous. In this
case $M$ is above $x$.  Let $b$ be the highest cell of $x$ and $a$
the cell of $M$ contiguous with $b$. By maximality of $x$, $a$ has
to lie above a cell $a'$ of $S$. But since this implies that the
string $s$ of $M$ that contains $a$ was pushed above during the
pushout, we have that the column of $a'$ is not a modified column of
$S$. Therefore there needs to be a cell of $S$ below $b$.  But this
is a contradiction to the fact that $x$ is a completion move.

Suppose $M$ is empty. In this case one may deduce that $(S\cup x,m)$
is compatible with $\tM$ empty.
Then $x$ and $M$ satisfy a trivial equivalence.

\noindent {\it III) $m$ is a column move and $x$ is a maximal
completion row move.}

Let $M$ and $x$ be intersecting.  By maximality of $x$, $x$
continues below $M$.  Suppose that $M$ continues above $x$. Let $b$
be the highest cell of $x$.  It is a lower augmentable corner of $S$
associated to a modified column $c$ of $S$.  Since $M$ continues
above, there is a cell $a$ of $M$ in column $c$ that lies above a
cell of $S$.  Suppose $a$ belongs to $m^+$. By
Lemma~\ref{L:continuebelowcol}, $a$ does not belong to the first
string of $m^+$ and so there is a cell of $m^+$ in the row below
that of $a$.  Given that $S$ is a horizontal strip, $a$ lies above a
cell of $m \cap S$ and so does $b$ by reasonableness and translation
of strings in a move.  But then we have the contradiction that $b$
lies above a cell of $S$. Therefore $a \in m_{\comp}$.  By
Lemma~\ref{L:liesabove} the cells of $m_{\comp}$ are in the same
column as the final string of $m$ and thus we get again the
contradiction that $b$ lies above a cell of $m \cap S$. Therefore if
$M$ and $x$ intersect $x$ continues above and below $M$.

Suppose $M$ and $x$ do not intersect and $M$ and $x$ are contiguous.
$M$ cannot be below $x$ due to the maximality of $x$. If $M$ is
above $x$ then a contradiction is reached as in the previous
paragraph.

Suppose $M$ is empty. Then $m$ is a single string that is matched
above by $S$ and $(S,m)$ is non-interfering. Since $x$ is a
completion row move for $S$, it follows that $m$ is matched above by
$S\cup x$ and $(S\cup x,m)$ is non-interfering. Therefore $(S\cup
x,m)$ is compatible with $\tM$ empty.

Hence $M$ and $x$ satisfy an elementary equivalence.

\noindent {\it IV) $m$ is a column move and $x$ a maximal completion
column move.}

Let $M$ and $x$ be intersecting. Suppose $M$ continues above $x$.
Let $b$ be the highest cell of $x$ and let $a$ be the cell of $M$
contiguous to it from above. By maximality of $x$, cell $a$ lies
above a cell of $S$. Therefore by Lemma~\ref{L:continuebelowcol} $a$
belongs to $m_{\comp}$ and we have as before the contradiction that
$x$ lies above a cell of $m$. Suppose $M$ continues below $x$. Let
$s$ be the string of $M$ that meets $x$. The bottommost cell $a$ of
$x$ is an upper augmentable corner of $S$ associated to row $R$,
say. The cell $a$ is also in $s$ and since $M$ continues below $x$,
$s$ has a cell $b$ contiguous to and below $a$. Since row $R$ has a
$\mu$-addable cell contiguous to $a$, we deduce that it is $b$. The
cell $b$ cannot belong to $m$ since $m$ is a vertical strip and
there are cells of $S$ to the left of $b$ since $R$ is a modified
row of $S$. Neither can $b$ belong to $m^+$ since in this case $R$
could not be a modified row of $S$ by normality. So $a$ and $b$
belong to $m_{\comp}$.  Since $m_{\comp}$ does not lie above the
last string of $m$ (otherwise there would be a cell of $S$ below
$a$), there is an upper augmentable corner of $S$ below $a$ by
Lemma~\ref{L:upperbelow}. This contradicts the assumption that $x$
is a maximal completion column move.

Now suppose that $x$ and $M$ interfere.  As in $(I)$, Lemma
\ref{L:maxcompletecol} covers the case where $M$ is above $x$. And
if $M$ is below $x$, the transpose of the argument given in $(I)$,
shows that $(M,x)$ satisfies the transpose analogue of an
upper-perfectible pair of interfering row moves.

Suppose $M$ is empty. Then $m$ consists of a single string that is
matched above by $S$ and $(S,m)$ is non-interfering. Since $m$ is
also contained in $S\cup x$ we see that $(S\cup x,m)$ is normal.
Suppose $m$ is matched above by $S\cup x$. The case that special
interference occurs for $(S\cup x,m)$, is handled by Lemma
\ref{L:perfectionspecialinters} below; in particular $M$ and $x$
satisfy an elementary equivalence. Otherwise $(S\cup x,m)$ is
non-interfering and therefore compatible. Then $\tM$ is empty, which
leads to a trivial elementary equivalence for $M$ and $x$. Otherwise
$m$ is continued above by $S\cup x$. Then the negatively modified
row of $x$ is the positively modified row of $m$ (say the $r$-th)
and $\Delta_\rs(S)_r=1$. In this case one may deduce the
noninterference of $(S\cup x,m)$ from that of $(S,m)$. Therefore
$(S\cup x,m)$ is compatible. We have $\tM=m^+$ (where $m^+$ is
defined for the pair $(S\cup x,m)$) which is continued above by $x$
to be a column move $x\cup \tM$ from $\mu$.

Hence $M$ and $x$ satisfy an elementary equivalence.
\end{proof}

\begin{lem}\label{L:perfectionspecialinters}
Suppose $m=s$ is a column move such that $\push(S,m)=(S \setminus m,
\emptyset)$, and suppose that $x$ is a maximal completion column
move from $\mu$ such that $(S\cup x,m)$ is in the special
interference case. Then $(S\cup x,m)$ is pushout-perfectible, with
$m_\comp$ such that $\push(S \cup x,m)=((S\cup x\cup m_{\comp}
)\setminus m,m_{\comp})$. Moreover $m_{\rm comp}$ corresponds to $m$
shifted up one cell and $m_{\comp}$ extends $x$ above to a column
move from $\mu$.
\end{lem}
\begin{proof}
Let $\eta= x * \mu$. By assumption $m\subset S$, the single string
of $m$ matches $S\cup x$ above, and $\rs(\eta)/\rs(\nu)$ is not a
horizontal strip where $\nu=m*\la$.

In this case, we must have $\rs(\eta)_{R^+}=\rs(\mu)_R$, with $R$
the negatively modified row of $m$ and $R^+$ the positively modified
row of $x$. Let $b$ be the leftmost cell in $\partial \eta$ in row
$R^+$, and let $c$ be the leftmost cell in $\partial \lambda$ in row
$R$. We claim that $b$ and $c$ lie in the same column.

Since the hook-length of $c$ in  $\la$ is $k$ by definition of moves,
we have from
$\rs(\mu)_{R}=\rs(\eta)_{R^+}$ that
the hook-length of $b$ in $\eta$
is $k-1$ or $k$.  In the case that it is
equal to $k-1$ we easily see that
$b$ and $c$ lie
in the same column.  In the other case,
the claim follows from Lemma~\ref{L:everywhere}(3).

The rest of the proof is then exactly as
in the proof of
Proposition~\ref{L:maximalinterferecol}.
\end{proof}

%

\begin{proof}[Proof of Proposition \ref{L:commutingaugment}]
The existence of an equivalence $\tilde M x = \tilde x M$ is
guaranteed by Lemma \ref{L:maxcompleteequiv}. In some cases (when $M
= \emptyset$ or when $(x,M)$ interferes) there may be more than one
choice for such an equivalence.  In such cases, the proof of Lemma
\ref{L:maxcompleteequiv} provides a particular $\tM$. In the other
cases $M$ and $x$ define a unique elementary equivalence which
uniquely specifies $\tM$.

It suffices to show that $(S \cup x, m)$ is compatible and that
$\push(S \cup x, m) = (\tilde S, \tilde M)$ for some strip $\tilde
S$ and with $\tM$ specified as above. It will then follow that
$\tilde x$ augments $S'$ to give $\tilde S$.

Let $\kappa$ and $\eta$ be defined by $S \cup x = \kappa/\lambda$
and $\eta = \tM * \kappa$. We will use the criteria of Propositions
\ref{P:pushcrit} and Proposition \ref{P:pushcritcol}. It is clear
that (when it is nonempty) $\tM$ has the same diagram as $M$ which
has the same diagram as $m$.  It is also clear from the
commutativity of the top face that $\nu \subset \eta$.  It remains
to check Condition (1) of Proposition \ref{P:pushcrit} (or
Proposition \ref{P:pushcritcol}).

{\it I) $m$ and $x$ are row moves.} The proof of Lemma
\ref{L:maxcompleteequiv} deals with the cases where $(x,M)$ is
interfering. If $(S\cup x,m)$ interferes while $(S,m)$ does not,
then $x$ and $M$ interfere and this case has already been covered.
Suppose that $(S,m)$ interferes while $(S \cup x,m)$ does not. In
this case the negatively modified columns of $m$ are immediately to
the left of the negatively modified columns of $x$, and we have that
$x$ and $M$ are matched above with $x$ continuing below $M$.
 So we get
\begin{align*}\ecs(S \cup x,m) &= \cs(\la) + \change_\cs(S\cup x) +
\changecol(m')\\
&= \cs(\mu) + \changecol(x) + \changecol(m') \\
&= \cs(\mu) + \changecol(\tilde x) + \changecol(M).
\end{align*}
If some (but not all) positively modified columns of $m$ are
negatively modified columns of $x$, then $M$ and $x$ interfere.
Hence this case has already been covered. If the positively modified
columns of $m$ are all negatively modified columns of $x$, then $M =
\emptyset$, which was covered in the proof of Lemma
\ref{L:maxcompleteequiv}.

Finally, in all the other cases $(x,M)$ does not interfere and $M$
is nonempty, so that there is a unique choice for the equivalence
$\tx M \equiv \tM x$. Moreover, the positively modified columns of
$m$ are not negatively modified columns of $x$, and if there is
interference in $(S,m)$ and $(S \cup x,m)$ then it will occur in the
same positions and require the same perfection (that is, the
interference has nothing to do with $x$). Let $m'$ be defined as
usual when calculating the pushout $(S \cup x, m)$.  Let $m_1$ be
the strings of $m$ matched below by $S$.  Let $m_2$ be the strings
of $m$ matched below by $x$. Then $\changecol(m') = \changecol(m)-
\changecol(m_1 \cup m_2)$ and we calculate:
\begin{align*}
\ecs(S \cup x,m) &= \per_{m}(\cs(\la) +\changecol(S \cup x) +
\changecol(m'))
\\
& = \per_{m}\bigl(\cs(\la) +\change_\cs(S) + \changecol(x) + \changecol(m')\bigr)
\\
& = \per_{m}\bigl(\cs(\mu) + \changecol(x) + \changecol(m) - \changecol(m_1 \cup m_2)\bigr)  \\
& = \per_{m}\bigl(\cs(\mu) + \changecol(m) - \changecol(m_1)+  \changecol(x)-\changecol(m_2)\bigr)  \\
&= \per_{m}\bigl(\cs(\mu) + \changecol(m) - \changecol(m_1)\bigr) + \changecol(\tilde x)
\\
& = \cs(\mu) + \changecol(M) + \changecol(\tilde x).
\end{align*}

{\it II) $m$ is a row move and $x$ is a column move.}  Notice that
$\changecol(S)=\changecol(S\cup x)$ and the strips $S$ and $S\cup x$
modify the same columns.  Therefore $\ecs(S\cup x,m)=
\ecs(S,m)=\cs(\eta)$ and the result follows immediately.

{\it III) $m$ is a column move and $x$ is a row move.} This is
similar to case \textit{II)}.

{\it IV) $m$ and $x$ are column moves.}  This case is basically the
same as case {\it I)}, except that special care needs to be taken
when there is special interference.  Suppose there is special
interference in $(S\cup x, m)$ but none in $(S,m)$.  In this case,
if $m \subset S$,  then we are done by
Lemma~\ref{L:perfectionspecialinters}.
If $m \cap S=\emptyset$, then $m\subseteq x$
and $(S,m)$ were interfering (not special interference). Therefore
$\push(S,m)=(S\cup m_{\rm  comp},m \cup m_{\rm comp})$ giving
$\push(S \cup x, m) = ((S \cup x \cup m_{\rm
  comp})
\setminus m, m_{\comp})$ and the equivalence $m_{\rm comp} x = (x\setminus m)
(m \cup m_{\rm comp})$ completing the cube.

Suppose there is special interference in $(S,m)$ but none in $(S\cup x,m)$.
In this case the negatively modified row
of $m$ is immediately below the negatively modified
row of $x$, and we have that $x$ and $M$ are matched below.
 So we have as in case {\it I)}
\begin{align*}\ers(S \cup x,m) &= \rs(\la) + \change_\rs(S\cup x) +
\changerow(m')\\
&= \rs(\mu) + \changerow(x) + \changerow(m') \\
&= \rs(\mu) + \changerow(\tilde x) + \changerow(M).
\end{align*}
All the other cases are as in case {\it I)}.
\end{proof}

\section{Pushouts of equivalent paths are equivalent}
\label{sec:commutingcube}

The goal of this section is to prove Proposition
\ref{P:pushseqequiv}. By Proposition \ref{P:commuaugmen} it suffices
to show that \textit{there exist} pushout sequences starting from
$(S,\bp)$ and $(S,\bp')$ respectively, such that the resulting final
pairs $(\tS,\bq)$ and $(\tS',\bq')$, satisfy $\tS=\tS'$ and
$\bq\equiv \bq'$. By Proposition \ref{P:canonicalpushseq} we may
assume that both pushout sequences start by maximizing the strip
$S$. We may therefore assume that $S$ is already maximal. Since
equivalences in the $k$-shape poset are generated by elementary
equivalences, we may assume by induction on the length of the paths,
that $\bp=\tn m \equiv \tm n = \bp'$ is an elementary equivalence
starting at $\la$.

To summarize, it suffices to show that given the elementary
equivalence $\tn m \equiv \tm n$ starting at $\la$ and a
$\la$-addable maximal strip $S$, there exist pushout sequences from
$(S,\tn m)$ and $(S,\tm n)$, producing $(\tS,\tN M)$ and $(\tS',\tM
N)$ respectively, such that $\tS=\tS'$ and $\tN M \equiv \tM N$.

Since $S$ is maximal, $(S,m)$ and $(S,n)$ are compatible. Let
$\push(S,m) = (S_m,M)$ and $\push(S,n) = (S_n,N)$. This furnishes
the three faces touching the vertex $\la$ in the cube pictured in
\eqref{E:cube}, and all vertices except $\omega$.
\begin{equation} \label{E:cube}
\begin{diagram}
\node[2]{\mu}  \arrow{sw,t}{M} \arrow[2]{e,t}{N}
\node[2]{\rho} \arrow{sw,t,3,..}{\tilde M} \\
\node{\eta} \arrow[2]{e,t,3,..}{\tilde N}
\node[2]{\omega}\\
\node[2]{\lambda} \arrow[2]{n,l,1}{S} \arrow[2]{e,b,1}{n}
\arrow{sw,b}{m} \node[2]{\kappa} \arrow{sw,b}{\tilde m}
\arrow[2]{n,r}{S_n} \\
\node{\nu} \arrow[2]{e,b}{\tilde n} \arrow[2]{n,l}{S_m}
\node[2]{\theta} \arrow[2]{n,r,3,..}{\tilde S}
\end{diagram}
\end{equation}
It suffices to prove the following.
\begin{enumerate}
\item If $\tn\ne\emptyset$ then $(S_m,\tn)$ is compatible. Let
$\push(S_m,\tn)=(\tS,\tN)$ with final shape $\omega$.
\item If $\tm\ne\emptyset$ then $(S_n,\tm)$ is compatible. Let
$\push(S_n,\tm)=(\tS',\tM)$ with final shape $\omega'$.
\item We may assume not both $\tm$ and $\tn$ are empty.
\begin{enumerate}
\item If $\tn\ne\emptyset$ and $\tm\ne\emptyset$ then
$\omega=\omega'$ and $\tN M \equiv \tM N$ is an elementary
equivalence.
\item If $\tn\ne\emptyset$ and $\tm=\emptyset$ then (with $\omega$
defined by (1)) $\omega/\rho$ is a move and the right face of
\eqref{E:cube} defines an augmentation move.
\item If $\tm\ne\emptyset$ and $\tn=\emptyset$ then (with $\omega=\omega'$
defined by (2)) $\omega/\eta$ is a move and the front face of
\eqref{E:cube} defines an augmentation move.
\end{enumerate}
\end{enumerate}

\subsection{Pushout of equivalences}

\begin{lem}\label{L:inter}
Suppose $(m,n)$ is an interfering pair of row (resp. column) moves
from $\la$ which is (either lower- or upper-) perfectible by adding
the set of cells $mn_\per$. Suppose $S = \mu/\lambda$ is a maximal
strip. Then for each string $s \in mn_\per$, we have $s \cap S =
\emptyset$ or $s \subset S$.
\end{lem}
\begin{proof} Let $s$ be a string of $mn_\per$ and let
$x,y\in s$ be contiguous cells with $x \in S$ while $y \not \in S$.
Suppose $x$ is above $y$.  We may assume that $y$ is $\mu$-addable,
for otherwise we may shift left or down to another string of
$mn_\per$. Then the column of $x$ is a modified column of $S$ and so
$y$ is a lower augmentable corner of $S$, a contradiction.  Suppose
$x$ is below $y$. Then the row of $x$ is a modified row of $S$ and
so $y$ is an upper augmentable corner of $S$, again a contradiction.
\end{proof}

\begin{lem}\label{L:MNinterfere}
Suppose $M$ and $N$ interfere.  Then so do $m$ and $n$.
\end{lem}
\begin{proof}
We first suppose that $M$ and $N$ are row moves and
we assume without loss of generality that $M$ is above $N$.  Let $c_N$
be the rightmost negatively modified column of $N$, so that $c_N^+$ is
the leftmost positively modified column of $M$.  Since $M$ and $N$ interfere
we have $\cs(\mu)_{c_N}= \cs(\mu)_{c_N^+}+1$.  Now let
$c_n$ be the rightmost negatively modified column
of $n$, and let $c_m$ be the leftmost  positively modified column of $M$.
Since strings of $M$ and $N$ are translates of those of $m$ and $n$ we have
$\cs(\mu)_{c_N}=\cs(\la)_{c_n}$ and $\cs(\mu)_{c_N^+}=\cs(\la)_{c_m}$.
But then  $\cs(\la)_{c_n}=\cs(\la)_{c_m}+1$ so $m$ and $n$ interfere.

When $M$ and $N$ are column moves, the proof is similar.
\end{proof}

\begin{lem}\label{L:topequiv}
Let $(m,n)$ be a pair of moves from $\la$ that define an elementary
equivalence and let $S = \mu/\lambda$ be a maximal strip. Write
$\push(S,m) = (S_m,M)$ and $\push(S,n) = (S_n,N)$. Then the pair
$(M, N)$ defines an elementary equivalence.
\end{lem}
\begin{proof}
{\it I) $m$ and $n$ are row moves.} We may assume that $M$ and $N$
are nonempty. By Lemma \ref{L:initialstrings} the final string of
$m$ and of $n$ is not matched below by $S$.

Suppose $M$ and $N$ do not intersect. Since both are row moves they
are not contiguous. If $(M,n)$ is non-interfering then they satisfy
an elementary equivalence. So we may assume that $(M,N)$ is
interfering. By Lemma \ref{L:MNinterfere}, $(m,n)$ is interfering.
Without loss of generality let $m$ be above $n$. First suppose
$(m,n)$ is lower-perfectible by adding the set of cells $mn_\per$.
If $(S,m)$ is non-interfering then by Lemma \ref{L:inter}, $(M,N)$
is lower-perfectible where $MN_\per$ lies in a subset of the columns
of $mn_\per$ (even if $(S,n)$ interfered). If $(S,m)$ is interfering
then $(M,N)$ is lower-perfectible; the required additional strings
for $MN_\per$ (which lie on the same set of rows; see
Proposition~\ref{L:maximalinterfere}) can be constructed using the
technique of Lemma \ref{L:extendcompletion}. If $(m,n)$ is
upper-perfectible, one may similarly show that $(M,N)$ is
upper-perfectible.

Otherwise we may assume that $M$ and $N$ intersect. By the
definition of row equivalence we may assume that there exist strings
$s$ and $t$ of $M$ and $N$ respectively such that $s$ continues
above $t$ while $t$ continues below $s$.

Suppose the string $s$ (resp. $t$) belongs to $m_{\comp}$ (resp.
$n_{\comp}$) and the string $t$ (resp. $s$) comes from a string of
$n$ (resp. of $m$) that was pushed above. Then we obtain the
contradiction that $S_m$ (resp. $S_n$) is not a horizontal strip.

Suppose the string $s$ (resp. $t$) belongs to $m_{\comp}$ (resp.
$n_{\comp}$) and the string $t$ (resp. $s$) is a string of $n$
(resp. $m$). By Proposition \ref{L:maximalinterfere} $m_{\comp}$
(resp. $n_{\comp}$) lies on the rows of the last string of $m$
(resp. $n$), yielding the contradiction that $m$ and $n$ already
intersected and did not satisfy an elementary equivalence.

Suppose the string $s$ (resp. $t$) belongs to $m$ (resp. $n$) and
the string $t$ (resp. $s$) was pushed above. This is a contradiction
since $m$, $n$, and $S$ are all $\la$-addable.

In all other cases, one deduces the contradiction that $m$ and $n$
meet but do not satisfy an elementary equivalence.

{\it II) $m$ is a row move and $n$ is a column move.} It suffices to
check that $M$ and $N$ are reasonable and non-contiguous.

Suppose $M$ and $N$ are not reasonable.  Let $s$ and $t$ be strings of $M$
and $N$ respectively that are not reasonable.

Suppose the strings $s$ and $t$ belong to $m_{\comp}$ and
$n_{\comp}$ respectively. Let $a\in s\cap t$. By Proposition
\ref{L:maximalinterferecol}, $a\in s=n_\comp$ lies atop a cell $b$
of $n$. In particular $b\not\in \la$. Since $a\in M$ and $M$ is a
$\la\cup S$-addable horizontal strip, $b\in \la\cup S$, that is,
$b\in S$. Then we have the contradiction that either $S_m$ is not a
horizontal strip or $m$ and $n$ already intersected and did not
satisfy an elementary equivalence.

Suppose the strings $s$ and $t$ come from strings of $m$ and $n$
that have been pushed up and right respectively. Then we have the
contradiction that either $S_n$ is not a horizontal strip or $m$ and
$n$ already intersected and did not satisfy an elementary
equivalence.

Suppose the string $s$ belongs to $m_{\comp}$ and the string $t$ is
a string of $n$. By Proposition \ref{L:maximalinterfere} $m_{\comp}$
lies on the rows of the last string of $m$. This leads to the
contradiction that $m$ and $n$ already intersected and did not
satisfy an elementary equivalence.

All the other cases can easily be ruled out.

Now suppose $M$ and $N$ are contiguous. Suppose $M$ is above $N$.
Let $x$ and $y$ be cells of $M$ and $N$ respectively that are
contiguous. Suppose $y\in n_{\comp}$. By Proposition
\ref{L:maximalinterferecol} it follows that the cell $y^-$ just
below $y$ is in $n$ and $\row(y^-)$ is a modified row of $n$. This
implies that the cell below $x$ does not belong to $\lambda$ and
thus needs to belong to $S$. Therefore $x$ cannot be part of
$m_{\comp}$ since otherwise $S_m$ would not be a horizontal strip by
Lemma~\ref{L:mcompcol}. So the string that contains $x$ was pushed
up during the pushout.  But then we have the contradiction that $m$
and $n$ are contiguous.

Suppose that $y$ belongs to a string of $n$ that was pushed right
during the pushout. In this case the row of $y$ is not a modified
row of $S$ and we have that the cell $x^-$ immediately to the left
of $x$ is also in $S$, but $x^-\notin m$ for otherwise $m$ and $n$
would already be contiguous. By Proposition \ref{L:maximalinterfere}
it follows that $x\not\in m_\comp$. Since $x^-\in S$ it follows that
$x$ was pushed up during $\push(S,m)$. $c=\col(x)$ is not a modified
column of $S$ and we get that the cell below $y$ is also in $S$,
which yields the contradiction $\cs(\la)_{c^-} = \cs(\la)_c$.

Finally, suppose that $y$ belongs to $n$.  Since $m$ and $n$ are not
contiguous $x$ does not belong to $m$.  If $x \in m_{\comp}$ then by
Proposition \ref{L:maximalinterfere} there is a cell $z$ in its row
that belongs to $m$. But then a hook-length analysis shows that the
column of $z$ cannot be a modified column of $m$, a contradiction.
So $x$ belongs to a string of $M$ that was pushed up during the
pushout. Hence the column of $x$ is not a modified column of $S$ and
there is a cell of $S\cap n$ below $y$ that is contiguous with the
cell below $x$, contradicting the assumption that $m$ and $n$ are
not contiguous.

The case in which $M$ is below $N$, is similar.

{\it III) $m$ and $n$ are column moves.}  The proof is similar to
that of {\it I)} (using the fact that the perfection of a column
move lies on the same columns as its final string).
\end{proof}

\subsection{Commuting cube (non-degenerate case)}
Suppose that $m,n,\tilde m,\tilde n, M$ and $N$ are non-empty. Then
the following cube commutes
\begin{equation} \label{DiagCube}
\begin{diagram}
\node[2]{\mu}  \arrow{sw,t}{M} \arrow[2]{e,t}{N}
\node[2]{\rho} \arrow{sw,t,3}{\tilde M} \\
\node{\eta} \arrow[2]{e,t,3}{\tilde N}
\node[2]{\omega}\\
\node[2]{\lambda} \arrow[2]{n,l,1}{S} \arrow[2]{e,b,1}{n}
\arrow{sw,b}{m} \node[2]{\kappa} \arrow{sw,b}{\tilde m}
\arrow[2]{n,r}{S_n} \\
\node{\nu} \arrow[2]{e,b}{\tilde n} \arrow[2]{n,l}{S_m}
\node[2]{\theta} \arrow[2]{n,r,3}{\tilde S}
\end{diagram}
\end{equation}
so that the two horizontal faces are elementary equivalences and the
four vertical faces are pushouts.  The three faces touching
$\lambda$ are assumed to be given.

By Lemma \ref{L:topequiv}, the top face defines an elementary
equivalence. Since $M$ and $N$ are non-empty $\omega$ is determined
uniquely.

We will use Proposition~\ref{P:pushcrit} (or
Proposition~\ref{P:pushcritcol}) to show that there exists a $\tilde
S$ such that $\push(S_m,\tilde n)=(\tilde S,\tilde N)$ and
$\push(S_n,\tilde m)=(\tilde S,\tilde M)$.

\medskip

{\bf Preliminary claim:}  {\it Conditions (2) and (3) of
Proposition~\ref{P:pushcrit} (or Proposition~\ref{P:pushcritcol}) hold.}

\medskip

It is obvious that
conditions (2) holds since by definition of pushouts and
equivalences, $M$ and $N$ are moves whose strings have
the same diagrams respectively as $m$ and $n$
and thus  $\tilde M$ and $\tilde N$ are moves whose strings have
 the same diagram respectively as $\tilde m$ and $\tilde n$ ($M$ and $N$
 interfere in this case if and only if $m$ and $n$ interfere).
We will now see that condition (3) also holds, that is that $\theta
\subset \omega$.  Suppose $m$ and $n$ are row moves. Strings of
$\tilde m$ that are strings of $m$ are obviously contained in
$\omega$. Suppose $s^+$ is a string of $\tilde m$ that corresponds
to a string $s$ of $m$ that has been pushed up (let's say it
intersected with string $t$ of $n$) .  Then $s \subset t \in n$ and
we either have $t \subset S$ or $t \cap S= \emptyset$.  In the
former case, $s^+ \in M$ and thus $s^+ \subset \omega$.  In the
latter case, $s \in M \cap N$ with $t \in N$ and thus  $s^+ \in
\tilde M$, which gives $s^+ \subset \omega$.  Finally, suppose that
$s$ is a string in the perfection $mn_{\per}$ of $m$ and $n$.  By
Lemma~\ref{L:inter}, we either have $s \subset S$ or  $s \cap S=
\emptyset$.  In the former case, obviously $s \subset \omega$.  In
the latter case, since $M$ and $N$ are not empty by hypothesis, we
have that 
$(M,N)$ interferes. In every case $s$ will belong to $MN_\per$ and
will thus belong to $\omega$. If $m$ and $n$ are column moves, or if
$m$ is a row move and $n$ is a column move, then condition (3) is
shown in a similar way.

To check that the vertical faces are pushouts, it remains to verify
condition (1) of Proposition~\ref{P:pushcrit} (or
Proposition~\ref{P:pushcritcol}).

We will use the fact (see the proof of Lemma \ref{L:topequiv}) that
if $(m,n)$ is interfering and lower (resp. upper) perfectible and
$(M,N)$ is interfering, then $(M,N)$ is lower (resp. upper)
perfectible.

{\bf  I) $m$ and $n$ are row moves.} Let $m_\comp$ and $n_\comp$ be
the sets of cells added in the pushout perfections of $(S,m)$ and
$(S,n)$ respectively; they are empty in the non-interfering case.
Also let $mn_\per$ denote the set of cells defining the lower or
upper perfection of $(m,n)$, if it exists. We will repeatedly use
the fact (Proposition~\ref{L:maximalinterfere}) that $m_\comp$
(resp. $n_\comp$) all lie on the same row as the last string of $m$
(resp. $n$).

{\bf Main claim:}
\begin{equation} \label{E:csvec}
\ecs(S_m,\tilde n) = \cs(\omega) = \ecs(S_n,\tilde m)
\end{equation}
To prove \eqref{E:csvec}, it suffices to make a calculation with
modified columns.

We shall be dividing our study into
four cases according to the type of row equivalence: $m$ and $n$
do not interact, $m$ and $n$ are matched below, $m$ and $n$
are matched above, and $m$ and $n$ interfere.

\subsubsection{$m$ and $n$ do not interact} \label{subsubmn}
By Lemma
\ref{L:MNinterfere},  $M$ and $N$ do not interfere.
Furthermore, since
$m_\comp$ (resp. $n_\comp$) all lie on the same row as the last
string of $m$ (resp. $n$), we have
that $m_\comp$ does not interact with $n$,
$n_\comp$ does not interact with $m$, and $m_\comp$ does not interact
with $n_\comp$.
 In particular, $M$ and $N$ do
not interact.  It is thus not difficult to see that we
obtain
\begin{align*}
\ecs(S_m,\tilde n) = \cs(\omega)
= \ecs(S_n,\tilde m)=  \cs(\lambda) & +\Delta_{\cs}(S)+
\Delta_{\cs}(m')+\Delta_{\cs}(n')  \\
 & +
\Delta_{\cs}(m_{\comp})+\Delta_{\cs}(n_{\comp})
\end{align*}
where $m'$ and $n'$ originate respectively from
$(S,m)$ and $(S,n)$.

\subsubsection{$m$ and $n$ are matched below, with $m$ continuing above $n$}
By Lemma \ref{L:introw}, $n$ has rank greater than $m$.  By
maximality of $S$, we see that $m_\comp$ and $n_\comp$ cannot
intersect.  Thus, by Lemma \ref{L:introw} applied to $M$ and $N$,
positively modified columns of $m_\comp$ are positively modified
columns of $n$.
We conclude that there are three interesting types
of modified columns of $S$ (the other types interact in a manner
that was covered in
\ref{subsubmn}): (a) those which are positively modified
columns of both $m$ and $n$, (b) those which are immediately to the
right of negatively modified columns of $m$ and cause interference,
and (c) those which are immediately to the right of negatively
modified columns of $n$ and cause interference.  Each such column
will affect the vectors in \eqref{E:csvec} at three different
indices: ($+$) a positively modified column of $m \cup m_\comp$ or
$n \cup n_\comp$, ($-m$) a negatively modified column of $m \cup
m_\comp$, or ($-n$) a negatively modified column of $n \cup
n_\comp$.   For each of the cases (a), (b) and (c),  we
draw a cube whose edges give
the entries $(-m), (-n),(+)$ associated to the corresponding move
or strip in the cube \eqref{DiagCube}.  The commutation of the three cubes implies
that \eqref{E:csvec} is satisfied.  We will write $\bar 1$ to denote a
negatively
modified column.

$$
{\rm case~(a)} \qquad \qquad \qquad
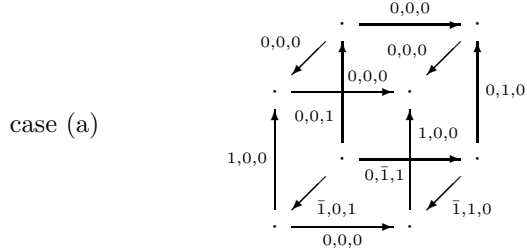
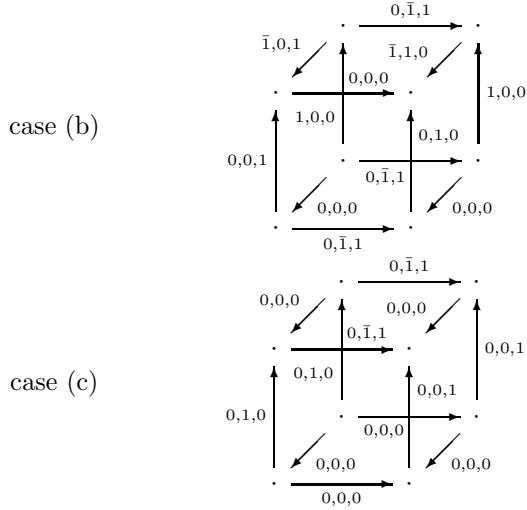
\begin{diagram}
\node[2]{\cdot}  \arrow{sw,t}{\mbox{\tiny{0,0,0}}}
\arrow[2]{e,t}{\mbox{\tiny{0,0,0}}}
\node[2]{\cdot} \arrow{sw,t,3}{\mbox{\tiny{0,0,0}}} \\
\node{\cdot} \arrow[2]{e,t,3}{\mbox{\tiny{0,0,0}}}
\node[2]{\cdot}\\
\node[2]{\cdot} \arrow[2]{n,l,1}{\mbox{\tiny{0,0,1}}}
\arrow[2]{e,b,1}{\mbox{\tiny{0,$\bar 1$,1}}}
\arrow{sw,b}{\mbox{\tiny{$\bar 1$,0,1}}}
\node[2]{\cdot} \arrow{sw,b}{\mbox{\tiny{$\bar 1$,1,0}}}
\arrow[2]{n,r}{\mbox{\tiny{0,1,0}}} \\
\node{\cdot} \arrow[2]{e,b}{\mbox{\tiny{0,0,0}}}
\arrow[2]{n,l}{\mbox{\tiny{1,0,0}}}
\node[2]{\cdot} \arrow[2]{n,r,3}{\mbox{\tiny{1,0,0}}}
\end{diagram}
$$

$$
{\rm case~(b)} \qquad \qquad \qquad
\begin{diagram}
\node[2]{\cdot}  \arrow{sw,t}{\mbox{\tiny{$\bar 1$,0,1}}}
\arrow[2]{e,t}{\mbox{\tiny{0,$\bar 1$,1}}}
\node[2]{\cdot} \arrow{sw,t,3}{\mbox{\tiny{$\bar 1$,1,0}}} \\
\node{\cdot} \arrow[2]{e,t,3}{\mbox{\tiny{0,0,0}}}
\node[2]{\cdot}\\
\node[2]{\cdot} \arrow[2]{n,l,1}{\mbox{\tiny{1,0,0}}}
\arrow[2]{e,b,1}{\mbox{\tiny{0,$\bar 1$,1}}}
\arrow{sw,b}{\mbox{\tiny{0,0,0}}}
\node[2]{\cdot} \arrow{sw,b}{\mbox{\tiny{0,0,0}}}
\arrow[2]{n,r}{\mbox{\tiny{1,0,0}}} \\
\node{\cdot} \arrow[2]{e,b}{\mbox{\tiny{0,$\bar 1$,1}}}
\arrow[2]{n,l}{\mbox{\tiny{0,0,1}}}
\node[2]{\cdot} \arrow[2]{n,r,3}{\mbox{\tiny{0,1,0}}}
\end{diagram}
$$

$$
{\rm case~(c)} \qquad \qquad \qquad
\begin{diagram}
\node[2]{\cdot}  \arrow{sw,t}{\mbox{\tiny{0,0,0}}}
\arrow[2]{e,t}{\mbox{\tiny{0,$\bar 1$,1}}}
\node[2]{\cdot} \arrow{sw,t,3}{\mbox{\tiny{0,0,0}}} \\
\node{\cdot} \arrow[2]{e,t,3}{\mbox{\tiny{0,$\bar 1$,1}}}
\node[2]{\cdot}\\
\node[2]{\cdot} \arrow[2]{n,l,1}{\mbox{\tiny{0,1,0}}}
\arrow[2]{e,b,1}{\mbox{\tiny{0,0,0}}}
\arrow{sw,b}{\mbox{\tiny{0,0,0}}}
\node[2]{\cdot} \arrow{sw,b}{\mbox{\tiny{0,0,0}}}
\arrow[2]{n,r}{\mbox{\tiny{0,0,1}}} \\
\node{\cdot} \arrow[2]{e,b}{\mbox{\tiny{0,0,0}}}
\arrow[2]{n,l}{\mbox{\tiny{0,1,0}}}
\node[2]{\cdot} \arrow[2]{n,r,3}{\mbox{\tiny{0,0,1}}}
\end{diagram}
$$

As an example of how to read these tables: if we look at a modified
column $c$ of $S$ of type (a), there is a string $s \in m$ and a
string $t \in n$ which end on the same column.  Suppose the
negatively modified column of $s$ is $c'$ and that of $t$ is $c''$.
Then reading the edge corresponding to $\tilde m$ in the cube in
case (a), we get that the changes in $\cs$ across $\tm$ in columns
$c'$, $c''$, and $c$, are $-1$, $1$, and $0$ respectively.

\subsubsection{$m$ and $n$ are matched above, with $m$ continuing below $n$}
By Lemma \ref{L:introw}, $n$ has greater rank than $m$.  The
rightmost positively modified column of $n$ agrees with that of the
rightmost negatively modified column of $\tilde m$. Therefore, if
$\push(S_n, \tilde m)$ involves interference due to cells of $S\cap
S_n$, then one can deduce that all positively modified columns of
$n$ are positively modified columns of $S$.  But this implies that
$n'= \emptyset$ in the pushout of $(S,n)$, which (without the
existence of a modified column of $S$ which causes interference with
$m$ and $n$) leads to an empty move $N$ which we assume is not the
case. We conclude that there are three interesting types of modified
columns of $S$: (a) those which are positively modified columns of
$n$ but not $m$, (b) those which are positively modified columns of
$m$ but not $n$, and (c) those which are immediately to the right of
negatively modified columns of both $m$ and $n$ and cause
interference (with both $m$ and $n$). Note that case (b) is
especially interesting: interference always occurs when calculating
$\push(S_m, \tilde n)$. We list vectors in the indices: ($-$)
negatively modified columns of $n \cup n_\comp$, ($+n$) positively
modified columns of $n \cup n_\comp$, and ($+m$) positively modified
columns of $m \cup m_\comp$.  For each of the cases (a), (b) and
(c),  we draw a cube whose edges give the entries $(-), (+n),(+m)$
associated to the corresponding move or strip in the cube
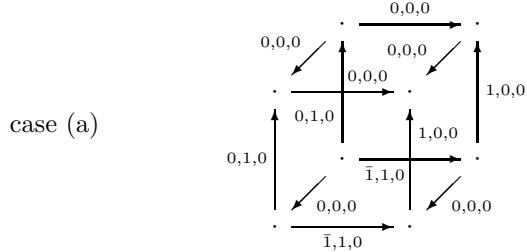
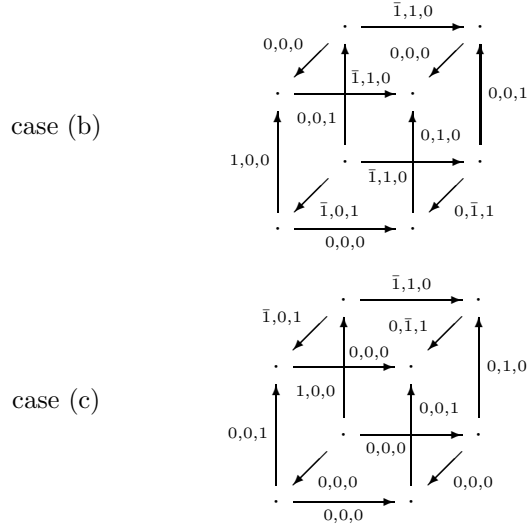
\eqref{DiagCube}.

$$
{\rm case~(a)} \qquad \qquad \qquad
\begin{diagram}
\node[2]{\cdot}  \arrow{sw,t}{\mbox{\tiny{0,0,0}}}
\arrow[2]{e,t}{\mbox{\tiny{0,0,0}}}
\node[2]{\cdot} \arrow{sw,t,3}{\mbox{\tiny{0,0,0}}} \\
\node{\cdot} \arrow[2]{e,t,3}{\mbox{\tiny{0,0,0}}}
\node[2]{\cdot}\\
\node[2]{\cdot} \arrow[2]{n,l,1}{\mbox{\tiny{0,1,0}}}
\arrow[2]{e,b,1}{\mbox{\tiny{$\bar 1$,1,0}}}
\arrow{sw,b}{\mbox{\tiny{0,0,0}}}
\node[2]{\cdot} \arrow{sw,b}{\mbox{\tiny{0,0,0}}}
\arrow[2]{n,r}{\mbox{\tiny{1,0,0}}} \\
\node{\cdot} \arrow[2]{e,b}{\mbox{\tiny{$\bar 1$,1,0}}}
\arrow[2]{n,l}{\mbox{\tiny{0,1,0}}}
\node[2]{\cdot} \arrow[2]{n,r,3}{\mbox{\tiny{1,0,0}}}
\end{diagram}
$$

$$
{\rm case~(b)} \qquad \qquad \qquad
\begin{diagram}
\node[2]{\cdot}  \arrow{sw,t}{\mbox{\tiny{0,0,0}}}
\arrow[2]{e,t}{\mbox{\tiny{$\bar 1$,1,0}}}
\node[2]{\cdot} \arrow{sw,t,3}{\mbox{\tiny{0,0,0}}} \\
\node{\cdot} \arrow[2]{e,t,3}{\mbox{\tiny{$\bar 1$,1,0}}}
\node[2]{\cdot}\\
\node[2]{\cdot} \arrow[2]{n,l,1}{\mbox{\tiny{0,0,1}}}
\arrow[2]{e,b,1}{\mbox{\tiny{$\bar 1$,1,0}}}
\arrow{sw,b}{\mbox{\tiny{$\bar 1$,0,1}}}
\node[2]{\cdot} \arrow{sw,b}{\mbox{\tiny{0,$\bar 1$,1}}}
\arrow[2]{n,r}{\mbox{\tiny{0,0,1}}} \\
\node{\cdot} \arrow[2]{e,b}{\mbox{\tiny{0,0,0}}}
\arrow[2]{n,l}{\mbox{\tiny{1,0,0}}}
\node[2]{\cdot} \arrow[2]{n,r,3}{\mbox{\tiny{0,1,0}}}
\end{diagram}
$$

$$
{\rm case~(c)} \qquad \qquad \qquad
\begin{diagram}
\node[2]{\cdot}  \arrow{sw,t}{\mbox{\tiny{$\bar 1$,0,1}}}
\arrow[2]{e,t}{\mbox{\tiny{$\bar 1$,1,0}}}
\node[2]{\cdot} \arrow{sw,t,3}{\mbox{\tiny{0,$\bar 1$,1}}} \\
\node{\cdot} \arrow[2]{e,t,3}{\mbox{\tiny{0,0,0}}}
\node[2]{\cdot}\\
\node[2]{\cdot} \arrow[2]{n,l,1}{\mbox{\tiny{1,0,0}}}
\arrow[2]{e,b,1}{\mbox{\tiny{0,0,0}}}
\arrow{sw,b}{\mbox{\tiny{0,0,0}}}
\node[2]{\cdot} \arrow{sw,b}{\mbox{\tiny{0,0,0}}}
\arrow[2]{n,r}{\mbox{\tiny{0,1,0}}} \\
\node{\cdot} \arrow[2]{e,b}{\mbox{\tiny{0,0,0}}}
\arrow[2]{n,l}{\mbox{\tiny{0,0,1}}}
\node[2]{\cdot} \arrow[2]{n,r,3}{\mbox{\tiny{0,0,1}}}
\end{diagram}
$$

\subsubsection{$(m,n)$ is interfering and upper-perfectible with $m$ above $n$}
Then $(M,N)$ is interfering and upper-perfectible with $M$ above
$N$. The set of cells $m \cup \tilde n = n \cup \tilde m$ is a
horizontal strip, so there are no unexpected coincidences of
modified columns. There are three interesting types of modified
columns of $S$: (a) those which are positively modified columns of
$n$, (b) those which are positively modified columns of $m$, and (c)
those which are immediately to the right of negatively modified
columns of $m$, and cause interference. Note that modified columns
of $S$ are not negatively modified columns of $mn_\per$ by
maximality of $S$. Also note that all the modified columns of $S$ in
case (b) cause interference with $n$.

The edges of the cubes give vectors in the indices: ($-m$) negatively modified columns
of $m \cup mn_\per \cup m_\comp$, ($-n+m$) negatively modified
columns of $n$ or positively modified columns of $m \cup m_\comp$,
and ($+n$) positively modified columns of $n \cup n_\comp$.

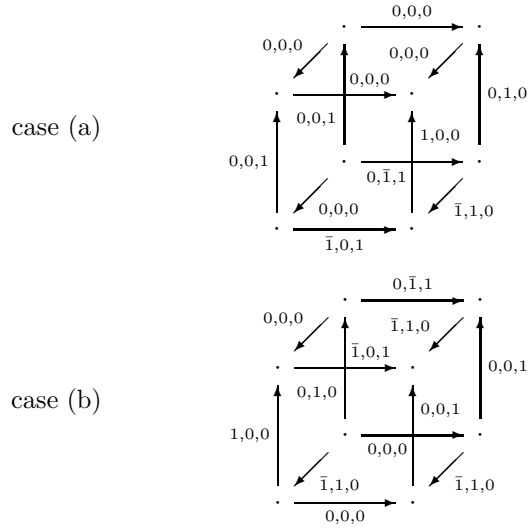
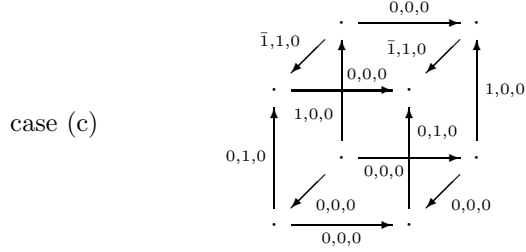
$$
{\rm case~(a)} \qquad \qquad \qquad
\begin{diagram}
\node[2]{\cdot}  \arrow{sw,t}{\mbox{\tiny{0,0,0}}}
\arrow[2]{e,t}{\mbox{\tiny{0,0,0}}}
\node[2]{\cdot} \arrow{sw,t,3}{\mbox{\tiny{0,0,0}}} \\
\node{\cdot} \arrow[2]{e,t,3}{\mbox{\tiny{0,0,0}}}
\node[2]{\cdot}\\
\node[2]{\cdot} \arrow[2]{n,l,1}{\mbox{\tiny{0,0,1}}}
\arrow[2]{e,b,1}{\mbox{\tiny{0,$\bar 1$,1}}}
\arrow{sw,b}{\mbox{\tiny{0,0,0}}}
\node[2]{\cdot} \arrow{sw,b}{\mbox{\tiny{$\bar 1$,1,0}}}
\arrow[2]{n,r}{\mbox{\tiny{0,1,0}}} \\
\node{\cdot} \arrow[2]{e,b}{\mbox{\tiny{$\bar 1$,0,1}}}
\arrow[2]{n,l}{\mbox{\tiny{0,0,1}}}
\node[2]{\cdot} \arrow[2]{n,r,3}{\mbox{\tiny{1,0,0}}}
\end{diagram}
$$

$$
{\rm case~(b)} \qquad \qquad \qquad
\begin{diagram}
\node[2]{\cdot}  \arrow{sw,t}{\mbox{\tiny{0,0,0}}}
\arrow[2]{e,t}{\mbox{\tiny{0,$\bar 1$,1}}}
\node[2]{\cdot} \arrow{sw,t,3}{\mbox{\tiny{$\bar 1$,1,0}}} \\
\node{\cdot} \arrow[2]{e,t,3}{\mbox{\tiny{$\bar 1$,0,1}}}
\node[2]{\cdot}\\
\node[2]{\cdot} \arrow[2]{n,l,1}{\mbox{\tiny{0,1,0}}}
\arrow[2]{e,b,1}{\mbox{\tiny{0,0,0}}}
\arrow{sw,b}{\mbox{\tiny{$\bar 1$,1,0}}}
\node[2]{\cdot} \arrow{sw,b}{\mbox{\tiny{$\bar 1$,1,0}}}
\arrow[2]{n,r}{\mbox{\tiny{0,0,1}}} \\
\node{\cdot} \arrow[2]{e,b}{\mbox{\tiny{0,0,0}}}
\arrow[2]{n,l}{\mbox{\tiny{1,0,0}}}
\node[2]{\cdot} \arrow[2]{n,r,3}{\mbox{\tiny{0,0,1}}}
\end{diagram}
$$

$$
{\rm case~(c)} \qquad \qquad \qquad
\begin{diagram}
\node[2]{\cdot}  \arrow{sw,t}{\mbox{\tiny{$\bar 1$,1,0}}}
\arrow[2]{e,t}{\mbox{\tiny{0,0,0}}}
\node[2]{\cdot} \arrow{sw,t,3}{\mbox{\tiny{$\bar 1$,1,0}}} \\
\node{\cdot} \arrow[2]{e,t,3}{\mbox{\tiny{0,0,0}}}
\node[2]{\cdot}\\
\node[2]{\cdot} \arrow[2]{n,l,1}{\mbox{\tiny{1,0,0}}}
\arrow[2]{e,b,1}{\mbox{\tiny{0,0,0}}}
\arrow{sw,b}{\mbox{\tiny{0,0,0}}}
\node[2]{\cdot} \arrow{sw,b}{\mbox{\tiny{0,0,0}}}
\arrow[2]{n,r}{\mbox{\tiny{1,0,0}}} \\
\node{\cdot} \arrow[2]{e,b}{\mbox{\tiny{0,0,0}}}
\arrow[2]{n,l}{\mbox{\tiny{0,1,0}}}
\node[2]{\cdot} \arrow[2]{n,r,3}{\mbox{\tiny{0,1,0}}}
\end{diagram}
$$

\subsubsection{$(m,n)$ is interfering and lower-perfectible with $m$ above $n$}
In this case $(M,N)$ is interfering and lower-perfectible with $M$
above $N$. We first observe that the positively modified columns of
$mn_\per$ are not modified columns of $S$, for otherwise all the
positively modified columns of $n$ would be modified columns of $S$
(and thus $N=\emptyset$).
 There are three interesting types of modified columns of $S$:
(a) those which are positively modified columns of $n$, (b)
those which are positively modified columns of $m$,
and (c) those which are immediately to the right of
negatively modified columns of $m$, and cause interference.  Note
that if $m_\comp$ is non-empty, the leftmost positively modified
column of it will always interfere with the rightmost negatively
modified column of $mn_\per \subset \tilde n$ in the calculation of
$\push(S_m,\tilde n)$.

The edges of the cubes give vectors in the indices: ($-m$) negatively modified columns
of $m \cup m_\comp$, ($-n+m$) negatively modified columns of $n \cup
mn_\per$ or positively modified columns of $m \cup m_\comp$, and
($+n$) positively modified columns of $n \cup mn_\per$ (and also
positively modified columns of the perfection arising from
$\push(S_m,\tilde n)$).

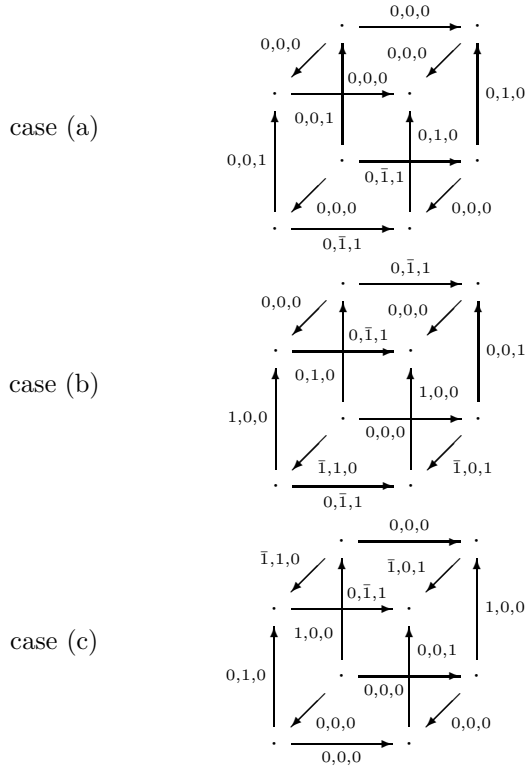
$$
\begin{diagram}
{\rm case~(a)} \qquad \qquad \qquad
\node[2]{\cdot}  \arrow{sw,t}{\mbox{\tiny{0,0,0}}}
\arrow[2]{e,t}{\mbox{\tiny{0,0,0}}}
\node[2]{\cdot} \arrow{sw,t,3}{\mbox{\tiny{0,0,0}}} \\
\node{\cdot} \arrow[2]{e,t,3}{\mbox{\tiny{0,0,0}}}
\node[2]{\cdot}\\
\node[2]{\cdot} \arrow[2]{n,l,1}{\mbox{\tiny{0,0,1}}}
\arrow[2]{e,b,1}{\mbox{\tiny{0,$\bar 1$,1}}}
\arrow{sw,b}{\mbox{\tiny{0,0,0}}}
\node[2]{\cdot} \arrow{sw,b}{\mbox{\tiny{0,0,0}}}
\arrow[2]{n,r}{\mbox{\tiny{0,1,0}}} \\
\node{\cdot} \arrow[2]{e,b}{\mbox{\tiny{0,$\bar 1$,1}}}
\arrow[2]{n,l}{\mbox{\tiny{0,0,1}}}
\node[2]{\cdot} \arrow[2]{n,r,3}{\mbox{\tiny{0,1,0}}}
\end{diagram}
$$

$$
{\rm case~(b)} \qquad \qquad \qquad
\begin{diagram}
\node[2]{\cdot}  \arrow{sw,t}{\mbox{\tiny{0,0,0}}}
\arrow[2]{e,t}{\mbox{\tiny{0,$\bar 1$,1}}}
\node[2]{\cdot} \arrow{sw,t,3}{\mbox{\tiny{0,0,0}}} \\
\node{\cdot} \arrow[2]{e,t,3}{\mbox{\tiny{0,$\bar 1$,1}}}
\node[2]{\cdot}\\
\node[2]{\cdot} \arrow[2]{n,l,1}{\mbox{\tiny{0,1,0}}}
\arrow[2]{e,b,1}{\mbox{\tiny{0,0,0}}}
\arrow{sw,b}{\mbox{\tiny{$\bar 1$,1,0}}}
\node[2]{\cdot} \arrow{sw,b}{\mbox{\tiny{$\bar 1$,0,1}}}
\arrow[2]{n,r}{\mbox{\tiny{0,0,1}}} \\
\node{\cdot} \arrow[2]{e,b}{\mbox{\tiny{0,$\bar 1$,1}}}
\arrow[2]{n,l}{\mbox{\tiny{1,0,0}}}
\node[2]{\cdot} \arrow[2]{n,r,3}{\mbox{\tiny{1,0,0}}}
\end{diagram}
$$

$$
{\rm case~(c)} \qquad \qquad \qquad
\begin{diagram}
\node[2]{\cdot}  \arrow{sw,t}{\mbox{\tiny{$\bar 1$,1,0}}}
\arrow[2]{e,t}{\mbox{\tiny{0,0,0}}}
\node[2]{\cdot} \arrow{sw,t,3}{\mbox{\tiny{$\bar 1$,0,1}}} \\
\node{\cdot} \arrow[2]{e,t,3}{\mbox{\tiny{0,$\bar 1$,1}}}
\node[2]{\cdot}\\
\node[2]{\cdot} \arrow[2]{n,l,1}{\mbox{\tiny{1,0,0}}}
\arrow[2]{e,b,1}{\mbox{\tiny{0,0,0}}}
\arrow{sw,b}{\mbox{\tiny{0,0,0}}}
\node[2]{\cdot} \arrow{sw,b}{\mbox{\tiny{0,0,0}}}
\arrow[2]{n,r}{\mbox{\tiny{1,0,0}}} \\
\node{\cdot} \arrow[2]{e,b}{\mbox{\tiny{0,0,0}}}
\arrow[2]{n,l}{\mbox{\tiny{0,1,0}}}
\node[2]{\cdot} \arrow[2]{n,r,3}{\mbox{\tiny{0,0,1}}}
\end{diagram}
$$

\bigskip

{\bf  II) $m$ is a row move and $n$ is a column move}.
We need to show that $\ecs(S,m)=\cs(\omega)=\ecs(S_n,\tilde m)$ and
$\ers(S,n)=\rs(\omega)=\ers(S_m,\tilde n)$.  We will prove the first equality,
and the second one will follow from the same principles.  We need to
show that
$$
\per_m(\cs(\lambda)+\Delta_{\cs}(S)+\Delta_{\cs}(m'))=
\per_{\tilde m}(\cs(\kappa)+\Delta_{\cs}(S_n)+\Delta_{\cs}(\tilde m')) \, .
$$
We have that $\cs(\lambda)=\cs(\kappa)$ ($n$ is a column move) and
$\Delta_{\cs}(S)=\Delta_{\cs}(S_n)$ ($N$ is a column move and
$\cs(\lambda)=\cs(\kappa)$).
Since $m$ and $\tilde m$ are row moves affecting the same columns (by
definition of mixed equivalences) and $S$ and $S_n$ are strips with the same
modified columns, we have that
$\Delta_{\cs}(m')=\Delta_{\cs}(\tilde m')$.   Therefore
$$
\cs(\lambda)+\Delta_{\cs}(S)+\Delta_{\cs}(m')=
\cs(\kappa)+\Delta_{\cs}(S_n)+\Delta_{\cs}(\tilde m')
$$
The perfections of $m$ and
$\tilde m$ are the same (given $\Delta_{\cs}(m')=\Delta_{\cs}(\tilde m')$), and
the equality follows.

\bigskip

{\bf  III) $m$ and $n$ are column moves.}  The proof is basically
the same as when $m$ and $n$ are row moves.

\newpage

\subsection{Commuting cube (degenerate case $M=\emptyset$)}
Suppose that $m,n,\tilde m$ and $\tilde n$ are non-empty and that
$M=\emptyset$. Then one can check that we have one of the following
two situations:
\begin{equation} \label{DiagCube2}
\begin{diagram}
\node[2]{\mu}  \arrow{sw,t}{\emptyset} \arrow[2]{e,t}{N}
\node[2]{\rho} \arrow{sw,t,3}{\emptyset} \\
\node{\mu} \arrow[2]{e,t,3}{N}
\node[2]{\rho}\\
\node[2]{\lambda} \arrow[2]{n,l,1}{S} \arrow[2]{e,b,1}{n}
\arrow{sw,b}{m} \node[2]{\kappa} \arrow{sw,b}{\tilde m}
\arrow[2]{n,r}{S_n} \\
\node{\nu} \arrow[2]{e,b}{\tilde n} \arrow[2]{n,l}{S\setminus m}
\node[2]{\theta} \arrow[2]{n,r,3}{\tilde S}
\end{diagram}
\end{equation}
or
\begin{equation} \label{DiagCube4}
\begin{diagram}
\node[2]{\mu}  \arrow{sw,t}{\emptyset} \arrow[4]{e,t}{N}
\node[4]{\rho} \arrow{sw,t,3}{\tilde M} \\
\node{\mu} \arrow[2]{e,t,3}{\tilde N}
\node[2]{\eta} \arrow[2]{e,t,2}{x} \node[2]{\sigma}\\
\node[2]{\lambda} \arrow[2]{n,l,1}{S} \arrow[4]{e,b,2}{n}
\arrow{sw,b}{m} \node[4]{\kappa} \arrow{sw,b}{\tilde m}
\arrow[2]{n,r}{S_n} \\
\node{\nu} \arrow[2]{n,l,2}{S \setminus m} \arrow[2]{e,b}{\tilde n}
\node[2]{\theta} \arrow[2]{n,r,3}{\hat S} \arrow[2]{e,b,2}{\emptyset}
\node[2]{\theta}
 \arrow[2]{n,r,3}{\tilde S}
\end{diagram}
\end{equation}
One obtains the commuting cube \eqref{DiagCube2} except when $(m,n)$
interferes and the perfection $mn_{\per}$ is made of strings that
are translates of those of $m$, in which case one obtains the
commuting cube \eqref{DiagCube4}. We consider the case that $m$ and
$n$ are row moves as the column move case is similar. The
exceptionsl situation can occur in two ways: either $n$ is above $m$
and the lower perfection exists, or $m$ is above $n$ and the upper
perfection exists.

Suppose first that $n$ is above $m$. Let $mn_{\per}=\bar n_I \cup
\bar n_F$ where $\bar n_I$ are the strings of $mn_{\per}$ whose
positively modified columns (resp. rows) are modified columns (resp.
rows) of $S$. Suppose $n=n_I \cup n_F$ where $n_I$ are the strings
of $n$ whose prolongation in $mn_{\per}$ is given by $\bar n_I$.
Suppose $(S,n)$ interferes with pushout perfection given by
$n_\comp$ (if there is no interference then the situation is
simpler). Then $(S_n,\tm)$ also interferes; let $\bar{n}_\comp$ be
the cells which define the pushout perfection. Finally, $(S\setminus
m,\tn)$ also interferes with pushout perfection given by
$n_\comp\cup\bar{n}_\comp$. Then $N=n_I^+ \cup n_F^+\cup n_\comp$,
$\tN=n_F^+\cup \bar{n}_F^+\cup n_\comp \cup \bar{n}_\comp$, $\tilde
M=\bar n_F^+\cup \bar{n}_\comp$, and $x=n_I^+$ are such that $\tilde
M N \equiv x \tilde N$ is an elementary equivalence. The last
vertical face is then such that $x$ is an augmentation move. Note
that $\bar n_I$ may be empty in which case $x$ is empty and we have
an ordinary cube but with $\tM\ne\emptyset$.

Suppose $m$ is above $n$. Let $n = n_I \cup n_F$ where $n_I$
consists of the strings of $n$ whose positively modified columns are
positively modified columns of $S$. Let $mn_\per = \bar{n}_I \cup
\bar{n}_F$ where $\bar{n}_I$ consists of the strings of $mn_\per$
that extend the strings of $n_I$ above. Then $N = n_F^+ \cup
n_\comp$ where $n_\comp$ is defined as in $\push(S,n)$ and $\tN =
\bar{n}_F^+ \cup n_F^+ \cup m^+ \cup n_{\comp}$ since $S \setminus
m$ and $\tilde n$ interfere and its completion is $m^+ \cup
n_{\comp}$. With $\tilde M=\bar n_F \cup m^+$ and $x=\emptyset$ we
are in the situation of the commuting cube \eqref{DiagCube4}.

\subsection{Commuting cube (degenerate case $m=\emptyset$)}
Suppose that $m=\emptyset$ and that $\tilde m n=\tilde n \emptyset$
is an elementary equivalence.  This situation can be seen as a
degenerate case of the $M=\emptyset$ case where we have the
following commuting situation.
\begin{equation} \label{DiagCube6}
\begin{diagram}
\node[2]{\mu}  \arrow{sw,t}{\emptyset} \arrow[4]{e,t}{N}
\node[4]{\rho} \arrow{sw,t,3}{\tilde M} \\
\node{\mu} \arrow[2]{e,t,3}{\tilde N}
\node[2]{\eta} \arrow[2]{e,t,2}{x} \node[2]{\sigma}\\
\node[2]{\lambda} \arrow[2]{n,l,1}{S} \arrow[4]{e,b,2}{n}
\arrow{sw,b}{\emptyset} \node[4]{\kappa} \arrow{sw,b}{\tilde m}
\arrow[2]{n,r}{S_n} \\
\node{\lambda} \arrow[2]{n,l,2}{S} \arrow[2]{e,b}{\tilde n}
\node[2]{\theta} \arrow[2]{n,r,3}{\hat S} \arrow[2]{e,b,2}{\emptyset}
\node[2]{\theta}
 \arrow[2]{n,r,3}{\tilde S}
\end{diagram}
\end{equation}
where vertical faces are either pushouts or augmentations moves, and where
$\tilde M N \equiv x \tilde N$ is
an elementary equivalence.

The case where $x$ is non-empty will occur when $n=n_I \cup n_F$ and $\tilde
m=\bar n_I \cup \bar n_F$ is such that the positively modified columns
(resp. rows) of
$\bar n_I$  are positively modified columns (resp. rows) of $S$.
Then $N=n_I^+ \cup n_F^+$, $\tilde M=\bar n_F^+$,
$\tilde N=n_F^+ \cup \bar n_F^+$ and $x=n_I^+$ are such that
$\tilde M N \equiv x \tilde N$ is
an elementary equivalence.
The last vertical face is then such that $x$ is an augmentation move.

\subsection{Commuting cube (degenerate case $\tilde m=\emptyset$)} This case is similar
to the $m=\emptyset$ case.
 Another way to see this case is to consider
that if we have $\tilde n m \equiv \emptyset n$ then we can use
$\emptyset n \equiv n \emptyset$ (which leads trivially to a commuting
cube) to fall back on the already treated $n=\emptyset$ case
(which is equal to $m=\emptyset$ by symmetry).

\section{Pullbacks}
\label{sec:pullbacks}

Given a strip $S=\mu/\lambda$ and a class of paths
$[{\mathbf {p}}]$ in the
$k$-shape poset from $\lambda$ to $\nu$,
the pushout algorithm
gives rise to
a maximal strip $\tilde S=\eta/\nu$ and a unique class of paths $[{\mathbf
  {q}}]$
in the $k$-shape poset from $\mu$ to some $\eta$:
\begin{equation}
\begin{diagram}
\node{\la} \arrow{s,l}{S}
 \arrow{e,t}{[{\mathbf {p}}]} \node{\nu} \arrow{s,r}{\tilde S}\\
\node{\mu} \arrow{e,t}{[{\mathbf {q}}]} \node{\eta}
\end{diagram}
\end{equation}
Our goal is to show that this process is invertible when the strip
$S$ is reverse maximal. That is, given the maximal strip $\tilde
S=\eta/\nu$ and the class of paths $[{\mathbf
  {q}}]$
in the $k$-shape poset from $\mu$ to $\eta$,
we will describe a {\it pullback algorithm}
that gives back the maximal strip $S=\mu/\lambda$ and the class of paths
$[{\mathbf {p}}]$ in the
$k$-shape poset from $\lambda$ to $\nu$.
To indicate that we are in the pullback situation, the direction of the
arrows will be reversed
\begin{equation}
\begin{diagram}
\node{\la}
 \node{\nu}   \arrow{w,t}{[{\mathbf {p}}]} \\
\node{\mu} \arrow{n,l}{S}  \node{\eta} \arrow{w,t}{[{\mathbf {q}}]}
\arrow{n,r}{\tilde S}
\end{diagram}
\end{equation}

The situation in the reverse case is quite similar to the situation we have
encountered so far (which we will refer to as the {\it forward} case).
We will establish a dictionary that allows to translate between the
forward and reverse cases.  Then only the main results will be stated.

\section{Equivalences in the reverse case}
If $m = \mu/\lambda$ is a move from $\lambda$ then we say that $m$
is a {\it move to}  $\mu$. We write $m\#\mu = \lambda$. We will use
the same notation for the string decomposition $m=s_1 \cup \cdots
\cup s_{\ell}$ of the forward case also in the reverse situation.
That is, string $s_1$ is the leftmost and string $s_{\ell}$ is the
rightmost. The following dictionary translates between the forward
and reverse situations:
\begin{align*}
&\la \quad &\longleftrightarrow \quad &\mu \\
&\text{move $m$ from $\la$} \quad &\longleftrightarrow \quad &\text{move $m$ to $\mu$} \\
&\mu = m * \lambda \quad &\longleftrightarrow \quad &\la = m \# \mu \\
&\text{leftmost (rightmost) string of $m$}  \quad
    &\longleftrightarrow \quad &\text{rightmost (leftmost) string of $m$} \\
&\text{continues below (resp. above)} \quad
&\longleftrightarrow \quad  &\text{continues below (resp. above)} \\
&\text{column  to the right (resp. left)} \quad
    &\longleftrightarrow \quad &\text{column  to the left (resp. right)} \\
&\text{row above (resp. below)}  \quad
    &\longleftrightarrow \quad  &\text{row below (resp. above)} \\
&\text{shifting to the right (resp. up)} \quad
    &\longleftrightarrow \quad &\text{shifting to the left (resp. down)}
\end{align*}

\begin{notation} \label{N:arrowrev}
For two sets of cells $X$ and $Y$, let $\leftarrow_X(Y)$ (resp.
$\downarrow_X(Y)$) denote the result of shifting to the left (resp.
down), each row (resp. column) of $Y$ by the number of cells of $X$
in that row (resp. column).
\end{notation}

\subsection{Reverse mixed elementary equivalence}
Let $\tilde m$ and $\tilde M$ be respectively a row move and a
column move to $\gamma$. The contiguity of two moves is defined as
in the forward case (that is, whether two disjoint strings can be
joined to form one string).

\begin{definition}
A \defit{reverse mixed elementary equivalence} is a relation of the form
\eqref{E:elemequiv} satisfying \eqref{E:commudiagram} arising from a
row move $\tilde m$ and column move $\tilde M$ to some $\gamma\in\Ksh^k$, which has
one of the following forms:
\begin{enumerate}
\item $\tilde m$ and $\tilde M$ do not intersect and
no cell of $\tilde m$ is contiguous to a cell of $\tilde M$. Then $m = \tilde m$ and
$M = \tilde M$.
\item $\tilde m$ and $\tilde M$ intersect and
\begin{itemize}
\item[(a)] $\tilde m$ continues above and below $\tilde M$. Then
\begin{equation*}
 m=\leftarrow_{\tilde M}(\tilde m) \qquad\text{and}\qquad
M=\leftarrow_{\tilde m}(\tilde M)
\end{equation*}
\item[(b)] $\tilde M$ continues above and below $\tilde m$. Then
\begin{equation*}
m= \downarrow_{\tilde M}(\tilde m)\qquad\text{and}\qquad M=
\downarrow_{\tilde m}(\tilde M).
\end{equation*}
\end{itemize}
\end{enumerate}
\end{definition}

\begin{proposition} If $(\tm, \tM)$ defines a reverse
mixed elementary equivalence,
then the prescribed sets of cells $m$ and $M$ are reverse moves such
that $m \cup \tM= M\cup\tm$ (that is, there is a shape $\lambda= \tilde M \# (m
\# \gamma)$ such
that the diagram \eqref{E:commudiagram} commutes), and
\eqref{E:charge} holds.
\end{proposition}

Mixed elementary equivalences and reverse mixed elementary equivalences are inverse operations in the following sense.

\begin{proposition} $~$
\begin{enumerate}
\item
Suppose $(m,M)$ is a (forward) mixed elementary equivalence.  Then $(\tm,\tM)$ is a reverse mixed elementary equivalence (determining $(m,M)$).
\item
Suppose $(\tm,\tM)$ is a reverse mixed elementary equivalence.  Then $(m,M)$ is a mixed elementary equivalence (determining $(\tm,\tM)$).
\end{enumerate}
Furthermore in both cases, the type -- (1), (2)(a), (2)(b) -- of the
equivalence is preserved (see Definition~\ref{D:rowcolcommute}).
\end{proposition}

\subsection{Reverse row elementary equivalence}
Let $\tilde m$ and $\tilde M$ be row moves to $\gamma$.
We say that $(\tm,\tM)$ is interfering if $\tm$ and $\tM$ do not intersect
and $\gamma \setminus(\tm\cup \tM)$ is not
a $k$-shape (or to be more precise $\cs(\gamma \setminus(\tm\cup \tM))$ is not
a partition).
Let $\tm=s_1 \cup
\cdots \cup s_r$ and $\tM=s_1' \cup \cdots
\cup s_{r'}'$ \defit{interfere}.  We immediately have
\begin{lemma} \label{L:interferereverse}

 Suppose $(\tm,\tM)$ is interfering and the top cell of $\tm$
  is above the top cell of $\tM$.  Then
\begin{enumerate}
\item  $c_{s_1',u}=c_{s_r,d}^+$. In particular, $\tm$ and $\tM$ are non-degenerate.
\item Every cell of $m$ is above every cell of $M$.
\item $\cs(\gamma)_{c_{s_r,d}}=\cs(\gamma)_{c_{s_1',u}}+1$.
\end{enumerate}
\end{lemma}
\begin{remark}  Lemma~\ref{L:interferereverse} illustrates well how the
  forward-reverse dictionary is used.
The condition for interference in the
forward case is $c_{s_1,d}^-=c_{s_{r'}',u}$
and $\cs(\lambda)_{c_{s_1,d}}=\cs(\lambda)_{c_{s_{r'}'+1}}+1$ which translates
into
$c_{s_r,d}^+=c_{s_{1}',u}$
and $\cs(\gamma)_{c_{s_r,d}}=\cs(\gamma)_{c_{s_{1}'+1}}+1$ in the reverse
case.
\end{remark}

Suppose $\tm$ is a move of rank $r$ and length $\ell$ and $\tM$ is a
move of rank $r'$ and length $\ell'$, both to $\gamma$. Suppose also
that $(\tm,\tM)$ is interfering and that the top cell of $\tm$ is
above the top cell of $\tM$. A \defit{lower perfection} (resp.
\defit{upper perfection} is a $k$-shape of the
form $\gamma \setminus (\tm \cup \tM \cup m_\per)$ (resp. $\gamma
\setminus (\tm\cup \tM\cup M_\per)$) where $m_\per$ (resp. $M_\per$)
is a $(\gamma \setminus(\tm\cup \tM))$-removable skew shape such
that $\tm\cup m_\per$ (resp. $\tM\cup M_\per$) is a row move to $\tM
\#\la$ (resp. $\tm\#\la$) of rank $r$ (resp. $r'$) and length
$\ell+\ell'$ and $\tM \cup m_\per$ (resp. $\tm\cup M_\per$) is a row
move to $\tm \#\gamma$ (resp. $\tM \#\gamma$) of rank $r+r'$ and
length $\ell'$ (resp. $\ell$). If $(\tm,\tM)$ is interfering then it
is \textit{lower} (resp. \textit{upper}) \textit{perfectible} if it
admits a lower (resp. upper) perfection.

\begin{definition}
\label{D:Rrowrowcommute} A \defit{reverse row elementary equivalence} is a
relation of the form \eqref{E:elemequiv} satisfying
\eqref{E:commudiagram} arising from two row moves $\tm$ and $\tM$ to
some $k$-shape $\gamma$, which has one of the following forms:
\begin{enumerate}
\item $\tm$ and $\tM$ do not intersect and $\tm$ and $\tM$ do not interfere.  Then $m = \tm$ and $M = \tM$.
\item $(\tm,\tM)$ is interfering (and say the top cell of $\tm$ is above the top cell of $\tM$)
and is lower (resp. upper) perfectible by adding cells $m_\per$
(resp. $M_\per$). Then $m= \tm \cup m_\per$ (resp. $m=\tm\cup
M_\per$) and $M= \tM \cup m_\per$ (resp. $M=\tM\cup M_\per$).
\item $\tm$ and $\tM$ intersect and are matched above (resp. below).
In this case $m= \tm \setminus (\tm \cap \tM)$ and $ M= \tM
\setminus (\tm \cap \tM)$.
\item $\tm$ and $\tM$ intersect and $\tm$ continues above and below $\tM$. In this case
$m= \downarrow_{\tm\cap \tM}(\tm)$ and $M= \downarrow_{\tm\cap
\tM}(\tM)$.
\item $\tM=\emptyset$ and there is a row move $m_\per$ to $\tm \# \gamma$
such that $\tm\cup m_\per$ is a row move to $\gamma$.
Then $M=m_\per$ and $m=\tm\cup m_\per$.
\end{enumerate}
In case (2),(4) and (5) the roles of $\tm$ and $\tM$ may be exchanged.
\end{definition}
\begin{proposition} \label{P:reverseforwardrowmoves}
 $~$
\begin{enumerate}
\item
Suppose $(m,M)$ defines a (forward) row elementary equivalence that
produces the pair $(\tm,\tM)$. Then $(\tm,\tM)$ defines a reverse
row elementary equivalence that produces $(m,M)$.
\item
Suppose $(\tm,\tM)$ defines a reverse row elementary equivalence
that produces the pair $(m,M)$. Then $(m,M)$ defines a row
elementary equivalence that produces $(\tm,\tM)$.
\end{enumerate}
Furthermore we have:
\begin{itemize}
\item
$(m,M)$ is in Case (1) if and only if $(\tm,\tM)$ is in Case (1)
\item
$(m,M)$ is in Case (2) or (5) if and only if $(\tm,\tM)$ is in Case (3)
\item
$(m,M)$ is in Case (3) if and only if $(\tm,\tM)$ is in Case (2) or (5)
\item
$(m,M)$ is in Case (4) if and only if $(\tm,\tM)$ is in Case (4)
\end{itemize}
\end{proposition}
\begin{remark}  According to Proposition~\ref{P:reverseforwardrowmoves},
it would seem  natural to join
Cases (2) and (5) of forward and reverse row elementary equivalences
under a single case.  Indeed, these are the only two cases that need perfections and one
can think of Case (5) as a degeneration of Case (2).   However, due to the
special nature of Case (5) (the presence of an empty move), we
prefer not to merge the two cases.
\end{remark}

\subsection{Reverse column elementary equivalence}  There is an obvious
transpose analogue of reverse row elementary equivalences which we shall
call {\it reverse column elementary equivalences}.

\subsection{Reverse diamond equivalences are generated by reverse
elementary equivalences}
A reverse diamond equivalence is just a usual diamond equivalence
$\tM m \equiv \tm M$ except that, instead of starting with $(m,M)$
and producing $(\tm,\tM)$, we start with $(\tm,\tM)$
and produce $(m,M)$. The next proposition follows immediately from
the forward situation and
Proposition~\ref{P:reverseforwardrowmoves}.

\begin{proposition} \label{P:Rdiamond} The equivalence relations generated respectively
by reverse diamond
equivalences and by reverse elementary equivalences are identical.
\end{proposition}

\section{Reverse operations on strips}

If $S= \mu/\lambda$ is a strip on $\lambda$ then we say that $S$ is
a strip {\it inside} $\mu$. To translate between the forward and
reverse situations we add these elements to our dictionary:

\begin{align*}
&\text{strip $S=\mu/\lambda$ on $\lambda$}  \quad &
\longleftrightarrow \quad &\text{strip $S=\mu/\lambda$ inside $\mu$}
\\
&\mu \quad & \longleftrightarrow \quad &\la \\
&\la \quad & \longleftrightarrow \quad &\mu \\
&\Bot_c(\gamma) \quad &\longleftrightarrow \quad &\text{cell below
$\Bot_c(\gamma)$} \\
&k\quad&\longleftrightarrow \quad&k+1 \\
&\text{lower (upper) augmentable corner} \quad &\longleftrightarrow
\quad &\text{lower (upper) reverse} \\
&&& \quad\quad\text{augmentable corner} \\
&\text{addable corner} \quad &\longleftrightarrow \quad
&\text{removable corner} \\
&\text{maximal strip (cover)} \quad &\longleftrightarrow \quad
&\text{reverse maximal strip (cover)} \\
&\changerow(S) \quad &\longleftrightarrow \quad &-\changerow(S) \\
&\changerow(m) \quad &\longleftrightarrow \quad &-\changerow(m) \\
\end{align*}

\begin{definition} Let $\mu\in\Ksh$ be fixed. Let $\Str^\mu\subset\Ksh$ be the
induced subgraph of $\nu\in\Ksh$ such that $\mu/\nu$ is a strip
inside $\mu$. If $\tm$ is a move such that $\lambda= \tm \# \nu$ in
$\Str^\mu$ we shall say that $\tm$ is a reverse $\mu$-augmentation
move from the strip $\mu/\nu$ to the strip $\mu/\la$. A reverse
augmentation of a strip $\tilde S=\mu/\la$ is a strip reachable from
$\tS$ via a reverse $\mu$-augmentation path. A strip $\tS=\mu/\la$
is reverse maximal if it admits no reverse $\mu$-augmentation move.
\end{definition}
Diagrammatically, a reverse augmentation move is such that the
following diagram commutes for strips $S$ and $\tilde S$ inside
$\mu$.
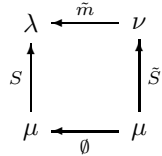
$$
\begin{diagram}
\node{\la} \node{\nu}  \arrow{w,t}{\tm} \\
\node{\mu} \arrow{n,l}{S}
\node{\mu} \arrow{n,r}{\tilde S} \arrow{w,b}{\emptyset}
\end{diagram}
$$

These definitions depend on a fixed $\mu\in\Ksh$, which shall
usually be suppressed in the notation. Later we shall consider
reverse augmentations of a given strip $\tS$, meaning reverse
$\mu$-augmentations where $\tilde S=\mu/\la$.


\begin{proposition} \label{L:augmentcolumnrev}
All reverse augmentation column moves of a
strip $\tS=\mu/\la$ have rank 1.
\end{proposition}

Let $\tS = \mu/\la$ be a strip and $a$ be a removable corner of
$\la$. We will call $a$ a
\begin{enumerate}
\item
{\it lower reverse augmentable corner} of $\tS$ if removing $a$ from
$\la$ adds a box to $\bdy\la$ in a modified column $c$ of $\tS$.
\item
{\it upper reverse augmentable corner} of $\tS$ if $a$ does not lie
below a box in $\tS$ and removing $a$ from $\la$ adds a box to $\bdy
\la$ in a modified row $r$ of $\tS$.
\end{enumerate}

A $\lambda$-removable string $s$ of row-type (resp. column-type) can
be reverse extended below (resp. above) if there is a
$\lambda$-removable corner contiguous and below (resp. above) the
lowest (resp. highest) cell of $s$.
\begin{definition} \label{D:revaugmentation}
A \defit{reverse completion row move} is one in which all strings start in the same row.
It is maximal if the first string cannot be reverse extended below.
A \defit{reverse quasi-completion column
  move} is a reverse column
 augmentation move from a strip $\tS$ that contains
no lower reverse augmentable corner. A \defit{reverse completion
column move} is a reverse quasi-completion move from a strip $\tS$
that contains no upper reverse augmentable corner below its unique
(by Proposition~\ref{L:augmentcolumnrev}) string. A reverse
completion column move or a reverse quasi-completion column move is
maximal if its string cannot be reverse extended above. A reverse
completion move is a reverse completion row/column move.
\end{definition}

\begin{prop} \label{P:maxstripuniquerev} Let $\tilde S=\mu/\la$ be a strip.
\begin{enumerate}
\item $\tilde S$ has a unique maximal reverse augmentation $\tS' \in\Str^\mu$.
\item There is one equivalence class of paths
in $\Str^\mu$ from $\tS$ to $\tS'$.
\item The unique equivalence class of paths in $\Str^\mu$ from $\tilde S$ to
  $\tilde S'$
has a representative consisting entirely of maximal reverse completion moves.
\end{enumerate}
\end{prop}

\subsection{Reverse maximal strips}

\begin{prop}\label{P:Rmaxstrip}
A strip $\tS$ is reverse maximal if and only if it has no reverse augmentable corners.
\end{prop}

\begin{proposition}\label{P:reversemaxcore}
Suppose $\mu$ is a $(k+1)$-core and $\tS = \mu/\la$ is a
reverse maximal
strip. Then $\la$ is a $(k+1)$-core.
\end{proposition}

\section{Pullback of strips and moves}
\label{sec:pullback}

Let $(\tS,\tm)$ be a final pair where $\tS=\eta/\nu$ is a strip and
$\tm=\eta/\mu$ is a nonempty row move.

We say that $(\tS,\tm)$ is compatible if it is reasonable, not
contiguous, (and normal if $\tm$ is a column move) and is either (1)
non-interfering, or (2) is interfering but is also
pullback-perfectible; all these notions are defined below.

For compatible pairs $(\tS,\tm)$ we define a $k$-shape $\la\in\Ksh$
(see Subsections \ref{SS:rowpullnointerfere} and
\ref{SS:colpullnointerfere} for case (1) and
\ref{SS:rowpullinterfere} and \ref{SS:colpullinterfere} for case
(2)). This given, we define the pullback
\begin{equation} \label{E:pullbackdef}
  \pull(\tS,\tm) = (S,m)=(\mu/\la,\nu/\la)
\end{equation}
which produces an initial pair $(S,m)$ where $S$ is a strip and $m$
is a (possibly empty) move. This is depicted by the following
diagram.
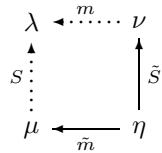
\begin{equation*}
\begin{diagram}
\node{\la}   \node{\nu} \arrow{w,t,..}{m}   \\
\node{\mu} \arrow{n,t,..}{S}  \node{\eta}  \arrow{w,b}{\tm}
\arrow{n,b}{\tS}
\end{diagram}
\end{equation*}
If $\tS$ is a reverse maximal strip then $(\tS,\tm)$ is compatible
by Corollaries \ref{C:maxpullrow} and \ref{C:maxpullcol}.

\subsection{Reasonableness}
We say that the pair $(\tS,\tm)$ is {\it reasonable} if for every
string $s \subset \tm$, either $s \cap \tS = \emptyset$ or $s
\subset \tS$.

\begin{proposition}\label{L:RSmreasonable}
Let $(\tS,\tm)$ be a final pair with $\tS$ is a reverse maximal
strip. Then $(\tS,\tm)$ is reasonable.
\end{proposition}

\subsection{Contiguity}
We say that $(\tS,\tm)$ is \textit{contiguous} if there is a box $b
\notin \bdy \mu \cup \bdy \nu$ which is present in $\bdy (\mu \cap
\nu)$. Call such a $b$ an {\it appearing box}.

\begin{proposition}
Let $(\tS,\tm)$ be a final pair with $\tS$ a reverse maximal strip.
Then $(\tS,\tm)$ is non-contiguous.
\end{proposition}

\subsection{Row-type pullback:  interference}
Suppose that $(\tS,\tm)$ is reasonable and non-contiguous with $\tm$
a row move. If $s \subset \tm$ is contained inside $\tS$, we say
that $\tS$ {\it matches} $s$ below if $c_{s,d}$ is a modified column
of $\tS$.  Otherwise we say that $\tS$ {\it continues below} $s$.
Define $\tm'$ and $\tm'$ by
\begin{align} \label{eqrevmp}
\tm' &= \bigcup \,\{\,\text{strings $s$}  \subset \tm \mid \text{$s$
and $\tS$ are
  not matched below}\} \\
  \label{E:tm-def}
  \tm^-&=\downarrow_{\tS}\tm'.
\end{align} \
We say that $(\tS,\tm)$ is \textit{non-interfering} if $\cs(\eta)
-\changerev_\cs(\tS)- \changerev_\cs(\tm')$ is a partition and is
\textit{interfering} otherwise.

\subsection{Row-type pullback: non-interfering case}
\label{SS:rowpullnointerfere} Assume that $(\tS,\tm)$ reasonable,
non-contiguous, and non-interfering with $\tm$ a row move. Then we
define $(\tS,\tm)$ to be compatible, with $\la=\tm^-\# \nu$ and
define $\pull(\tS,\tm)$ by \eqref{E:pullbackdef}.

\begin{proposition}\label{L:Retastrip}
Let $(\tS,\tm)$ be a reasonable, non-contiguous and non-interfering
final pair. Then $\mu/\la$ is a strip.
\end{proposition}

\subsection{Row-type pullback: interfering case}
\label{SS:rowpullinterfere} Assume that $(\tS,\tm)$ is reasonable,
non-contiguous, and interfering with $\tm$ a row move. Say that
$(\tS,\tm)$ is \textit{pullback-perfectible} if there is a set of
cells $\tm_\comp$ inside $(\tm^-)\#\nu$ so that if $\la =
((\tm^-)\#\nu) \setminus \tm_\comp$ then $\mu/\la$ is a strip and
$\nu/\la$ is a row move to $\nu$ with the same initial string as
$\tm^-$.

\begin{proposition}\label{L:Rmaximalinterfere}
Suppose $(\tS,\tm)$ is a reasonable, non-contiguous, interfering
final pair such that $\tm$ is a row move and $\tS$ is a reverse
maximal strip. Then $(\tS,\tm)$ is pullback-perfectible. Furthermore
the strings of $\tm_\comp$ lie on exactly the same rows as the
initial string of $\tm$.
\end{proposition}

\begin{cor} \label{C:maxpullrow}

Suppose $(\tS,\tm)$ is a final pair such that $\tS$ is a reverse
maximal strip and $\tm$ is a row move. Then $(\tS,\tm)$ is
compatible.
\end{cor}

\subsection{Column-type pullback: normality}
Suppose that $(\tS,\tm)$ is a reasonable final pair with $\tm$ a
column move. If $s \subset \tm$ is contained inside $\tS$, we say
that \textit{$\tS$ matches $s$ above} if $r_{s,u}$ is a modified row
of $\tS$. Otherwise we say that $\tS$ \textit{continues above} $s$.

Let $s \subset \tm$ be the final string of the move $\tm$. We say
that $(\tS,\tm)$ is \textit{normal} if it is reasonable, and, in the
case that $s$ is continued above,  (a) none of the modified rows of
$\tS$ contains boxes of $s$ and (b) the negatively modified row of
$s$ is not a modified row of $S$.

\begin{proposition}\label{L:RSmreasonablenormal}
Let $\tS$ be a reverse maximal strip and $\tm$ a column move.  Then
$(\tS,\tm)$ is normal and non-contiguous.
\end{proposition}

\subsection{Column-type pullback: interference}
Define $\tm'$ and $\tm^-$ by
\begin{align}
\tm' &= \bigcup \,\{\,\text{strings $s$}  \subset \tm \mid \text{$s$
and $\tS$ are
  not matched above}\} \\
  \tm^- &= \leftarrow_{\tS} \tm'.
\end{align}
If $\tm'\ne\emptyset$ we say that $(\tS,\tm)$ is
\defit{non-interfering} if $\rs(\eta) -\changerev_\rs(\tS) -
\changerev(\tm')$ is a partition and \defit{interfering} otherwise.
If $\tm'=\emptyset$ we say that $(\tS,\tm)$ is non-interfering if
$\rs(\mu)/\rs(\nu)$ is a horizontal strip and interfering otherwise
(observe that $\rs(\eta) -\changerev_\rs(\tS) -
\changerev(\tm')=\rs(\eta) -\changerev_\rs(\tS)=\rs(\nu)$ is always
a partition in that case). The latter case is referred to as
\defit{special interference}.

\subsection{Column-type pullback: non-interfering case}
\label{SS:colpullnointerfere}

Assume that $(\tS,\tm)$ is normal, non-contiguous and
non-interfering with $\tm$ a column move. In this case we declare
$(\tS,\tm)$ to be compatible. $\tm^-$ is a move to $\nu$ and we
define $\la= \tm^- \# \nu$. The pullback is defined by
\eqref{E:pullbackdef}.

\begin{proposition}\label{L:Retastripcol}
Suppose $(\tS,\tm)$ is normal, non-contiguous and non-interfering.
Then $\mu/\la$ is a strip.
\end{proposition}

\subsection{Column-type pullback: interfering case}
\label{SS:colpullinterfere} Assume that $(\tS,\tm)$ is normal,
non-contiguous and interfering with $\tm$ a column move. Say that
$(\tS,\tm)$ is \textit{pullback-perfectible} if there is a set of
cells $\tm_\comp$ inside $(\tm^-)\#\nu$ so that if $\la =
((\tm^-)\#\nu) \setminus \tm_\comp$ then $\mu/\la$ is a strip and
$\nu/\la$ is a row move to $\nu$ with the same initial string as
$\tm^-$. In the case that $(\tS,\tm)$ is pullback-perfectible, we
declare that $(\tS,\tm)$ is compatible and use the above $\la$ to
define the pullback via \eqref{E:pullbackdef}.

\begin{proposition} Suppose
$(\tS,\tm)$ is reasonable, normal, non-contiguous and interfering
with $\tm$ a column move and $\tS$ a reverse maximal strip. Then
$(\tS,\tm)$ is pullback-perfectible. Furthermore $\tm_\comp$
consists of a single string that lies on the same columns as the
initial string of $\tm$.
\end{proposition}

\begin{cor} \label{C:maxpullcol}
Suppose $(\tS,\tm)$ is a final pair with $\tS$ a reverse
maximal strip and $\tm$ a column move. Then
$(\tS,\tm)$ is compatible.
\end{cor}

\section{Pullbacks sequences are all equivalent}
\label{sec:pullbacksequence}
Given a strip $\tilde S=\eta/\nu$ and a path ${\mathbf {q}}$ in the
$k$-shape poset from $\mu$ to $\eta$,
one can do a sequence of pullbacks and reverse augmentations to obtain
a reverse maximal strip $S=\mu/\la$ and a
path ${\mathbf {p}}$ in the $k$-shape poset
from $\la$ to $\nu$:
\begin{equation}
\begin{diagram}
\node{\la}
 \node{\nu}  \arrow{w,t}{{\mathbf {p}}}
\\
\node{\mu} \arrow{n,l}{S}
\node{\eta} \arrow{n,r}{\tilde S} \arrow{w,t}{{\mathbf {q}}}
\end{diagram}
\end{equation}
Such a process, which we will call a {\it pullback sequence},
can always be done since we have seen that
a reverse maximal strip is compatible with any move.
As in the forward case,
it does not matter which pullout sequence is used since they give
rise to
equivalent paths (and therefore to a unique reverse maximal strip $S$).
\begin{proposition} \label{P:commuaugmenrev}
 Let $\tS=\eta/\nu$ be strip and  ${\mathbf {q}}$ a
path in the $k$-shape poset from $\mu$ to $\eta$, and suppose that a
given pullback sequence gives rise to a reverse maximal strip $S=\mu/\la$
and a path ${\mathbf {p}}$ in the $k$-shape poset
from $\la$ to $\nu$. Then any other given pullback sequence gives
rise to the reverse maximal strip $S=\mu/\la$
and a path ${\mathbf {\tilde p}}$ equivalent to ${\mathbf {p}}$.
\end{proposition}

\section{Pullbacks of equivalent paths are equivalent}
\label{sec:pullbackequivalence}
 The next proposition tells us
that the pullbacks of equivalent paths produce equivalent paths.
\begin{proposition} \label{P:commucuberev}
 Let $\tS=\eta/\nu$ be a strip and and let ${\mathbf {q}}$ and
${\mathbf {q'}}$ be equivalent paths
in the $k$-shape poset from $\mu$ to $\eta$.  If the
 pullback sequence associated to $\tS$ and  ${\mathbf {q}}$
gives rise to a reverse maximal strip $S=\mu/\la$
and a path ${\mathbf {p}}$ in the $k$-shape poset
from $\la$ to $\nu$,  then the
 pullback sequence associated to $\tS$ and  ${\mathbf {q}'}$
gives rise to the same reverse maximal strip $S=\mu/\la$
and a path ${\mathbf {p}'}$ equivalent to ${\mathbf {p}}$.
\end{proposition}
Propositions~\ref{P:commucuberev} and
Proposition~\ref{P:commuaugmenrev} provide an algorithm, which we
will call the {\it pullback algorithm}, that, given a strip
$\tS=\eta/\nu$ and a class of paths $[{\mathbf {q}}]$ in the
$k$-shape poset from $\mu$ to $\eta$, gives rise to a reverse
maximal strip $S=\mu/\la$ and a unique class of paths $[{\mathbf
  {p}}]$
in the $k$-shape poset from $\la$ to $\nu$:
\begin{equation}
\begin{diagram}
\node{\la}
 \node{\nu}  \arrow{w,t}{[{\mathbf {p}}]} \\
\node{\mu} \arrow{n,l}{S}
 \node{\eta}   \arrow{w,t}{[{\mathbf {q}}]} \arrow{n,r}{\tilde S}
\end{diagram}
\end{equation}

\section{Pullbacks are inverse to pushouts} \label{S:prooftheo}
\begin{prop}\label{P:pullpush} $~$
\begin{enumerate}
\item
Let $(S,m)$ be a compatible initial pair with $\push(S,m) = (\tS,\tm)$.
If $\tm$ is not empty
then $(\tS,\tm)$ is a
compatible final pair such that $\pull(\tS,\tm) = (S,m)$.
If $\tm$ is empty then $m$ is a reverse augmentation move on the strip
$\tS$.
\item If $\tm$ is an augmentation move on the strip $S$
such that $\tm * S= \tilde S$,
then
$(\tS,\tm)$ is a
compatible final pair such that $\pull(\tS,\tm) = (S,\emptyset)$.
\item
Let $(\tS,\tm)$ be a compatible final pair with $\pull(\tS,\tm) =
(S,m)$.
If $m$ is not empty then $(S,m)$ is a
compatible initial pair such that $\push(S,m) = (\tS,\tm)$.
If $m$ is empty then $\tm$ is an augmentation move on the strip $S$.
\item If $m$ is a reverse  augmentation move on the strip $\tilde S$
such that $m \# \tS= S$,
then  $(S,m)$ is a
compatible initial pair such that $\push(S,m) = (\tS,\emptyset)$.
\end{enumerate}
\end{prop}
\begin{proof}
The non-empty cases follow from the alternative descriptions
of pushouts and its analogue for pullbacks via expected row and column shape.
The empty cases are immediate.
\end{proof}

We now prove Theorem~\ref{T:pushpullbij}. As already mentioned after
the statement of Theorem~\ref{T:pushpullbij}, it suffices to prove
the case where $S$ and $T$ are single strips.  That is, we need to
show that given a reverse maximal strip $S=\mu/\lambda$ and a class
of paths $[{\mathbf {p}}]$ from $\lambda$ to $\nu$, the pushout
algorithm gives rise to a maximal strip $\tilde S=\eta/\nu$ and the
class of paths $[{\mathbf
  {q}}]$ from $\mu$ to a $\eta$:
\begin{equation}
\begin{diagram}
\node{\la} \arrow{s,l}{S}
 \arrow{e,t}{[{\mathbf {p}}]} \node{\nu} \arrow{s,r}{\tilde S}\\
\node{\mu} \arrow{e,t}{[{\mathbf {q}}]} \node{\eta}
\end{diagram}
\end{equation}
if and only if given the maximal
strip $\tilde S=\eta/\nu$ and the class of paths $[{\mathbf
  {q}}]$
from $\mu$ to $\eta$,
the {pullback algorithm}
gives rise to the reverse maximal strip $S=\mu/\lambda$ and the class of paths
$[{\mathbf {p}}]$ from $\lambda$ to $\nu$:
\begin{equation}
\begin{diagram}
\node{\la}
 \node{\nu}   \arrow{w,t}{[{\mathbf {p}}]} \\
\node{\mu} \arrow{n,l}{S}  \node{\eta} \arrow{w,t}{[{\mathbf {q}}]}
\arrow{n,r}{\tilde S}
\end{diagram}
\end{equation}

Suppose we are given a reverse maximal
strip $S=\mu/\lambda$ and a class of paths
$[{\mathbf {p}}]$, and suppose that the pushout algorithm leads to
the maximal strip $\tilde S=\eta/\nu$ and the class of paths $[{\mathbf
  {q}}]$.  As we have seen, this implies that any pushout sequence
leads to
the maximal strip $\tilde S=\eta/\nu$ and the class of paths $[{\mathbf
  {q}}]$.  By Proposition~\ref{P:pullpush}, every pushout sequence can be
reverted to give a pullback sequence from the maximal strip $\tilde
S=\eta/\nu$ and the class of paths $[{\mathbf{q}}]$.  This ensures that there
is at least one pullback sequence from
the maximal strip $\tilde S=\eta/\nu$ and the class of paths $[{\mathbf
  {q}}]$ that leads to the reverse maximal
strip $S=\mu/\lambda$ and the class of paths $[{\mathbf {p}}]$. As
we have seen, this implies that the pullback algorithm always leads
to the reverse maximal strip $S=\mu/\lambda$ and the class of paths
$[{\mathbf {p}}]$.  Therefore the pullback of a pushout gives back
the initial pair. We can prove that the pushout of a pullback gives
back the final pair in a similar way. Note that for the bijection to
work, we need $S$ to be reverse maximal and $\tS$ to be maximal.
This is because the pushout algorithm produces a maximal strip while
the pullback algorithm yields a reverse maximal strip.

\begin{appendix}

\section{Tables of branching polynomials} \label{Appendix}

\noindent
We list here all the branching polynomials $\tilde b_{\mu \lambda}^{(k)}(t)$
for partitions of degree up to 6.

\bigskip

\noindent \underline{Degree 2:}

\begin{center}
\begin{tabular}{|c||c|c|}
\hline
$b_{\mu \lambda}^{(2)}$  & $1^2$ & 2 \\
\hline\hline
$1^2$& 1 & $t$  \\ \hline
\end{tabular}

\end{center}

\noindent \underline{Degree 3:}

\begin{center}
\begin{tabular}{|c||c|c|}
\hline
$b_{\mu \lambda}^{(2)}$  & $1^3$ & 21 \\
\hline\hline
$1^3$& 1 & $t^2$  \\ \hline
\end{tabular}
\qquad \begin{tabular}{|c||c|c|c|}
\hline
$b_{\mu \lambda}^{(3)}$ & $1^3$ & 21 & 3\\
\hline\hline
$1^3$& 1 & $t$ & \\ \hline
21&   & 1& $t$\\ \hline
\end{tabular}
\end{center}

\noindent \underline{Degree 4:}

\begin{center}
\begin{tabular}{|c||c|c|c|}
\hline
$b_{\mu \lambda}^{(2)}$  & $1^4$ & $21^2$ & $2^2$ \\
\hline\hline
$1^4$& 1 & $t^2+t^3$ & $t^4$  \\ \hline
\end{tabular}
\qquad \qquad
 \begin{tabular}{|c ||c|c|c|c| }
\hline
$b_{\mu \lambda}^{(3)}$ & $1^4$ &$21^2$ & $2^2$ & 31  \\
\hline
\hline
$1^4$ & 1 &  & $t^2$ &  \\ \hline
$21^2$ &  & 1 &  &   \\ \hline
$2^2$ &  &  & $1$ & $t$ \\ \hline
\end{tabular}

\bigskip

\begin{tabular}
{|c||c|c|c|c|c| }
\hline
{$b_{\mu \lambda}^{(4)}$} &
{$1^4$} & {$21^2$} &{$2^2$} & {$31$} & {$4$}
 \\
\hline
\hline
{$1^4$} & {1} & $t$ & & &
\\
\hline
{$21^2$} & & {1} & &$t$ &
\\
\hline
{$2^2$} & & & {1} & &
\\
\hline
{$31$} & & & & {1} & {$t$}
\\
\hline
\end{tabular}
\end{center}

\noindent \underline{Degree 5:}

\begin{center}
\begin{tabular}{|c||c|c|c|}
\hline
$b_{\mu \lambda}^{(2)}$  & $1^5$ & $21^3$ & $2^21$   \\
\hline\hline
$1^5$& 1 & $t^3+t^4$ & $t^6$  \\ \hline
\end{tabular}
\qquad \qquad
\begin{tabular}
{|c ||c|c|c|c|c| }
\hline
$b_{\mu \lambda}^{(3)}$ &
$1^5$ & $21^3$ & $2^21$ & $31^2$ & 32
\\
\hline
\hline
{$1^5$ }
& 1 & $t^2$ & $t^3$  &   &
\\
\hline
{$21^3$}
& & 1 & $t$ & $t^2$  &
\\
\hline
{$2^21$}
&  &  & $1$  & $t$  & $t^2$
\\
\hline
\end{tabular}
\qquad

\bigskip

\begin{tabular}
{|c||c|c|c|c|c|c|c| }
\hline
{$b_{\mu \lambda}^{(4)}$} &
{$1^5$} & {$21^3$} &{$2^21$} & {$31^2$} & {$32$} & {41}
 \\
\hline
\hline
{$1^5$} & {1} &  & $t^2$ & & &
\\
\hline
{$21^3$} & & {1} & &   & &
\\
\hline
{$2^21$} & & & {1} & & $t$ &
\\
\hline
{$31^2$} & & & &  {1} & &
\\
\hline
{$32$} & & & & & {1} & $t$
\\
\hline
\end{tabular}
\qquad
\begin{tabular}
{|c||c|c|c|c|c|c|c| }
\hline
{$b_{\mu \lambda}^{(5)}$} &
{$1^5$} & {$21^3$} &{$2^21$} & {$31^2$} & {$32$} & {41} & 5
 \\
\hline
\hline
{$1^5$} & {1} & $t$ & & & & &
\\
\hline
{$21^3$} & & {1} &  &$t$ & & &
\\
\hline
{$2^21$} & & & 1 &  & & &
\\
\hline
{$31^2$} & & & & 1 &  & $t$ &
\\
\hline
{$32$} & & & & & 1 &  &
\\
\hline
{$41$} & & & & & & 1 & $t$
\\
\hline
\end{tabular}
\end{center}

\bigskip

\noindent \underline{Degree 6:}

\begin{center}
\begin{tabular}{|c||c|c|c|c|}
\hline
$b_{\mu \lambda}^{(2)}$  & $1^6$ & $21^4$ & $2^21^2$ & $2^3$   \\
\hline\hline
$1^6$& 1 & $t^3+t^4+t^5$ & $t^6+t^7+t^8$ & $t^9$ \\ \hline
\end{tabular}

\bigskip

\begin{tabular}
{|c||c|c|c|c|c|c|c| }
\hline
{$b_{\mu \lambda}^{(3)}$} &
{$1^6$} & {$21^4$} &{$2^21^2$} & {$2^3$}  & {$31^3$} & 321 & $3^2$
 \\
\hline
\hline
{$1^6$} & {1} & $t^2$ & $t^4$ & & & &
\\
\hline
{$21^4$} & & {1} &  &$t^2$ & $t^2$ &  &
\\
\hline
{$2^21^2$} & & & 1 &  &  $t$ & $t^2$ &
\\
\hline
{$2^3$} & & & & 1 &  & $t^2$  & $t^3$
\\
\hline
\end{tabular}

\bigskip

\begin{tabular}
{|c||c|c|c|c|c|c|c|c|c|}
\hline
{$b_{\mu \lambda}^{(4)}$} &
{$1^6$} & {$21^4$} &{$2^21^2$} & {$2^3$} & {$31^3$} & {$321$} & $3^2$ & $41^2$ &$42$
 \\
\hline
\hline
{$1^6$} & {1} &  &  & $t^3$ &  & & & &
\\
\hline
{$21^4$} & & {1} &$t$  & & & &  & &
\\
\hline
{$2^21^2$} & & & {1} &  & &   &$t^2$ & &
\\
\hline
{$2^3$} & & & & {1} &  & $t$& & &
\\
\hline
{$31^3$} & & & & & 1 &  &  & &
\\
\hline
{$321$} & & & & &  &1 &  & $t$&
\\
\hline
{$3^3$} & & & & & &  & 1 &  &$t$
\\
\hline
\end{tabular}

\bigskip

\begin{tabular}
{|c||c|c|c|c|c|c|c|c|c|c| }
\hline
{$b_{\mu \lambda}^{(5)}$} &
{$1^6$} & {$21^4$} &{$2^21^2$} & {$2^3$} & {$31^3$} & {$321$} & $3^2$ & $41^2$ &$42$ &$51$
 \\
\hline
\hline
{$1^6$} & {1} &  & $t^2$ & & & & & & &
\\
\hline
{$21^4$} & & {1} &  & & & &  & & &
\\
\hline
{$2^21^2$} & & & {1} &  & &$t$  & & & &
\\
\hline
{$2^3$} & & & & {1} &  & & & & &
\\
\hline
{$31^3$} & & & & & 1 &  &  &  & &
\\
\hline
{$321$} & & & & &  &1 &  & & $t$ &
\\
\hline
{$3^3$} & & & & & &  & 1&  & &
\\
\hline
{$41^2$} & & & & & &  &  & 1& &
\\
\hline
{$42$} & & & & & &  & &  &$1$ &$t$
\\
\hline
\end{tabular}

\begin{tabular}
{|c||c|c|c|c|c|c|c|c|c|c|c| }
\hline
{$b_{\mu \lambda}^{(6)}$} &
{$1^6$} & {$21^4$} &{$2^21^2$} & {$2^3$} & {$31^3$} & {$321$} & $3^2$ & $41^2$ &$42$ &$51$ & 6
 \\
\hline
\hline
{$1^6$} & {1} & $t$ &  & & & & & & & &
\\
\hline
{$21^4$} & & {1} &  &  & $t$  & &  & & & &
\\
\hline
{$2^21^2$} & & & {1} &  & &  & & & & &
\\
\hline
{$2^3$} & & & & {1} &  & & & & & &
\\
\hline
{$31^3$} & & & & & 1 &  &  & $t$ & & &
\\
\hline
{$321$} & & & & &  &1 &  & &  & &
\\
\hline
{$3^3$} & & & & & &  & 1&  & & &
\\
\hline
{$41^2$} & & & & & &  &  & 1& & $t$ &
\\
\hline
{$42$} & & & & & &  & &  &$1$ &  &
\\
\hline
{$51$} & & & & & &  & &  &   & 1 & $t$
\\
\hline
\end{tabular}

\end{center}

\end{appendix}

\end{document}